\documentclass{amsart}

\usepackage{amsmath}
\usepackage{amsfonts}
\usepackage{amssymb}
\usepackage{amsthm}
\usepackage{comment}
\usepackage{epsfig}
\usepackage{psfrag}
\usepackage{mathrsfs}
\usepackage{amscd}
\usepackage[all]{xy}
\usepackage{rotating}
\usepackage{lscape}
\usepackage{amsbsy}
\usepackage{verbatim}
\usepackage{moreverb}
\usepackage{color}
\usepackage{bbm}
\usepackage{eucal}

\newtheorem{theorem}{Theorem}[subsection]
\newtheorem{prop}[theorem]{Proposition}
\newtheorem{lemma}[theorem]{Lemma}
\newtheorem{cor}[theorem]{Corollary}

\newtheorem{observation}[theorem]{Observation}

\theoremstyle{definition}
\newtheorem{definition}[theorem]{Definition}
\newtheorem{lem/def}[theorem]{Lemma/Definition}

\newtheorem{notation}[theorem]{Notation}
\newtheorem{construction}[theorem]{Construction}
\newtheorem*{convention}{Convention}

\newtheorem{remark}[theorem]{Remark}
\newtheorem{example}[theorem]{Example}

\newtheorem{terminology}[theorem]{Terminology}

\newtheorem{warning}[theorem]{Warning}

\theoremstyle{remark}

\definecolor{orange}{rgb}{1,0.5,0}
\definecolor{light-gray}{gray}{0.75}
\definecolor{brown}{cmyk}{0, 0.8, 1, 0.6}
\definecolor{plum}{rgb}{.5,0,1}

\DeclareMathOperator{\Aut}{\sf Aut}
\DeclareMathOperator*{\colim}{{\sf colim}}
\DeclareMathOperator*{\limit}{{\sf lim}}

\DeclareMathOperator{\Hom}{\sf Hom}

\DeclareMathOperator{\Fun}{\sf Fun}

\DeclareMathOperator{\Map}{\sf Map}

\DeclareMathOperator{\Ch}{\sf Ch}
\DeclareMathOperator{\stab}{\sf Stab}

\DeclareMathOperator{\Ker}{\sf Ker}
\DeclareMathOperator{\coker}{\sf cKer}

\DeclareMathOperator{\lkan}\LKan

\DeclareMathOperator{\uno}{\mathbbm{1}}

\DeclareMathOperator{\m}{\sf Mod}
\DeclareMathOperator{\perf}{\sf Perf}

\DeclareMathOperator{\Alg}{\mathsf{Alg}}

\DeclareMathOperator{\op}{\mathsf{op}}

\DeclareMathOperator{\Emb}{\mathsf{Emb}}
\DeclareMathOperator{\Diff}{\mathsf{Diff}}
\DeclareMathOperator{\Cat}{\mathsf{Cat}}
\DeclareMathOperator{\conf}{\mathsf{Conf}}

\DeclareMathOperator{\Spaces}{\mathsf{Spaces}}
\DeclareMathOperator{\spaces}{\mathsf{Spaces}}

\DeclareMathOperator{\ZMfld}{\cZ\cM\mathsf{fld}}
\DeclareMathOperator{\ZEmb}{\mathsf{ZEmb}}
\DeclareMathOperator{\Mfld}{\cM\mathsf{fld}}
\DeclareMathOperator{\mfld}{\cM\mathsf{fld}}
\DeclareMathOperator{\ZDisk}{\cZ\cD\mathsf{isk}}

\DeclareMathOperator{\fr}{\sf fr}

\DeclareMathOperator{\Sym}{\mathsf{Sym}}
\DeclareMathOperator{\LKan}{\mathsf{LKan}}
\DeclareMathOperator{\id}{\mathsf{id}}

\DeclareMathOperator{\Sing}{\mathsf{Sing}}

\DeclareMathOperator{\BO}{{\mathsf BO}}
\DeclareMathOperator{\Top}{{\mathsf Top}}

\DeclareMathOperator{\Lie}{\sf Lie}

\DeclareMathOperator{\bsc}{\cB{\sf sc}}

\DeclareMathOperator{\Bsc}{\cB{\sf sc}}
\DeclareMathOperator{\Snglr}{\cS{\sf nglr}}

\def\ot{\otimes}
\def\smash{\wedge}

\DeclareMathOperator{\oo}{\infty}

\DeclareMathOperator{\hh}{\sf HH}

\DeclareMathOperator{\free}{\sf Free}
\DeclareMathOperator{\Free}{\sf Free}

\DeclareMathOperator{\FPres}{\sf FPres}
\DeclareMathOperator{\Perf}{\sf Perf}

\DeclareMathOperator{\ran}{\sf Ran}

\DeclareMathOperator{\disk}{\cD{\sf isk}}
\DeclareMathOperator{\Disk}{\cD{\sf isk}}
\DeclareMathOperator{\bBar}{\sf Bar}
\DeclareMathOperator{\cBar}{\sf cBar}
\DeclareMathOperator{\cAlg}{\sf cAlg}

\DeclareMathOperator{\Psh}{\mathsf{PShv}}
\DeclareMathOperator{\PShv}{\mathsf{PShv}}

\DeclareMathOperator{\Mod}{\sf Mod}
\DeclareMathOperator{\Fin}{\sf Fin}

\DeclareMathOperator{\Ran}{\sf Ran}
\DeclareMathOperator{\Triv}{\sf Triv}
\DeclareMathOperator{\Artin}{\sf Artin}
\DeclareMathOperator{\Spf}{\sf Spf}
\DeclareMathOperator{\MC}{\sf MC}
\DeclareMathOperator{\fpres}{\sf FPres}
\DeclareMathOperator{\FModuli}{\sf Moduli}

\newcommand{\ra}{\rightarrow}
\newcommand{\la}{\leftarrow}
\newcommand{\xra}{\xrightarrow}
\newcommand{\xla}{\xleftarrow}

\newcommand{\ov}{\overline}
\newcommand{\un}{\underline}

\newcommand{\tl}{\triangleleft}

\def\cB{\mathcal B}\def\cC{\mathcal C}\def\cD{\mathcal D}
\def\cE{\mathcal E}\def\cF{\mathcal F}
\def\cI{\mathcal I}\def\cJ{\mathcal J}\def\cK{\mathcal K}
\def\cM{\mathcal M}\def\cO{\mathcal O}\def\cP{\mathcal P}
\def\cS{\mathcal S}
\def\cV{\mathcal V}\def\cW{\mathcal W}\def\cX{\mathcal X}
\def\cY{\mathcal Y}\def\cZ{\mathcal Z}

\def\CC{\mathbb C}\def\DD{\mathbb D}
\def\FF{\mathbb F}

\def\NN{\mathbb N}
\def\RR{\mathbb R}
\def\VV{\mathbb V}
\def\ZZ{\mathbb Z}

\def\sB{\mathsf B}\def\sC{\mathsf C}
\def\sH{\mathsf H}
\def\sL{\mathsf L}
\def\sN{\mathsf N}\def\sO{\mathsf O}
\def\sR{\mathsf R}\def\sT{\mathsf T}
\def\sU{\mathsf U}

\def\bDelta{\mathbf\Delta}
\def\bdelta{\mathbf\Delta}

\begin{document}

\title{Poincar\'e/Koszul duality}
\author{David Ayala \& John Francis}
\date{}
\address{Department of Mathematics\\Montana State University\\Bozeman, MT 59717}
\email{david.ayala@montana.edu}
\address{Department of Mathematics\\Northwestern University\\Evanston, IL 60208}
\email{jnkf@northwestern.edu}
\thanks{DA was partially supported by ERC adv.grant no.228082, and by the National Science Foundation under awards 0902639 and 1507704. JF was supported by the National Science Foundation under awards 1207758 and 1508040.}

\begin{abstract} We prove a duality for factorization homology which generalizes both usual Poincar\'e duality for manifolds and Koszul duality for $\cE_n$-algebras. The duality has application to the Hochschild homology of associative algebras and enveloping algebras of Lie algebras. We interpret our result at the level of topological quantum field theory.
\end{abstract}



\maketitle

\tableofcontents

\section*{Introduction}

This paper arises from the following question: what is Poincar\'e duality in factorization homology?

\smallskip

Before describing our solution, we give some background for this question. After Lurie in \cite{HA}, a factorization homology, or topological chiral homology, theory is a homology-type theory for $n$-manifolds; these theories are natural with respect to embeddings of manifolds and satisfy a multiplicitive generalization of the Eilenberg--Steenrod axioms for ordinary homology---see \cite{fact}. This relatively new theory has two particularly notable antecedents: the labeled or amalgamated configuration space models of mapping spaces of Salvatore \cite{salvatore}, Segal \cite{segallocal}, and Kallel \cite{kallel}, after \cite{bodig}, \cite{mcduff}, \cite{may}, and \cite{segal}; the algebraic approaches to conformal field theory of Beilinson \& Drinfeld in \cite{bd}, via factorization algebras, and of Segal in \cite{segalconformal}.

\smallskip

In the last few years, significant work has taken place in this subject in addition to the basic investigation of the foundations of factorization homology in \cite{HA}, \cite{fact}, and \cite{aft2}. In algebraic geometry, Gaitsgory \& Lurie use factorization algebras to prove Weil conjecture's on Tamagawa numbers for algebraic groups in the function field case---see \cite{tamagawa}. In mathematical physics, Costello has developed in \cite{kevin} a renormalization machine for quantizing field theories; by work of Costello \& Gwilliam in \cite{kevinowen}, this machine outputs a factorization homology theory as a model for the observables in a perturbative quantum field theory. Their work gives an array of interesting examples of factorization homology theories and manifold invariants connected with gauge theory, quantum groups, and knot and 3-manifold invariants.

\smallskip

The question of Poincar\'e duality finds motivation from all the preceding works. We focus on two veins.  First, factorization homology theories are characterized by a monoidal generalization of the Eilenberg--Steenrod axioms for usual homology, so that factorization homology specializes to ordinary homology in the case the target symmetric monoidal category is that of chain complexes with direct sum. As such, it makes sense to ask that different values of factorization homology theories, valued in a general symmetric monoidal category, likewise enjoy a relationship specializing to that of Poincar\'e duality. From this perspective, the present work fills in the bottom middle in the following table of analogies.

\smallskip

\begin{center}
    \begin{tabular}{|p{4cm}|  p{5cm} |p{4cm}| }
    \hline
        {\bf Ordinary/Generalized} & {\bf Factorization} & {\bf Physics} \\ \hline
    topological space $M$ & $n$-manifold $M$ & spacetime $M$ \\ \hline
    abelian group $A$ & $n$-disk stack $X$ & quantum field theory $Z$   \\ \hline
    additivity $\amalg \rightsquigarrow \oplus$ & multiplicativity $\amalg \rightsquigarrow \ot$ & locality $\amalg \rightsquigarrow \ot$ \\ \hline
        homology $\sC_\ast(M,A)$ & factorization homology $\displaystyle\int_MX$ & observables ${Z}(M)$ \\ \hline
       linearity $\sC_\ast(M, A)\simeq P_1 \sC_\ast(M,A)$ & Goodwillie calculus $P_\bullet \displaystyle\int_M X$ & perturbative observables $Z_{\sf pert}(M)$   \\ \hline
        Poincar\'e duality $\sC_\ast(M,A) \simeq \sC^\ast(M, A[n])$& Poincar\'e/Koszul duality $\displaystyle\int_M X^{\wedge}_x\simeq \Bigl(\int_M T_xX[-n]\Bigr)^\vee$ &   \\ \hline

    \end{tabular}
\end{center}

\smallskip

Second, in the perspective espoused in the works \cite{bd} and \cite{kevinowen}, factorization homology theories are algebraic models for physical field theories. 
In extended topological quantum field theory, as appears in the cobordism hypothesis after Baez \& Dolan \cite{baezdolan} and Lurie \cite{cobordism}, there is likewise a fundamental duality: the duality of higher $n$-categories which appear as the values of the field theories on points. These dual theories take equal values on $n$-manifolds, and linearly dual values on $(n-1)$-manifolds; their values on manifolds of lower dimension is expressible, in a less familiar way, as a higher categorical form of duality.
\smallskip

While both of these perspectives augur for a notion of duality in factorization homology, they likewise both indicate that an essential ingredient is missing. In the first case, usual Poincar\'e duality is a relationship between usual homology and compactly supported cohomology -- this suggests that a notion of compactly supported factorization cohomology is necessary. From the perspective of the cobordism hypothesis in the 1-dimensional case, the factorization homology theories $\int A$ is closely related to extended topological field theories $Z$ whose value on a point is $Z(\ast) = \perf_A$, and the value on the circle is $Z(S^1) \simeq \int_{S^1} A\simeq \hh_\ast(A)$, the Hochschild chains of $A$. One would expect the dual field theory $Z^\vee$ to take the value $Z^\vee(S^1) = \hh_\ast(A)^\vee$, the $\Bbbk$-linear dual of the Hochschild chains of $A$. However, there is in general no algebra $B$ for which $\hh_\ast(B)$ is equivalent to $\hh_\ast(A)^\vee$. That is, the category $Z^\vee(\ast)$ is not given by perfect modules for some other algebra. Under restrictive conditions, however, this sometimes happens: namely, the algebra $\DD A$, Koszul dual to $A$, is a candidate. One can construct a natural map, which is an instance of our Poincar\'e/Koszul duality map, \[\hh_\ast(\DD A) \longrightarrow \hh_\ast(A)^\vee\] which is sometimes an equivalence; for instance, this is a formal exercise in the case $A$ is $\Sym(V)$, for $V$ a finite complex in strictly positive homological degrees. In general, we shall see, $\hh_\ast(A)^\vee$ is a type of completion of $\hh_\ast(\DD A)$.

\smallskip

Our two sources of motivation thus offer two pointers on where to look for Poincar\'e duality in factorization homology. The first says that the factorization homology with coefficients in $A$ should be equivalent to some other construction, not factorization homology per se, but some cohomological variant. Our TQFT motivation suggests that the choice of coefficients for this factorizable generalization of cohomology should be related to the Koszul dual of $A$. Assembling these hints, one might conclude that such a Poincar\'e duality should relate the factorization homology with coefficients in $A$ to a not necessarily perturbative form of factorization homology with coefficients related to the Koszul dual of $A$, a generalization that one can contemplate after the originating works of \cite{gk} and \cite{priddy}.

\smallskip

There is, however, already a not necessarily perturbative form of factorization homology, defined in \cite{fact}. Namely, for an $n$-manifold $M$ and a scheme $X$ whose structure sheaf $\cO_X$ is enhanced to form a sheaf of $n$-disk algebras, then one can define the factorization homology of $M$ with coefficients in $X$ as \[\int_MX = \Gamma\Bigl(X, \int_M\cO_X\Bigr)\] the global sections over $X$ of the sheaf obtained by computing the factorization homology of the structure sheaf. Intuitively, one can also think of this object as a quantization of functions on maps from $M$ to the underlying space $X_0$ -- the observables in a sigma model formed by quantizing a symplectic structure on the underlying space $X_0$. We adopt the point of view, after Costello \& Gwilliam, that factorization homology with coefficients in stacks over $n$-disk algebras offers a suitable generic model for the observables in a sigma model. We regard this as a generalization of the perturbative theory in the following sense: perturbation theory involves the calculation of the theory by formal expansion at a fixed solution; for a sigma model, this involves the formal neighborhood of the mapping space $\Map(M,X)$ at a point $M\ra X$, such as a constant map. If $X = \Spf A$ is itself a formal neighborhood of a point, given by the formal spectrum of an augmented pro-nilpotent algebra $\epsilon : A\ra \Bbbk$, then both the mapping space $\Map(M,\Spf A)$ and the factorization homology 
\[
\int_M A~ \simeq~ \Gamma\bigl(\Spf A, \int_M\cO\bigr)
\]
can be completely understood by perturbation theory about the $\Bbbk$-point $\epsilon\in \Spf A$, which we understand as an instance of Goodwillie calculus. From this point of view, we think of the factorization homology $\int_M X$ as being perturbative if the classical target $X_0$ is a subspace of the quantization $X$, and the factorization homology $\int_M X$ can be obtained from $\int_M X_0\sim \cO\bigl(\Map(M,X_0)\bigr)$ by Goodwillie calculus; this should be the case when the quantization $X$ is a formal neighborhood of the classical phase space $X_0$, in particular, where $\hbar$ is taken to be a formal parameter.

\smallskip

In this paper, we deal completely with the case in which $X_0$ is a point and $X$ is a general formal space, which one can think of as perturbation theory around a single constant map. For a general augmented algebra $\epsilon:A\ra \Bbbk$, the factorization homology $\int_MA$ cannot be understood by perturbation theory at the augmentation. So one might ask for some more general non-affine geometric target $X$ associated to $A$, where the greater global geometry of $X$ similarly reflects the noncovergence properties of $\int_MA$ at  $\epsilon$. There is, finally, just such a formal space associated to an $n$-disk algebra, lifting the functor of Koszul duality.

\smallskip

The original formulations of Koszul duality, after Priddy \cite{priddy} and Moore \cite{moore}, involve only algebra and coalgebra -- for a review of Koszul duality in algebraic topology, we suggest \cite{sinha}. However, from Quillen's Lie algebraic model for rational homotopy types \cite{quillen}, one can think of a Lie algebra as giving rise to a homotopy type via its Maurer--Cartan space. That is, to a differential graded Lie algebra $\frak g$, one can consider the set of solutions to the Maurer--Cartan equation: elements $x\in \frak g_{-1}$ for which $dx+\frac 1 2[x,x] =0$. Applying this procedure to a resolution of the Lie algebra, such as the simplicial object in Lie algebras $\frak g \ot \Omega^\ast(\Delta^\bullet)$ obtained by tensoring with Sullivan's polynomial de Rham forms on simplices \cite{sullivan}, gives a simplicial set whose associated homotopy type is the Maurer--Cartan space. More generally, we take the Maurer--Cartan functor of a Lie algebra $\frak g$ to be a space-valued functor
\[
\xymatrix{
\Artin\ar[rr]^{\MC_{\frak g}} &&\Spaces\\}
\]
which assigns to a local Artin $\Bbbk$-algebra $R$ the space of maps of Lie algebras
\[\MC_{\frak g}(R):=\Map_{\Lie}(\DD R, \frak g)\]
where $\DD R$ is the Lie algebra which is Koszul dual to $R$; equivalently $\DD R\simeq \sT_R[-1]$ is a shift of the tangent complex of $R$ given by the maximal ideal $R\ra \Bbbk$ shifted by $-1$.

\smallskip

This construction gains important because of the following principle, which very often holds. If $X$ is an object of interest for which there exists a tangent complex $\sT_X$, then in favorable situations there exists a Lie algbera structure on $\sT_X$, and the deformations of the structure of $X$ is given by solutions to the Maurer--Cartan equation in $\sT_X$. This pattern emerged in the 1950s, particularly in complex geometry (see \cite{kodairaspencer}, \cite{nr}, \cite{fn}, and \cite{kuranishi}), where solutions to the Maurer--Cartan equation of the Kodaira--Spencer Lie algebra classify deformations of a complex structure. This construction was then studied more generally in other works; in particular, see \cite{gm1}, \cite{gm2}, \cite{hinichschechtman}, \cite{hinichdeligne}, \cite{getzler}, and \cite{dag10}.

\smallskip

The above pattern applies to $n$-disk algebra as well, after \cite{dag10}. Given an augmented $n$-disk algebra $A$, there is a formal moduli functor
\[
\xymatrix{
\Artin_n\ar[rr]^{\MC_A} &&\Spaces\\}
\]
which assigns to a local Artin $n$-disk $\Bbbk$-algebra $R$ the space of maps of augmented $n$-disk algebras
\[\MC_A(R) :=\Map(\DD^nR, A)\]
from the Koszul dual of $R$ as an $n$-disk algebra, to $A$. This is a lift of usual Koszul duality through moduli, in that the ring of global functions of $\MC_A$ is exactly the Koszul dual of $A$.

\smallskip

We now have the ingredients necessary to state our main theorem; see Theorem \ref{main}.
\smallskip
\begin{theorem}[Poincar\'e/Koszul duality]\label{1main} Let ${M}$ be a compact smooth $n$-dimensional cobordism with boundary partitioned as $\partial {M} \cong \partial_{\sL} \sqcup \partial_{\sR}$. For $A$ an augmented $n$-disk algebra over a field $\Bbbk$ with $\MC_A$ the associated formal moduli functor of $n$-disk algebras, there is a natural equivalence 
\[
\Bigl(\int_{{M}\smallsetminus \partial_{\sR}}A\Bigl)^\vee~ \simeq~ \int_{{M}\smallsetminus \partial_{\sL}}\MC_A
\] 
between the $\Bbbk$-linear dual of the factorization homology of ${M}\smallsetminus \partial_{\sR}$ with coefficients in $A$ and the factorization homology of ${M}\smallsetminus \partial_{\sL}$ with coefficients in the moduli functor $\MC_A$.
\end{theorem}

Note that if the boundary of the compact $n$-manifold $M$ is empty, then the statement above reduces to the simpler expression of an equivalence 
\[
\Bigl(\int_{ M} A\Bigr)^\vee~\simeq~ \int_{ M}\MC_A~.
\]

\begin{remark}
If we write $X$ for the formal moduli functor $\MC_A$, then $A$ itself is equivalent to $\sT_x X[-n]$, the tangent space at the distinguished $\Bbbk$-point $x\in X$ shifted by $-n$. The assertion of Theorem \ref{1main} then becomes 
\[
\int_{M\smallsetminus \partial_{\sf L}} X~ \simeq ~\Bigl(\int_{M\smallsetminus \partial_{\sf R}} \sT_x X[-n]\Bigr)^\vee~.
\]
\end{remark}

\smallskip

This theorem coheres to our dual motivations. First, the result specializes to the dual of usual Poincar\'e duality by setting $A$ to be an algebra with respect to direct sum; in this case the lefthand side becomes usual homology with coefficients in $A$, the formal moduli problem is representable, and the righthand side becomes usual cohomology with coefficients in an $n$-fold shift of $A$. Second, from the cobordism formulation one can see this result involves a duality for the extended topological field theories defined by the factorization homology with coefficients in $A$ and $\MC_A$. Theorem \ref{1main} is a more general than the duality assured by the cobordism hypothesis, however. Since limits need not commute with tensor products of infinite dimensional vector spaces, the natural map
\[
\int_{M} \MC_A \ot \int_{N} \MC_A \longrightarrow \int_{M\amalg N}\MC_A
\]
need not be an equivalence. So factorization homology with coefficients in $\MC_A$ does not define a symmetric monoidal functor from the bordism category unless an additional requirement is made on $A$, to ensure the factorization homologies $\int_M A$ and $\int_M \MC_A$ are finite dimensional.

\smallskip

Given connectivity or coconnectivity hypotheses on the algebra $A$, one can replace the moduli problem $\MC_A$ with its algebra of global sections $\DD^nA$, the Koszul dual of $A$. We have the following, combining Theorem \ref{compare-towers}, Theorem \ref{convergence}, and Proposition \ref{switch}.

\begin{theorem}\label{summ}
Let ${M}$ be a compact smooth $n$-dimensional cobordism with boundary partitioned as $\partial M \cong \partial_{\sL} \sqcup \partial_{\sR}$. Let $A$ be a finitely presented augmented $n$-disk algebra such that either:
\begin{itemize}
\item $A$ has a connected augmentation ideal;
\item $A$ is an algebra over a field $\Bbbk$ and the augmentation ideal of $A$ is $(-n)$-coconnective.
\end{itemize}
There is a natural equivalence \[\Bigl(\int_{{M}\smallsetminus \partial_{\sR}}A\Bigl)^\vee~ \simeq~ \int_{{M}\smallsetminus \partial_{\sL}}\DD^nA~.\]

\end{theorem}

\smallskip

This result is interesting even for dimension $n=1$. The factorization homology of the circle is equivalent to Hochschild homology, so in this case we obtain a linear duality between the Hochschild homology of an associative algebra and either the Hochschild homology of the noncommutative moduli problem $\MC_A$ or of the Koszul dual $A\simeq \Hom_{A}(\Bbbk,\Bbbk)$. See Corollary \ref{assoc}. This specialization is particularly comprehensible in the case where the algebra $A=\sU \frak g$ is the enveloping algebra of a Lie algebra over a field of characteristic zero. In this case, we obtain a relation between the enveloping algebra of a Lie algebra $\frak g$ and the Lie algebra cohomology of $\frak g$. Then we have the following, which generalizes a result of Feigin \& Tsygan in \cite{feigintsygan}.

\smallskip

\begin{theorem} For $\frak g$ a Lie algebra over a field of characteristic zero which is finite dimensional and concentrated in either homological degrees less than $-1$ or in degrees greater than $0$, then there is an equivalence
\[
\hh_\ast(\sU\frak g)^\vee ~\simeq~ \hh_\ast(\sC^\ast \frak g)
\] 
between the dual of the Hochschild homology of the enveloping algebra and the Hochschild homology of Lie algebra cochains.
\end{theorem}

We make two remarks on generalizations, or the lack thereof.

\begin{remark}
All of our results are valid, with identical proofs, if smooth $n$-manifolds are replaced with $G$-structured smooth $n$-manifolds, where $G$ is a Lie group with a continuous homomorphism $G\ra {\Diff}(\RR^n)$. However, for visual simplicity we omit this notational clutter from most of the current work.
\end{remark}
\begin{remark}
We had originally imagined this work in the setting of {\it topological}, rather than smooth, manifolds, but serious technical obstructions dissuaded us. In topology, one runs into the difficulty that  the configuration spaces $\conf_i(M)$ in a compact topological manifold $M$ might not admit a compactification to a topological manifold with corners.  
Indeed, without some regularity conditions, which are assured by smoothness, our present methods do not allow us to identify the $\oo$-category $\Disk^{\sf top}_{n/M}$ with the exit-path $\oo$-category of the Ran space of $M$ (or even to show that the simplicial space of exit-paths in $\Ran(M)$ indeed satisfy the Segal condition and so form an $\oo$-category).
This is the key technical point which allows us to analyze the layers of the Goodwillie calculus towers in terms of configuration spaces. In algebra, one runs into difficulties stemming from $\Top(n)$, the topological group of homeomorphisms of $\RR^n$, which is far less understood than the group of diffeomorphisms of $\RR^n$. 
Specifically, the homology $\sH_\ast\Top(n)$ is not known, at least to us, to be finite rank---unlike $\sH_\ast\sO(n)$---and so the notions of coherent and perfect $\Top(n)$-modules differ. As a result, Koszul duality of topological $n$-disk algebras does not enjoy the same good duality properties as that of smooth $n$-disk algebra. We do not think that a statement at the generality of Theorem \ref{1main} is true for topological $n$-disk algebras and topological manifolds.
\end{remark}

\smallskip

We now overview the contents of this paper, section by section.
\smallskip

In \textbf{Section 1}, we review the category $\ZMfld_{n}$ of zero-pointed $n$-manifolds and the factorization homology of zero-pointed manifolds from \cite{ZP}. A zero-pointed manifold consists of a pointed topological space $M_\ast$, which is a smooth $n$-manifold $M$ with an extra point $\ast$ and a conically smooth extension of the topology of $M$ to $M_\ast$; the essential example is a space $ M/\partial  M$, the quotient of an $n$-manifold by its boundary. This theory naturally incorporates functoriality for both embeddings and Pontryagin--Thom collapse maps of embeddings, a feature we employ in order to present a unified treatment of duality in homology/cohomology and algebra/coalgebra. More precisely, the zero-pointed theory provides additional functorialities for factorization homology with coefficients in an augmented $n$-disk algebra algebra which, in particular, endows the factorization homology \[\int_{(\RR^n)^+} A\] with the structure of an $n$-disk coalgebra, where $(\RR^n)^+$ is the 1-point compactification of $\RR^n$. Consequently, we arrive at a geometric presentation of an $n$-disk coalgebra structure on the $n$-fold iterated bar construction of an augmented $n$-disk algebra.  We apply this to construct the Poincar\'e/Koszul duality map, which goes from factorization homology with coefficients in an $n$-disk algebra to factorization cohomology with coefficients in the Koszul dual $n$-disk coalgebra. We lastly recall a version of twisted Poincar\'e duality, which asserts that our duality map is an equivalence in the case of stable $\oo$-category with direct sum.

\smallskip

In \textbf{Section 2}, we introduce two (co)filtrations of factorization homology and cohomology. One comes from Goodwillie's calculus of homotopy functors. A second comes from a cardinality filtration $\disk_n^{\leq k}$ of $\disk_n$. The cardinality filtration is a common generalization of the Goodwillie--Weiss calculus filtration from topology and the Hodge filtration of Hochschild homology (see Remark \ref{hodge}). We prove that the Poincar\'e/Koszul duality map exchanges the Goodwillie and the cardinality cofiltrations. That is, in this instance, we prove that Goodwillie calculus and Goodwillie--Weiss calculus are Koszul dual to one another. As a consequence, we obtain spectral sequences for factorization homology whose $E^1$-terms are identified as homologies of configuration spaces; for the circle, one of these spectral sequences generalizes the B\"okstedt spectral sequence. Finally, in the case of our main theorem, we conclude that the Poincar\'e/Koszul duality map is an equivalence when the algebra $A$ is connected.

\smallskip

In \textbf{Section 3}, we introduce factorization homology with coefficients in a formal moduli problem. We prove that the Poincar\'e/Koszul duality map is an equivalence in the case of a  $(-n)$-coconnective $n$-disk algebra over a field. We also prove an instance of Koszul duality proper, that there is equivalence between Artin $n$-disk algebras and finitely presented $(-n)$-coconnective $n$-disk algebras. Using this, we show our main theorem, that the moduli-theoretic Poincar\'e/Koszul duality map is an equivalence for any augmented $n$-disk algebra. We conclude by specializing these results to the case of associative algebras and Lie algebras in \textbf{Section 4}.

\smallskip

\begin{remark}
In this work, we use Joyal's {\it quasi-category} model  of $\oo$-category theory \cite{joyal}. 
Boardman \& Vogt first introduced these simplicial sets in \cite{bv}, as weak Kan complexes, and their and Joyal's theory has been developed in great depth by Lurie in \cite{HTT} and \cite{HA}, our primary references. See the first chapter of \cite{HTT} for an introduction. We use this model, rather than model categories or simplicial categories, because of the great technical advantages for constructions involving categories of functors, which are ubiquitous in this work. 
More specifically, we work inside of the quasi-category associated to this model category of Joyal's.  In particular, each map between quasi-categories is understood to be an iso- and inner-fibration; and (co)limits among quasi-categories are equivalent to homotopy (co)limits with respect to Joyal's model structure.

We will also make use of topological categories, such as the topological category $\mfld_n$ of smooth $n$-manifolds and smooth embeddings. By a functor $\mfld_n\ra \cV$ from a topological category such as $\mfld_n$ to an $\oo$-category $\cV$ we will always mean a functor $\sN\Sing\mfld_n\ra \cV$ from the simplicial nerve of the simplicial category $\Sing \mfld_n$ obtained by applying the singular functor $\Sing$ to the hom spaces of the topological category. 

The reader uncomfortable with this language can substitute the words ``topological category" for ``$\oo$-category" wherever they occur in this paper to obtain the correct sense of the results, but they should then bear in mind the proviso that technical difficulties may then abound in making the statements literally true. The reader only concerned with algebra in chain complexes, rather than spectra, can likewise substitute ``pre-triangulated differential graded category" for ``stable $\oo$-category" wherever those words appear, with the same proviso.
\end{remark}

\subsection*{Terminology}[$\ot$-conditions]\label{stab-pres}
Throughout this document, we will use the letter $\cV$ for a symmetric monoidal $\infty$-category, and $\uno$ for its symmetric monoidal unit. We will not distinguish in notation between it and its underlying $\infty$-category.  
\begin{itemize}

\item 
We say $\cV$ is \emph{$\ot$-presentable} if its underlying $\infty$-category is presentable and its symmetric monoidal structure $\ot$ distributes over small colimits separately in each variable.  
We say $\cV$ is \emph{$\ot$-stable-presentable} if it is $\ot$-presentable and its underlying $\infty$-category is stable.

\item We say $\cV$ is \emph{$\ot$-cocomplete} if its underlying $\infty$-category admits small colimits and its symmetric monoidal structure $\ot$ distributes over small colimits separately in each variable.

\item We say $\cV$ is \emph{$\ot$-sifted cocomplete} if its underlying $\infty$-category admits sifted colimits and its symmetric monoidal structure $\ot$ distributes over sifted colimits separately in each variable.  

\end{itemize}

\medskip

\subsection*{Acknowledgements} Our collaboration began at a workshop in Glanon in 2011. We are grateful to people of Glanon for their warm hospitality in hosting this annual workshop, and to Gr\'egory Ginot for inviting us to participate in it. We have learned an enormous amount from Jacob Lurie; we use, in particular, his opuses \cite{HTT} and \cite{HA} throughout. We thank Greg Arone for several very helpful conversations on Goodwillie calculus. JF thanks Kevin Costello for offering many insights in many conversations over the years. We thank the referees for their informed and detailed readings, which have significantly improved this paper.

\section{Review of reduced factorization homology}\label{sec:review}
We recall some notions among zero-pointed manifolds and factorization (co)homology thereof, as established in~\cite{ZP}.

\subsection{Zero-pointed manifolds}\label{sec:zero}

We recall the enlargement of the symmetric monoidal topological category $\Mfld_n$ of smooth $n$-manifolds and open embeddings among them (with compact-open $C^\infty$ topologies), to \emph{zero-pointed} $n$-manifolds.

\begin{definition}[Zero-pointed manifolds]\label{def:zero}
An object of the symmetric monoidal topological category of zero-pointed manifolds 
$
\ZMfld
$
is a locally compact Hausdorff based topological space $M_\ast$ together with the structure of a smooth manifold on the complement of the base point $M:=M_\ast \smallsetminus \ast$ subject to the following finitary condition:
\begin{itemize}
\item The 1-point compactification $(M_\ast)^+$ admits a conically smooth structure with respect to which the inclusion $M\hookrightarrow (M_\ast)^+$ is a conically smooth open embedding.  
\end{itemize}
The topological space of morphisms $\ZEmb(M_\ast,M'_\ast)$ consists of based continuous maps $f\colon M_\ast \to M'_\ast$ for which the restriction $f_|\colon f^{-1}M'\to M'$ is a smooth open embedding, endowed with the (compactly generated weak Hausdorff replacement of the) compact-open topology.   
Composition is given by composing based maps.
The symmetric monoidal structure is wedge sum.  
There is the full sub-symmetric monoidal topological category
\[
\ZMfld_n~\subset~\ZMfld
\]
consisting of those zero-pointed manifolds $M_\ast$ for which $M$ has dimension exactly $n$. For a smooth $n$-manifold $M$, $M_+$ is the zero-pointed manifold given by $M$ with a disjoint zero-point; $M^+$ is the zero-pointed manifold defined by the 1-point compactification of $M$.
The full sub-symmetric monoidal categories of $\ZMfld_n$
\[
\Disk_{n,+} \subset \ZDisk_n \supset \Disk_n^+
\]
consist of finite wedge sums of $\RR^n_+$, of $\RR^n_+$ and $(\RR^n)^+$, and of $(\RR^n)^+$, respectively. Likewise, $\Mfld_{n,+}$ and $\Mfld_n^+$ are the full sub-symmetric monoidal categories of $\ZMfld_n$ consisting of zero-pointed manifolds of the form $M_+$ and $M^+$, respectively.
\end{definition}

\begin{remark}
In particular, a conically smooth zero-pointed manifold $M_\ast$ determines a stratified space structure on the pointed topological space $M_\ast$. 
\end{remark}

\begin{example}[Cobordisms]\label{cobordisms}
Let $\ov{M}$ be a cobordism, i.e., a compact manifold with partitioned boundary $\partial\ov{M} = \partial_{\sL} \amalg \partial_{\sR}$.  
The based topological space \[M_\ast := \ast \underset{\partial_{\sL}}\coprod (\ov{M}\smallsetminus \partial_{\sR})\] is a zero-pointed manifold.  

\end{example}

\begin{remark}
The \emph{unzipping} construction of~\cite{aft1} reveals that a conically smooth zero-pointed manifold is the data of a compact smooth manifold $\ov{M}$ with boundary which is partitioned $\partial \ov{M} = \partial_{\sL}\sqcup \partial_{\sR}$, as in Example~\ref{cobordisms}.  

\end{remark}

\begin{remark}
Consider the notation of Definition~\ref{def:zero}.
A zero-pointed embedding from $\ast \underset{\partial_{\sL}} \amalg (\ov{M}\smallsetminus \partial_{\sR})$ to $\ast \underset{\partial'_{\sL}}\amalg (\ov{M}'\smallsetminus \partial'_{\sR})$ is vastly different from an embedding from $\ov{M}$ to $\ov{M}'$ that respect the partitioned boundaries in any sense.  

\end{remark}

We now catalogue some facts about $\ZMfld_n$ that are proven in~\cite{ZP}.
Note first the continuous functor $\Mfld_n \to \Mfld_{n,+}$, given by $M\mapsto M_+$, which is symmetric monoidal.

\begin{theorem}[\cite{ZP}]\label{zero-facts}
Let $\cV$ be a symmetric monoidal $\infty$-category.
\begin{itemize}

\item
There is a contravariant involution $\neg\colon \ZMfld_n \cong \ZMfld_n^{\op}\colon \neg$ which sends a zero-pointed manifold \[M_\ast\longmapsto(M_\ast)^+\smallsetminus \ast\] to the 1-point compactification of $M$ minus the original zero-point.

\item
There is a canonical equivalence of $\infty$-categories
\[
\Fun^{\ot}(\Mfld_{n,+},\cV) \xra{~\simeq~} \Fun^{\ot, {\sf aug}}(\Mfld_{n},\cV)
\]
to between symmetric monoidal functors from $\Mfld_{n,+}$ and symmetric monoidal functors from $\Mfld_n$ which are augmented over the constant functor $\Mfld_{n}\ra \cV$ whose value is the unit $\uno$ of $\cV$. (That is, an augmented functor $F$ has the data of a projection $F(M)\ra \uno$ and a section $\uno\ra F(M)$ compatibly for all $M$.)

\end{itemize}

\end{theorem}

\subsection{Configuration zero-pointed manifolds}
We define two natural pointed spaces of configurations associated to a zero-pointed manifold.

\begin{notation}\label{def.conf}
Let $i$ be a finite cardinality.
Let $M_\ast$ be a pointed locally compact Hausdorff topological space.
We denote the pointed topological space 
\[
\conf_i(M_\ast)  ~:=~
\Bigl\{ \bigl[ \{1,\dots,i\} \xra{c} M_\ast \bigr] \mid c_{|M} \text{ is injective} \Bigr\}
~\subset~
(M_\ast)^{\smash i}  ~,
\]
which is endowed with the subspace topology.
We denote the pointed topological space
\[
\conf^{\neg}_i(M_\ast)  ~:=~
\ast \underset{ \Bigl\{ \bigl[ \{1,\dots,i\} \xra{c} M_\ast \bigr] \mid c_{|M} \text{ is not injective} \Bigr\}}
\coprod
(M_\ast)^{\smash i}  ~,
\]
which is endowed with the quotient topology.  
The evident $\Sigma_i$-action on $\conf_i(M_\ast)$ and on $\conf^{\neg}_i(M_\ast)$ results in the pointed Hausdorff topological spaces of $\Sigma_i$-coinvariants:
\[
\conf_i(M_\ast)_{\Sigma_i}
\qquad \text{ and }\qquad
\conf_i^\neg(M_\ast)_{\Sigma_i}~.  
\]

\end{notation}

\begin{remark}
Consider an injection $c\colon \{1,\dots,i\}\hookrightarrow M$.
Intuitively, $c$ is near the zero-point in $\conf_i(M_\ast)$ if one of its elements is near the zero-point of $M_\ast$. In comparison, $c$ is near the zero-point in $\conf_i^\neg(M_\ast)$ if either: one of its elements is near the zero-point of $M_\ast$; or at least two of its elements are near one another.  

\end{remark}

The pointed topological space $\conf_i(M_\ast)$ is not generally a zero-pointed manifold.
Indeed, for $i\geq 2$, it is locally compact only if the zero-point $\ast \in M_\ast$ is isolated.
In this case that the zero-point is isolated, $M_\ast = M_+$, there is a canonical identification
\[
\conf_i(M_+) ~ = ~  \conf_i(M)_+
\]
with $\conf_i(M)$ adjoined a disjoint zero-point.
Likewise, $\conf^{\neg}_i(M_\ast)$ defines a zero-pointed manifold only if $M_\ast$ is compact.
In this case that $M_\ast$ is compact, $M_\ast = M^+$, there is a canonical identification
\[
\conf^{\neg}_i(M^+)~ =~ \conf_i(M)^+
\]
with the 1-point compactification of $\conf_i(M)$.

The next result corrects this local compactness issue, up to homotopy equivalence, at the expense of considering zero-pointed manifolds \emph{with corners}.
The definition of a zero-pointed manifold with corners is the same as that of a zero-pointed manifold, just allowing $M$ to be a smooth manifold with corners in place of simply a smooth manifold.  (See~\S1.5.2 of \cite{ZP}.)
This correction also relates the two constructions $\conf_i(M_\ast)$ and $\conf_i^{\neg}(M_\ast)$ via negation of their representative zero-pointed manifolds with corners.  
\begin{lemma}\label{conf.finitary}
Let $i$ be a finite cardinality.
For each zero-pointed manifold $M_\ast$ there are $\Sigma_i$-zero-pointed manifolds with corners $C_i(M_\ast)$ and $C^{\neg}_i(M_\ast)$ with the following properties.  

\begin{enumerate}

\item
There are $\Sigma_i$-equivariant based homotopy equivalences,
\[
\conf_i(M_\ast)~\simeq~C_i(M_\ast)
\qquad\text{ and }\qquad
\conf^{\neg}_i(M_\ast)~\simeq~ C^{\neg}_i(M_\ast)~.
\]

\item
There are $\Sigma_i$-equivariant isomorphisms between zero-pointed manifolds:
\[
C_i(M_\ast)^{\neg}
~\cong~
C^{\neg}_i(M_\ast^{\neg})~.
\]

\end{enumerate}

\end{lemma}

\begin{proof}
Let $M_\ast$ be a zero-pointed manifold.  
Fix a conically smooth structure on the 1-point compactification $(M_\ast)^+$, which exists by Definition~\ref{def:zero}.  
Taking blow-ups along $\{\ast\}\subset (M_\ast)^{+}$ and $\{+\}\subset (M_\ast)^+$ determines the compact smooth manifold with boundary $\ov{M}$.
By construction, its boundary is a disjoint union $\partial \ov{M} = \partial_- \ov{M} \sqcup \partial_+ \ov{M}$ where these two unions of connected components are the preimages of the canonical projection from the blow-up:
\[
\xymatrix{
\partial_- \ov{M}  \ar[rr]  \ar[d]
&&
\ov{M}  \ar[d]
&&
\partial_+\ov{M}  \ar[d]  \ar[ll]
\\
\{\ast\} \ar[rr]
&&
(M_\ast)^+
&&
\{+\} .   \ar[ll]  
}
\]
Consider the map 
\begin{equation}\label{e60}
\ov{M} \longrightarrow \{-<0>+\}
\end{equation}
to the poset with 3 elements, defined by declaring the preimage of $\{\pm \}$ to be precisely $\partial_\pm \ov{M} \subset \ov{M}$.
This map is continuous with respect to the poset topology on the codomain and factors through the understood stratifying poset of the manifold with boundary $\ov{M}$.

Consider the $i$-fold product $\ov{M}^{i}$, which is a smooth manifold with corners.
We refine $\ov{M}^{i}$ as follows.
Denote the set $\un{i}:=\{1,\dots,i\}$.  
Consider the set 
\[
\cP_i~:=~\Bigl\{  \bigl(\sim ~,~ \un{i} \xra{v} \{-<0>+\}  \bigr)  \Bigr\}
\]
in which $\sim$ is an equivalence relation on the set $\un{i}$ and $v$ is a map that factors through this equivalence relation. Endow $\cP_i$ with the poset structure in which $(\sim,v)\leq (\sim',v')$ if $\sim$ both is coarser than $\sim'$ and $v(j)\leq v'(j)$ for each $1\leq j\leq i$.  
Note that the evident projection 
\begin{equation}\label{e61}
\cP_i \longrightarrow \{-<0>+\}^{i} 
\end{equation}
is a surjective functor between posets.  
Consider the map
\[
\ov{M}^{i}
\longrightarrow
\cP_i
\]
whose value on 
$
\bigl(\un{i}\xra{c} \ov{M}\bigr) 
$
is the pair $(\sim,v)$ where $\sim$ is the coarsest equivalence relation on $\un{i}$ through which $c$ factors, while $v\colon \un{i} \xra{c} \ov{M} \xra{(\ref{e60})} \{-<0>+\}$ is the composition.  
By inspection, this map is continuous with respect to the poset topology on the codomain, and it factors the map~(\ref{e60}) via the map~(\ref{e61}).
As we proceed, we regard $\ov{M}^{i}$ as a stratified space with this stratifying poset $\cP_i$.

Denote the closed union of strata
\[
\Delta~:=~\Bigl\{ \un{i}\xra{c}\ov{M} \mid c \text{ is not injective} \Bigr\}~\subset~\ov{M}^{i}~.
\]
As in~\cite{fm}, we consider the blow-up of this manifold with corners $\ov{M}^{i}$ along this proper constructible subspace, which we denote as 
\[
\ov{C}_i(\ov{M})~:=~{\sf Bl}_{\Delta}(\ov{M}^{i})~.
\]
This is another smooth manifold with corners.  
Denote the preimages by the canonical projection from the blow-up:
\[
\xymatrix{
\partial_0  \ar[r]  \ar[d]
&
\ov{C}_i(\ov{M})  \ar[d]
&
&
\partial_{\pm}  \ar[r]  \ar[d]
&
\ov{C}_i(\ov{M})  \ar[d]
\\
\Delta  \ar[r]
&
\ov{M}^{i}
&
\text{ and }
&
\bigl\{ c\mid c^{-1}(\partial_\pm \ov{M}) \neq \emptyset \bigr\}  \ar[r]
&
\ov{M}^{i}   .
}
\]
In this diagram, all of the terms are compact stratified spaces, all of the downward arrows in these diagrams are weakly constructible bundles, and all of the horizontal arrows are proper constructible embeddings.

We pause to explain the developments that follow.
Optimistically, the pair $\partial_- \subset \ov{C}_i(\ov{M}) \supset \partial_+\cup \partial_0$ determines a zero-pointed stratified space, along the lines of Example~\ref{cobordisms}.  
However, the intersection of faces $\partial_- \cap (\partial_+ \cup \partial_0)$ is generally not empty.  
We remedy the situation by blowing-up along this intersection; because this is a transverse intersection (as faces of a manifold with corners), such a blow-up will, in effect, make the preimages in this blow-up of these subspaces disjoint.
Because each stratum in the link of this intersection is contractible, this blow-up will not effect homotopy types.

We now implement this approach.
Consider the two blow-ups
\begin{equation}\label{e62}
\ov{C}_i(\ov{M})^{\pm}~:=~ 
{\sf Bl}_{\partial_{\mp}\cap (\partial_{\pm} \cup \partial_0)} \bigl( \ov{C}_i(\ov{M})\bigr)~.
\end{equation}
This is again a smooth manifold with corners.
Using that the blow-up site of~(\ref{e62}) is a closed union of faces in a smooth manifold with corners, there are unique lifts in the diagram among stratified spaces
\begin{equation}\label{e63}
\xymatrix{
&&
\ov{C}_i(\ov{M})^{\pm}  \ar[d]
&&
\\
\partial_{\mp}  \ar[rr]  \ar@{-->}[urr]
&&
\ov{C}_i(\ov{M})
&&
\partial_{\pm} \cup \partial_0  .  \ar[ll]  \ar@{-->}[ull]   
}
\end{equation}
By construction, these lifts are proper constructible embeddings, and their images are disjoint.
As we proceed, we regard $\partial_\mp$ and $\partial_{\pm}\cup \partial_0$ as subspaces of $\ov{C}_i(\ov{M})^{\pm}$ via these lifts.

Now, as in Example~\ref{cobordisms}, consider the stratified spaces
\[
C_i(M_\ast)~:=~
\ast \underset{\partial_{\mp}} \coprod \bigl( \ov{C}_i(\ov{M})^{\pm} \smallsetminus 
(\partial_{\pm}\cup \partial_0)   \bigr)
\qquad \text{ and }\qquad
C^{\neg}_i(M_\ast)~:=~
\ast \underset{\partial_{\mp}\cup \partial_0} \coprod \bigl( \ov{C}_i(\ov{M})^{\mp} \smallsetminus 
\partial_{\pm}  \bigr)~,
\]
which are obtained by collapsing certain faces of manifolds with corners.
By construction, each of $C_i(M_\ast)$ and $C^{\neg}_i(M_\ast)$ defines a zero-pointed manifold with corners, equipped with an action by the symmetric group $\Sigma_i$. 
Furthermore, by construction, there are canonical $\Sigma_i$-identifications
\begin{eqnarray}
\nonumber
C_i(M_\ast)^{\neg}
&
~:=~
&
(C_i(M_\ast))^+ \smallsetminus \ast
\\
\nonumber
&
~:=~
&
\Bigl(  \ast \underset{\partial_{\mp}} \coprod \bigl( \ov{C}_i(\ov{M})^{\pm} \smallsetminus 
(\partial_{\pm}\cup \partial_0)   \bigr)  \Bigr)^+ \smallsetminus \ast
\\
\nonumber
&
~\cong~
&
\Bigl(  \ast \sqcup \{+\} \underset{\partial_{\mp} \sqcup (\partial_{\pm}\cup \partial_0)} \coprod \bigl( \ov{C}_i(\ov{M})^{\pm}   \bigr)  \Bigr) \smallsetminus \ast
\\
\nonumber
&
~\cong~
&
\{+\} \underset{\partial_{\pm}\cup \partial_0 } \coprod \bigl( \ov{C}_i(\ov{M})^{\pm} \smallsetminus \partial_{\mp}   \bigr) 
\\
\nonumber
&
~=:~
&
C_i^{\neg}(M_\ast^{\neg})~.
\end{eqnarray}

The blow-up site defining $\ov{C}_i(\ov{M})^{\pm}$ in~(\ref{e62}) is along a face of a manifold with corners, whose links are therefore contractible.  
It follows that the vertical arrow in the diagram~(\ref{e63}), which is a weakly constructible bundle, has contractible fibers.
Therefore, the vertical arrow in~(\ref{e63}) is a homotopy equivalence.
Furthermore, restrictions of that vertical arrow define a commutative diagram among $\Sigma_i$-topological spaces
\[
\xymatrix{
\partial_\mp  \ar[rr]  \ar[d]
&&
\ov{C}_i(\ov{M})^+ \smallsetminus (\partial_{\pm} \cup \partial_0)  \ar[d]
\\
\Bigl\{  c\colon \un{i}\hookrightarrow \ov{M}\smallsetminus \partial_{\pm}\ov{M}  \mid  c^{-1}(\partial_{\mp}\ov{M}) \neq \emptyset \Bigr\}
\ar[rr]
&&
\conf_i(\ov{M}\smallsetminus \partial_{\pm} \ov{M})  
}
\]
in which the vertical arrows are homotopy equivalences.
Being inclusions of closed unions of faces into smooth manifolds with corners, the existence of collaring neighborhoods about such ensures that both of these horizontal maps are cofibrations between $\Sigma_i$-topological spaces.  
Therefore the map between pushouts
\[
C_i(M_\ast)~:=~
\ast \underset{\partial_{\mp}} \coprod \bigl( \ov{C}_i(\ov{M})^{\pm} \smallsetminus 
(\partial_{\pm}\cup \partial_0)   \bigr)
\xra{~{}~\simeq~{}~}
\ast \underset{ \{  c\colon \un{i}\hookrightarrow \ov{M}\smallsetminus \partial_{\pm}\ov{M}  \mid  c^{-1}(\partial_{\mp}\ov{M}) \neq \emptyset \}  
}
\coprod
\conf_i(\ov{M}\smallsetminus \partial_{\pm} \ov{M})  
\]
is a pointed $\Sigma_i$-homotopy equivalence.
Finally, by construction of $\ov{M}$, the evident map from the pushout
\[
\ast \underset{ \{  c\colon \un{i}\hookrightarrow \ov{M}\smallsetminus \partial_{\pm}\ov{M}  \mid  c^{-1}(\partial_{\mp}\ov{M}) \neq \emptyset \}  }
\coprod 
\conf_i(\ov{M}\smallsetminus \partial_{\pm} \ov{M})  
\xra{~{}~\simeq~{}~}
\conf_i(M_\ast)
\]
is a homeomorphism.  
There results a pointed $\Sigma_i$-homotopy equivalence:
\[
C_i(M_\ast)
\longrightarrow
\conf_i(M_\ast)
\]
A pointed $\Sigma_i$-homotopy equivalence
\[
C^{\neg}_i(M_\ast)
\longrightarrow
\conf^{\neg}_i(M_\ast)
\]
is constructed similarly.
This completes this proof.  
\end{proof}

We record the following homological coconnectivity bound for configurations of points in a zero-pointed manifold.  
\begin{prop}\label{conf-coconnectivity}\label{conf-stuff}
Let $M_\ast$ be a zero-pointed $n$-manifold with at most $\ell$ components, and let $i$ be a finite cardinality.  
Then the reduced singular homology \[\ov{\sH}_q\bigr(\conf_i(M_\ast); A\bigr) = 0\] vanishes for any local system of abelian groups $A$ for $q>n\ell + (n-1)(i-\ell)$ and $i>>0$.  

\end{prop}

\begin{proof}
First, consider the case in which $M_\ast$ is open, i.e., has no compact component. In this case we show by induction on $i$ the vanishing of $\ov{\sH}_q\bigr(\conf_i(M_\ast); A\bigr)$ given $q>(n-1)i$. Consider the base case of $i=1$, in which case the finitary assumption in the definition of zero-pointed manifolds implies that $M_\ast$ has the homotopy type of a CW complex with no $q$-cells for $q>n-1$. From this it is immediate that $\ov{\sH}_q(M_\ast, A)$ vanishes for $q>n-1$.

We proceed by induction on $i>1$.
Consider the map 
\[
\pi\colon \conf_i(M_\ast) \longrightarrow   \conf_{i-1}(M_\ast)
\]
given by sending a non-base element $\bigl(c\colon \{1,\dots, i\} \hookrightarrow M\bigr)$ to its restriction $\bigl(c_{|}\colon \{1,\dots, i-1\} \hookrightarrow M\bigr)$.   
Recall that the $E^2$-term of the Leray spectral converging to $\ov{\sH}_{p+q}\bigr(\conf_i(M_\ast); A\bigl)$ is a direct sum of groups of the form $\ov{\sH}_p\bigl(\conf_{i-1}(M_\ast); \sR_q\pi_\ast A\bigr)$.
Here $\sR_q \pi_\ast A$ is the degree-$q$ homology of the sheaf of chain complexes on $\conf_{i-1}(M_\ast)$ which is the pushforward of the sheaf $A$ along the continuous map $\pi$.  
By induction it is enough to show that the push-forward $\sR_q\pi_\ast A$ vanishes for $q>n-1$; equivalently, each point $c\in \conf_{i-1}(M_\ast)$ has a neighborhood $U_c$ such that $\ov{\sH}_q(\pi^{-1}U_c, A_{|\pi^{-1}U_c})$ vanishes for $q>n-1$. Since $\pi$ is a fiber bundle away from the base point, there are two cases to check. 
First, if $c = (\{1,\dots,i\}\hookrightarrow M)$ is not the base point, the result follows from the vanishing of $\ov{\sH}_q(M_\ast \smallsetminus c(\{1,\cdots,i-1\})\};A)$ for $q>n-1$ considered in the base case of the induction.
 Second, if $c$ is the base point, then the inverse image $\pi^{-1}c$ has a contractible neighborhood (again from the definition of zero-pointedness).

To prove the general case, consider $M_\ast\cong M_{1,\ast} \vee \ldots \vee M_{\ell,\ast}$, where each $M_{j,\ast}$ is connected. 
Notice the isomorphism 
\[
\conf_i(M_\ast)~{}~ \cong \bigvee_{\{1,\dots,i\} \xra{f} \{1,\dots, \ell\}} \conf_{i_1}(M_{1,\ast})\wedge\ldots\wedge \conf_{i_\ell}(M_{\ell,\ast})
\] 
where here we have used the notation $i_r$ for the cardinality of each set $ f^{-1}r$.  
So it suffices to bound the degree of the non-vanishing homology of each term. 
Because each $M_{j,\ast}$ is connected, we have already concluded that $\conf_{i_r}(M_{r,\ast})$ has homology bounded by the sum of the bounds for $M_{j,\ast}\smallsetminus \{c(\{1,\dots,i\})\}$ and $\conf_{i_r-1}(M_{j,\ast})$, for a point $c\in \conf_{i_j-1}(M_{j,\ast})$ a non-base point. 
We assess these bounds separately.
Should the zero-pointed manifold $M_{r,\ast}$ be such that $M_r:=M_{r,\ast}\smallsetminus \ast$ is compact, then $M_r$ has a the degree of the non-vanishing homology is bounded of $n$. The topological space $M_{j,\ast} \smallsetminus \{c(\{1,\dots,i\})\}$ is not compact, and so $\conf_{i_r-1}(M_{r,\ast}\smallsetminus \{c\})$ has a bound of $(n-1)(i_j-1)$. 
The result follows by summing over $r\in \{1,\dots,\ell\}$.
\end{proof}

\begin{example}
In the case $M_\ast = M_+$, then $\conf_i(M_\ast) = \conf_i(M)_+$, and the coconnectivity statement of Proposition~\ref{conf-stuff} follows by induction on $i$ through the standard fibration sequence
$M\smallsetminus \{x_1,\dots,x_i\} \to \conf_{i+1}(M_\ast) \to \conf_i(M_\ast)$.

\end{example}

\begin{remark}[$B$-structures]\label{B-structures}
The results of~\cite{ZP} extend to zero-pointed manifolds equipped with a $B$-structure, where $B \to \BO(n)$ is a map of spaces.  
We will make use of this generalization for the simple case of Atiyah duality where $B \to \BO(ni)$ is the classifying space of the block-sum homomorphism $\Sigma_i\wr \sO(n) \to \sO(ni)$, and the $\sB\bigl(\Sigma_i\wr \sO(n)\bigr)$-structured zero-pointed manifolds are of the form $\conf_i(M_\ast)$ and $\conf_i^\neg(M_\ast)$.  

\end{remark}

\subsection{Reduced factorization homology}\label{sec:reduced}

Theorem~\ref{zero-facts} justifies the following definitions.
\begin{definition}[Reduced factorization (co)homology]\label{def:fact-homology}
Let $\cV$ be a symmetric monoidal $\infty$-category.

The $\infty$-categories of \emph{augmented $n$-disk algebras} and of \emph{augmented $n$-disk coalgebras}, respectively,
are those of symmetric monoidal functors
\[
\Alg_n^{\sf aug}(\cV) := \Fun^\ot\bigl(\Disk_{n,+},\cV)\qquad\text{ and }\qquad \cAlg_n^{\sf aug}(\cV) := \Fun^\ot(\Disk_n^+,\cV)~.
\] 
Restrictions along the inclusions $\Disk_{n,+} \hookrightarrow \ZMfld_n\hookleftarrow \Disk_n^+$ have (a priori partially defined) adjoints depicted in the diagram
\[
\xymatrix{
\bBar \colon \Alg_n^{\sf aug}(\cV) \ar@{-->}@(-,u)[rr]^-{\int_-}
&&
\Fun^\ot(\ZMfld_n,\cV)  \ar@(-,u)[rr]^-{|_{\Disk_n^+}}  \ar@(-,d)[ll]^-{|_{\Disk_{n,+}}} 
&&
\cAlg_n^{\sf aug}(\cV) \colon \cBar  \ar@{-->}@(-,d)[ll]^-{\int^{(-)^\neg}}
}
\]
whose left and right composites are as depicted.  
Explicitly, for $A$ an augmented $n$-disk algebra, $C$ an augmented $n$-disk coalgebra, and $M_\ast$ a zero-pointed $n$-manifold, the values of these adjoints are given as
\[
\int_{M_\ast}A := \colim\Bigl((\Disk_{n,+})_{/M_\ast} \to \Disk_{n,+}\xra{A} \cV\Bigr)
\]
and
\[
\int^{M_\ast} C := \limit\Bigl((\Disk_n^+)^{M_\ast^\neg/} \to \Disk_n^+ \xra{C} \cV\Bigr)
\]
which we refer to respectively as the \emph{factorization homology} $M_\ast$ with coefficients in $A$, and as the \emph{factorization cohomology} of $M_\ast$ with coefficients in $C$. Note that factorization cohomology is contravariant, whereas factorization homology is covariant.

\end{definition}

\begin{theorem}[\cite{ZP}]\label{fact.exists}
Let $\cV$ be a symmetric monoidal $\infty$-category which admits sifted colimits.  
If $M_\ast$ is a zero-pointed $n$-manifold and $A$ is an augmented $n$-disk algebra, then $\int_{M_\ast} A$ exists and $\int_- A$ defines a covariant functor to $\cV$ from zero-pointed $n$-manifolds.
In addition, $\int_- A$ defines a \emph{symmetric monoidal} functor from zero-pointed $n$-manifolds provided $\cV$ is $\ot$-sifted cocomplete.
The dual result holds for $\int^- C$.  

\end{theorem}

\begin{remark} The dual conditions for factorization homology typically do not hold in cases of interest. For instance, when $\cV$ is chain complexes, the tensor product does not distribute over cosifted limits limits, although direct sum does.
\end{remark}

There is a canonical comparison arrow between factorization homology and factorization cohomology.
\begin{theorem}[Poincar\'e/Koszul duality map~(\cite{ZP})]\label{thm.PD-map}
Let $A$ be an $n$-disk algebra in $\cV$, a symmetric monoidal $\oo$-category which admits sifted colimits and cosifted limits.  
Let $M_\ast$ be a zero-pointed $n$-manifold.
There is a canonical arrow in $\cV$
\begin{equation}\label{PD-map}
\int_{M_\ast} A ~\longrightarrow~ \int^{M_\ast^\neg} \bBar A
\end{equation}
which is functorial in $M_\ast$ and $A$.  

\end{theorem}

The present work continues the analysis of this Poincar\'e/Koszul duality map. In \cite{ZP}, we showed that it is an equivalence in several instances. One special case is when $\cV$ is the $\oo$-category of spaces with Cartesian product; in this case factorization cohomology is a mapping space. Another special case is where $\cV$ is chain complexes equipped with direct sum; in this case, factorization homology is usual homology with coefficients in the chain complex $A$ (possibly twisted by the $\sO(n)$ action), factorization cohomology is usual generalized cohomology with coefficients in $A[n]$ (where $\sO(n)$ acts by the sign representation), and this equivalence this usual Poincar\'e duality with twisted coefficients. In this work, we study this map when $\cV$ is chain complexes with tensor product (or, more generally, a stable symmetric monoidal $\oo$-category which is $\ot$-cocomplete).

\subsection{Exiting disks}\label{sec.exiting-disks}
The slice $\infty$-category $\Disk_{n,+/M_\ast}$ appears in the defining expression for factorization homology.  
We review a variant of this $\infty$-category, $\Disk_+(M_\ast)$, of \emph{exiting disks} in $M_\ast$, which offers several conceptual and technical advantages.  
Heuristically, objects of $\Disk_+(M_\ast)$ are embeddings from finite disjoint unions of disks into $M$, while morphisms are isotopies of such to embeddings with some of these isotopies witnessing disks slide off to infinity where they are forgotten. Disks are not allowed to slide in \emph{from} infinity, unlike in $\Disk_{n,+/M_\ast}$.  
To define this $\infty$-category of exiting disks, we require the regularity around the zero-point granted by a conically smooth structure.

For this section, we fix a conically smooth zero-pointed $n$-manifold $M_\ast$.  
We recall the following notion from~\S4 of~\cite{aft1}.  

\begin{definition}[\cite{aft1}]\label{def.basics}
A \emph{basic}, or \emph{basic singularity type}, is a stratified space of the form 
\[
\RR^i\times \sC(L)
\]
where $i$ is a finite cardinality, $L$ is a compact stratified space, and $\sC(L):=\ast \underset{L\times \{0\}} \amalg L\times [0,1)$ is its \emph{open cone},
where here the half-open interval is stratified as the two strata $\{0\}$ and $(0,1)$.  
The $\infty$-category of \emph{basics} is the full $\infty$-subcategory consisting of the basics
\[
\Bsc~\subset~\Snglr
\]
in the $\infty$-category of stratified spaces and spaces of open embeddings among them.  
\end{definition}

In~\S2 of~\cite{aft2} appears a stratified version of $\Disk_{n/M}$, which we now recall.  
\begin{definition}[\cite{aft2}]\label{def.disk-basics}
For each stratified space $X$ we denote the full $\infty$-subcategory
\[
\Disk(\bsc)_{/X}~\subset~ \Snglr_{/X}
\]
consisting of those open embeddings $U\hookrightarrow X$ for which $U$ is isomorphic to a finite disjoint union of basics.  
\end{definition}

\begin{definition}[$\Disk_+(M_\ast)$]\label{def.of-M}
The $\infty$-category of \emph{exiting disks} of $M_\ast$ is the full $\infty$-subcategory
\[
\Disk_+(M_\ast)~\subset~ \Disk(\bsc)_{/M_\ast}
\]
consisting of those $V\hookrightarrow M_\ast$ whose image contains $\ast$.  
We use the notation
\[
\Disk^+(M^\neg_\ast) ~:=~ \Disk_+(M_\ast)^{\op}~.  
\]

\end{definition}
\noindent
Explicitly, an object of $\Disk_+(M_\ast)$ is a conically smooth open embedding $B\sqcup U \hookrightarrow M_\ast$ where $B\cong \sC(L)$ is a cone-neighborhood of $\ast\in M_\ast$ and $U$ is abstractly diffeomorphic to a finite disjoint union of Euclidean spaces, and a morphism is a isotopy to an embedding among such.

For each zero-pointed manifold $M_\ast$, the unique zero-pointed embedding $\ast \to M_\ast$ induces the functor
\[
\Disk_{n,+} = \Disk_{n,+/\ast} \longrightarrow \Disk_{n,+/M_\ast} ~.
\]
We denote the resulting pushout among $\infty$-categories as
\[
\xymatrix{
\Disk_{n,+}   \ar[r]  \ar[d]
&
\Disk_{n,+/M_\ast}   \ar[d]
\\
\ast \ar[r]
&
\bigl(\Disk_{n,+/M_\ast}\bigr)/\bigl(\Disk_{n,+}\bigr).
}
\]
The next result makes reduced factorization homology tractable.  
\begin{theorem}[\cite{ZP}]\label{thank-god}
\begin{enumerate}
\item[]

\item The $\infty$-category $\Disk_+(M_\ast)$ is sifted. 

\item There is a final functor
\[
\Disk_+(M_\ast)\longrightarrow \Disk_{n,+/M_\ast}
\]
whose value on $\bigl(B\sqcup U\hookrightarrow M_\ast\bigr)$ is represented by $(U_+\hookrightarrow M_\ast)\in \Disk_{n,+/M_\ast}$.

\end{enumerate}

\end{theorem}

Consider the composite functor
\begin{equation}\label{extend-of-M}
\Alg^{\sf aug}_n(\cV) \longrightarrow \Fun({\Disk_{n,+/M_\ast}},\cV) \longrightarrow \Fun\bigl(\Disk_+(M_\ast), \cV\bigr)~.
\end{equation}
The first arrow is restriction along the projection $\Disk_{n,+/M_\ast} \to \Disk_{n,+}$;
the second arrow is restriction along that asserted in Theorem~\ref{thank-god}.  

\begin{notation}\label{just-as-A}
Given an augmented $n$-disk algebra $A\colon \Disk_{n,+} \to \cV$, we will use the same notation 
$
A\colon \Disk_+(M_\ast)\to \cV
$
for the value of the functor~(\ref{extend-of-M}) on $A$.  

\end{notation}
We content ourselves with this Notation~\ref{just-as-A} because of the immediate corollary of Theorem~\ref{thank-god}.
\begin{cor}\label{still-computes}
Let $\cV$ be a symmetric monoidal $\infty$-category whose underlying $\infty$-categoy admits sifted colimits.
Let $A\colon \Disk_{n,+}\to \cV$ be an augmented $n$-disk algebra, and let $C\colon \Disk_n^+\to \cV$ be an augmented $n$-disk coalgebra.  There are canonical identifications in $\cV$:
\[
\int_{M_\ast}  A 
~{}~\simeq~{}~
\underset{(B\sqcup U\hookrightarrow M_\ast)\in \Disk_+(M_\ast)}\colim~ A(U_+)~,
\]
and
\[
\int^{M_\ast}  C
~{}~\simeq~{}~
\underset{(B\sqcup V\hookrightarrow M_\ast)\in \Disk^+(M^\neg_\ast)}\limit~ C(V^+)~.
\]

\end{cor}

\subsection{Free and trivial algebras}\label{Free and trivial algebras}
We give two procedures for constructing augmented $n$-disk algebras.
In this subsection we fix a symmetric monoidal $\infty$-category $\cV$ which is $\ot$-presentable.
From Corollary 3.2.3.5 of~\cite{HA}, the $\infty$-category $\Alg_n^{\sf aug}(\cV)$ is presentable.

\begin{definition}[$\sO(n)$-modules]\label{def:top-module}
Let $G$ be a topological group and let $\cV$ be an $\infty$-category.
The $\infty$-category of \emph{$G$-modules} is the functor category
\[
\Mod_{G}(\cV) := \Fun(\mathsf{B}G,\cV)
\]
from the $\infty$-groupoid associated to the classifying space of $G$.  

\end{definition}

\begin{warning}\label{warning}
In the case where the topological group is the orthogonal group, $G=\sO(n)$, and $\cV=\Ch_\Bbbk$ is chain complexes over a ring $\Bbbk$, there is an equivalence
\[
\m_{\sO(n)}(\Ch_\Bbbk)\simeq \Mod_{\sC_\ast(\sO(n);\Bbbk)}(\Ch_\Bbbk) 
\]
between $\sO(n)$-modules (in the sense of Definition \ref{def:top-module}) and modules for the differential graded algebra of $\Bbbk$-linear chains on $\sO(n)$. This should {\it not} be confused, in the case $\Bbbk$ is $\RR$ or $\CC$, with the usual category of representations of $\sO(n)$ as a Lie group. There is a functor from the representation category to the functor category, but it is far from being an equivalence.

\end{warning}

Throughout this work we will make use of the basic and essential result from differential topology that the inclusions
\[
\sO(n)\hookrightarrow {\sf GL}(n)\hookrightarrow \Diff(\RR^n)\hookrightarrow \Emb(\RR^n,\RR^n)
\]
are all homotopy equivalences. Gram--Schmidt orthogonalization defines a deformation retraction onto the inclusion $\sO(n) \xra{\simeq} {\sf GL}(n)$. Conjugating by scaling and translation, $(f,t)\mapsto \bigl(x\mapsto  \frac{f(tx)-f(0)}{t}+f(0)\bigr)$, then defines a deformation retraction onto the inclusion ${\sf GL}(n) \xra{\simeq} \Emb(\RR^n,\RR^n)$.

Write $\cV_{\uno//\uno}$ for the $\infty$-category of objects $E\in \cV$ equipped with a retraction onto the symmetric monoidal unit: $\id_{\uno}\colon \uno\to E\to \uno$. Note that if $\cV$ is stable, then there is a natural equivalence $\cV_{\uno//\uno}\xra{\sim} \cV$ sending an object $\uno\to E\to \uno$ to the cokernel of the unit ${\sf cKer}(\uno\to E)$.
Write $\Disk_{n,+}^{\leq 1}\subset \Disk_{n,+}$ for the full $\oo$-subcategory consisting of those zero-pointed Euclidean spaces with at most one non-base component. By the equivalence $\sO(n)\simeq \Emb(\RR^n,\RR^n)$, this full $\oo$-subcategory is initial among pointed $\infty$-categories under $\BO(n)$.  
In other words, there there is a canonical equivalence of $\infty$-categories
\[
\Fun_{\uno}(\Disk_{n,+}^{\leq 1}, \cV) \xra{~\simeq~}\Mod_{\sO(n)}(\cV_{\uno//\uno})
\]
where the source is functors whose value on $\ast$ is a symmetric monoidal unit of $\cV$; the target is $\sO(n)$-modules in retractive objects over the unit.  
\begin{definition}[Free]\label{free}
Restriction along $\Disk_{n,+}^{\leq 1} \subset \Disk_{n,+}$ determines the solid arrow, referred to as the \emph{underlying $\sO(n)$-module}:
\begin{equation}\label{free-forgetful}
\xymatrix{
\Alg^{\sf aug}_n(\cV)  \ar[rr]_-{_{|\Disk_{n,+}^{\leq 1}}}
&&
\Mod_{\sO(n)}(\cV_{\uno//\uno}) \ar@/_1pc/@{-->}[ll]_{\FF^{\sf aug}} ~.
}
\end{equation}
This forgetful functor preserves limits, and so there is a left adjoint, as depicted, referred to as the \emph{augmented free} functor.
\end{definition}

If $\cV$ is a stable $\oo$-category, then the forgetful functor from algebras to $\sO(n)$-modules has an inverse, the {\it trivial algebra} functor ${\sf t^{aug}}$. Given an $\sO(n)$-module $V$, the augmented $n$-disk algebra ${\sf t^{aug}}V$ has $V$ as its underlying $\uno \oplus \sO(n)$-module; the restriction to $V$ of the multiplication map $V\ot V \to V$ factors as the augmentation followed by the unit: $V\ot V \to \uno\ot \uno \simeq \uno \to V$. See \S7.3 of \cite{HA} for a formal construction. This allows the following definition of the adjoint

\begin{definition}[Cotangent space]\label{def:cotangent-space}
The \emph{augmented cotangent space} functor is the left adjoint to the \emph{augmented trivial} functor:
\[
\xymatrix{
\Alg^{\sf aug}_n(\cV)  \ar@/_1.5pc/@{-->}[rr]^-{L^{\sf aug}}
&&
\Mod_{\sO(n)}(\cV_{\uno//\uno}) \ar[ll]_{\sf t^{aug}}~.
}
\]
\end{definition}

The cotangent space is an inverse to the free functor.
\begin{lemma}\label{LF}
Let $\cV$ be a $\ot$-presentable symmetric monoidal $\infty$-category.
There is a canonical equivalence of endofunctors of $\Mod_{\sO(n)}(\cV_{\uno//\uno})$:
\[
L^{\sf aug}\circ \FF^{\sf aug} \xra{~\simeq~} \id~.
\]
\end{lemma}
\begin{proof}
The composition of the forgetful functor and the trivial algebra functor is equivalent to the identity, therefore the composite of their left adjoints is the identity. 
\end{proof}

\subsubsection{\bf Stable case}\label{stable-case}
In this section, fix a $\ot$-stable-presentable symmetric monoidal $\infty$-category $\cV$. Recall that $\cV$ is naturally tensored over the $\oo$-category of pointed spaces. Using this structure, we define a functor  $ -\underset{\sO(n)}\bigotimes -:  \Mod_{\sO(n)}(\Spaces_\ast) \times\Mod_{\sO(n)}(\cV) \ra \cV$ by the following composite:
\[\xymatrix{
 \Mod_{\sO(n)}(\Spaces_\ast) \times\Mod_{\sO(n)}(\cV)\ar[rr]\ar[d]&& \cV\\
 \Mod_{\sO(n)\times\sO(n)}(\Spaces_\ast \times \cV)\ar[r]&\Mod_{\sO(n)\times\sO(n)}(\cV) \ar[r]&\Mod_{\sO(n)}(\cV)\ar[u]\\}
\]
where the second step is given by the tensoring operation $\Spaces_\ast\times\cV \ra \cV$; the third step is restriction along the diagonal map $\sO(n) \ra \sO(n)\times\sO(n)$; the last step is taking the coinvariants of the action by $\sO(n)$. Dually, we define a functor
\[
\Map^{\sO(n)}(-,-)\colon \Mod_{\sO(n)}(\Spaces_\ast)^{\op} \times \Mod_{\sO(n)}(\cV) \longrightarrow \cV
\]
by substituting the cotensor $(\Spaces_\ast)^{\op} \times\cV \ra \cV$ for the tensor, and invariants for coinvariants.

In the case that $\cV$ is stable there is an equivalence $\Ker^{\sf aug}\colon \cV_{\uno//\uno} \simeq \cV\colon \uno\oplus (-)$, and thereafter an equivalence 
\begin{equation}\label{w/o-units}
\Ker^{\sf aug}\colon \Mod_{\sO(n)}(\cV_{\uno//\uno})~ \simeq~ \Mod_{\sO(n)}(\cV)\colon \uno\oplus (-)~.
\end{equation}
\begin{notation}[$\FF$ and $L$]
In the case that $\cV$ is stable, we denote
\[
\FF\colon \Mod_{\sO(n)}(\cV) \underset{(\ref{w/o-units})}\simeq\Mod_{\sO(n)}(\cV_{\uno//\uno})\xra{\FF^{\sf aug}} \Alg_n^{\sf aug}(\cV)~,
\]
\[
\xymatrix{
L\colon \Alg_n^{\sf aug}(\cV) \ar@<1ex>[r]^-{L^{\sf aug}}
&
\Mod_{\sO(n)}(\cV_{\uno//\uno}) \underset{(\ref{w/o-units})}\simeq\Mod_{\sO(n)}(\cV) \colon {\sf t}  \ar@<1ex>[l]^-{\sf t^{aug}}~.
}
\]

\end{notation}

In this stable case, we make $\FF$ explicit and so recognize $L$.
The underlying $\sO(n)$-module of the value of $\FF$ on a $\sO(n)$-module $V$ is
\begin{equation}\label{free-value}
\FF(V)  ~\simeq~ \bigoplus_{i\geq 0} \Bigl(\conf_i^{\sf fr}(\RR^n_+) \underset{\Sigma_i\wr \sO(n)}\bigotimes V^{\ot i}\Bigr)~,
\end{equation}
because the monoidal structure distributes over colimits.
Stability of $\cV$ implies stability of $\Mod_{\sO(n)}(\cV)$.  In the solid diagram among $\infty$-categories
\[
\xymatrix{
&
{\sf Stab}\bigl(\Alg_n^{\sf aug}(\cV)\bigr)   \ar@{-->}[dr]^-{\alpha}
&
\\
\Alg_n^{\sf aug}(\cV) \ar[rr]^-{L}   \ar[ur]^-{\Sigma^\infty}
&
&
\Mod_{\sO(n)}(\cV)
}
\]
there is a canonical filler, from the stabilization, as a colimit preserving functor. See \S7.3.4 of \cite{HA} or 
Proposition 2.23 of~\cite{cotangent}, which state that the functor $\alpha$ is an equivalence, and so the horizontal functor in the above diagram witnesses $\Mod_{\sO(n)}(\cV)$ as the stabilization of $\Alg_n^{\sf aug}(\cV)$.

\subsection{Linear Poincar\'e duality}
In the case that the symmetric monoidal $\infty$-category is of the form $\cS^\oplus$, with underlying $\infty$-category $\cS$ stable and presentable, and whose symmetric monoidal structure is given by direct sum, factorization homology and factorization cohomology profoundly simplify.  
Here we state Poincar\'e duality in this simplified setting.

For $\cX$ a small $\infty$-category with a zero object, we denote by $\Psh_\ast(\cX)$ the $\infty$-category of those (space-valued) presheaves on $\cX$ whose value on the zero object is $\ast$.  
\begin{definition}[Frame bundle]\label{def:frame bundle}
The \emph{frame bundle} functor is the composition
\[
{\sf Fr}\colon  \ZMfld \to \Psh_\ast(\ZMfld) \to \Psh\bigl(\BO(n)\bigr)^{\ast /} \simeq \Mod_{\sO(n)^{\op}}(\Spaces_\ast)
\]
of the Yoneda embedding, followed by restriction along the full subcategory $\Disk_{n,+}^{\leq 1}\subset \ZMfld$ -- this subcategory is initial among $\infty$-categories under $\BO(n)$ with a zero object.  
Explicitly, ${\sf Fr}_{M_\ast}$ can be identified as the pointed space $\ZEmb(\RR^n_+,M_\ast)$ with $\sO(n)$-action given by precomposition by homeomorphisms of $\RR^n$.  

\end{definition}

\begin{notation}[Framed configurations]\label{framed-conf}
For $M_\ast$ a zero-pointed $n$-manifold,
we denote the $\Sigma_i\wr \sO(n)$-module in pointed spaces
\[
\conf^{\sf fr}_i(M_\ast) := {\sf Fr}_{\conf_i(M_\ast)}~.
\]

\end{notation}

The following is a formulation of Poincar\'e or Atiyah duality for a suitable class of zero-pointed manifolds.

\begin{theorem}[Linear Poincar\'e duality~\cite{ZP}]\label{linear-PD}
Let $\cS$ be a stable and presentable $\infty$-category, and let $E$ and $F$ be $\sO(n)$-modules in $\cS$.
Let $M_\ast$ be a zero-pointed $n$-manifold. 
A morphism of $\sO(n)$-modules $\alpha\colon (\RR^n)^+ \ot E \to F$
canonically determines a morphism in $\cS$
\[
\alpha_{M_\ast}\colon {\sf Fr}_{M_\ast} \underset{\sO(n)}\bigotimes E\longrightarrow \Map^{\sO(n)}({\sf Fr}_{M_\ast^\neg},F)~.
\]
Furthermore, if $\alpha$ is an equivalence then so is $\alpha_{M_\ast}$.

\end{theorem}

\begin{remark}[With $B$-structures]\label{PD-B}
We follow up on Remark~\ref{B-structures}.
There is a version of Theorem~\ref{linear-PD} that is also true in the context of $B$-manifolds -- it is stated and proved there.  
We will make use of this version as it applies to $\sB\bigl(\Sigma_i\wr \sO(n)\bigr)$-structured zero-pointed manifolds of the form $\conf_i(M_\ast)$.  

\end{remark}

\section{Filtrations}
In this section we establish the cardinality (co)filtration of factorization (co)homology, as well as the Goodwillie cofiltration of factorization homology.  
We use these to give partial results for Poincar\'e/Koszul duality for when the monoidal structure of $\cV$ does not distribute over totalizations, as with the case of chain complexes with tensor product or spectra with smash product. Throughout this section, if not otherwise specified, the following parameters are assumed to be fixed.  
\begin{itemize}
\item A zero-pointed $n$-manifold $M_\ast$ (possibly with corners) (see Definition~\ref{def:zero}).
\item A $\ot$-stable-presentable symmetric monoidal $\infty$-category $\cV$.
\item An augmented $n$-disk algebra $A\colon \Disk_{n,+} \to \cV$.
\item An augmented $n$-disk coalgebra $C\colon \Disk_n^+ \to \cV$.  
\end{itemize}

\begin{example}
Here are some standard examples of such entities.
\begin{itemize}
\item 
Let $\ov{N}$ be a smooth cobordism from $\partial_{\sL}$ to $\partial_{\sR}$.
Set $N_\ast :=\ast\underset{\partial_{\sL}}\amalg(\ov{N}\smallsetminus \partial_{\sR})$.
In this case, we have $N_\ast^\neg \cong \ast\underset{\partial_{\sR}}\amalg(\ov{N}\smallsetminus \partial_{\sL})$.
The zero-pointed $n$-manifolds with corners of Lemma~\ref{conf.finitary}, which are pointed homotopy equivalent to configuration spaces:
\[
C_i(N_\ast)~\simeq~\conf_i(N_\ast) = \ast\underset{B}\coprod \bigl\{f\colon \{1,\dots,i\}\hookrightarrow \ov{N}\smallsetminus \partial_{\sR}\bigr\}
\]
where $B=\{f\mid \emptyset \neq f^{-1}(\partial_{\sf R})\}$, and
\[
C^{\neg}_i(N_\ast^\neg)~\simeq~\conf_i^\neg(N_\ast^\neg) = \ast\underset{B'}\coprod \bigl\{f\colon \{1,\dots,i\}\to \ov{N}\smallsetminus \partial_{\sL}\bigr\}
\]
where $B' =\{f\mid \emptyset \neq f^{-1}(\partial_{\sf L}) \text{ or }f \text{ is not injective}\}$.

\item Write $\Ch_\Bbbk$ for the $\infty$-category of chain complexes over a commutative ring $\Bbbk$. 
Then $\Ch_\Bbbk^\oplus$ and $\Ch_\Bbbk^\ot$, equipped with direct sum and tensor product, are examples of such symmetric monoidal $\infty$-categories.
In general, any such $\cV$ is tensored and cotensored over finite spaces.  

\item 
For $\cS$ a stable presentable $\infty$-category, let $E\in \cS$ be an object.
The assignment $A\colon U_\ast \mapsto E^{U_\ast^\neg}$ depicts an augmented $n$-disk algebra in $\cS^\oplus$ -- its underlying object is (non-canonically) identified as $E[-n] \simeq \Omega^n E$.  
Likewise, the assignment $C\colon U_\ast \mapsto U_\ast \ot E$ depicts an augmented $n$-disk coalgebra in $\cS^\oplus$ -- its underlying object is (non-canonically) identified as $E[n]\simeq \Sigma^n E$.  
Constructing $n$-(co)algebras in more general $\cV$ is substantially more interesting, and also more involved, depending on the specifics of $\cV$.  

\end{itemize}

\end{example}

\subsection{Main results}
Here we display the main results in this section and prove them based upon results developed in later subsections.

\subsubsection{\textbf{Cardinality (co)filtration}}
We exhibit a natural filtration of factorization homology and identify the filtration quotients; we do likewise for factorization cohomology. Write $[M_\ast]$ for the set of connected components of $M_\ast$, which we regard as a based set.  
For each finite cardinality $i$, we will denote the subcategories of based finite sets
\[
(\Fin^{\leq i}_\ast)_{/[M_\ast]}~\subset~(\Fin^{\sf surj}_\ast)_{/[M_\ast]}~\subset~(\Fin_\ast)_{/[M_\ast]}
\]
where an object of the middle is a \emph{surjective} based map $I_+\to [M_\ast]$, and a morphism between two such is a \emph{surjective} map over $[M_\ast]$; and where the left is the full subcategory consisting of those $I_+\to[M_\ast]$ for which the cardinality $|I|\leq i$ is bounded. 
Taking connected components gives a functor 
\[
[-]\colon \Disk_+(M_\ast) \to (\Fin_\ast)_{/[M_\ast]}
\]
to based finite sets over the based set of connected components of $M_\ast$.  
\begin{definition}[$\Disk^{\sf \leq i}_+(M_\ast)$]\label{disk-i}
We define the $\infty$-category
\[
\Disk^{\sf surj}_+(M_\ast)~:=~\Disk_+(M_\ast)_{|(\Fin^{\sf surj}_\ast)_{/[M_\ast]}}~\subset~\Disk_+(M_\ast)
\]
and, for each finite cardinality $i$, the full $\infty$-subcategory
\[
\Disk^{\leq i}_+(M_\ast)~:=~\Disk^{\sf surj}_+(M_\ast)_{|(\Fin^{\leq i}_\ast)_{/[M_\ast]}}~\subset~\Disk^{\sf surj}_+(M_\ast)~.
\]
We denote the opposites:
\[
\Disk^{+,\sf surj}(M_\ast^\neg) ~:=~\bigl(\Disk^{\sf surj}_+(M_\ast)\bigr)^{\op}\qquad\text{ and }\qquad
\Disk^{+,\leq i}(M_\ast^\neg)~:=~\bigl(\Disk^{\leq i}_+(M_\ast)\bigr)^{\op}~.
\]

\end{definition}

\begin{definition}\label{def:truncations}
Let $i$ be a finite cardinality.
We define the object of $\cV$
\[
\tau^{\leq i} \int_{M_\ast} A := \underset{(B\sqcup U\hookrightarrow M_\ast)\in \Disk_+^{\leq i}(M_\ast)} \colim A(U_+) = \colim\Bigl(\Disk_+^{\leq i}(M_\ast) \to \Disk_+(M_\ast) \xra{A} \cV\Bigr)~.
\]

We define the object of $\cV$
\[
\tau^{\leq i} \int^{M_\ast^\neg} C := \underset{(B\sqcup V\hookrightarrow M_\ast)\in \Disk^{+,\leq i}(M_\ast)} \limit C(V^+) = \limit\Bigl(\Disk^{+,\leq i}(M_\ast^\neg) \to \Disk_{n}^+(M_\ast^\neg) \xra{C} \cV\Bigr)~.
\]

We likewise define such objects for the comparison ``$\leq  i$'' replaced by other comparisons among finite cardinalities, such as ``$\geq i$'' and ``$=i$''.

\end{definition}

The $\ZZ_{\geq 0}$-indexed sequence of fully faithful functors
\[
\dots \longrightarrow \Disk^{\leq i}_+(M_\ast) \longrightarrow \Disk^{\leq i+1}_+(M_\ast) \longrightarrow \dots \longrightarrow \Disk^{\sf surj}_+(M_\ast)
\]
witnesses $\Disk^{\sf surj}_+(M_\ast)$ as a sequential colimit.  
There results a canonical sequence in $\cV$
\begin{equation}\label{card-compare-ho}
\cdots \longrightarrow \tau^{\leq i-1} \int_{M_\ast}A \longrightarrow \tau^{\leq i} \int_{M_\ast}A \longrightarrow \cdots ~ \longrightarrow\int_{M_\ast}A~.
\end{equation}
Dually, there is a canonical sequence in $\cV$
\begin{equation}\label{card-compare-co}
\int^{M_\ast^\neg} C \longrightarrow\cdots \longrightarrow \tau^{\leq i} \int^{M_\ast^\neg}C  \longrightarrow \tau^{\leq i-1} \int^{M_\ast^\neg} C\longrightarrow \cdots~.
\end{equation}
There are likewise sequences with $\tau^{\leq -}$ replaced by $\tau^{\geq -}$.

\begin{lemma}[Cardinality convergence]\label{cardinality-converges}
The morphism in $\cV$ from the colimit of the cardinality sequence
\[
\tau^{\leq \infty}\int_{M_\ast} A~\xra{~\simeq~}~\int_{M_\ast} A
\]
is an equivalence.
Likewise, the morphism in $\cV$ to the limit
\[
\int^{M_\ast^\neg} C~\xra{~\simeq~}~ \tau^{\leq \infty}\int^{M^\neg_\ast} C
\]
is an equivalence.  

\end{lemma}
\begin{proof}
Directly apply Corollary~\ref{still-computes} and Lemma~\ref{surj-final} which states that $\Disk_+^{\sf surj}(M_\ast) \to \Disk_+(M_\ast)$ is final. 
\end{proof}

The following main result of this subsection identifies the layers of the filtration~(\ref{card-compare-ho}) and the cofiltration~(\ref{card-compare-co}) in terms of configuration spaces.  
We give a proof of this result at the end of~\S\ref{sec.reduced}.

\begin{theorem}[Cardinality cokernels and kernels]\label{truncation-quotients}
Each arrow in the cardinality filtration of factorization homology in~(\ref{card-compare-ho}) belongs to a canonical cofiber sequence
\begin{equation}\label{ho-cofiber}
\tau^{\leq i-1} \int_{M_\ast} A \longrightarrow \tau^{\leq i}\int_{M_\ast}A \longrightarrow \conf^{\neg, {\sf fr}}_i(M_\ast) \underset{\Sigma_i\wr \sO(n)} \bigotimes  A^{\ot i}~.
\end{equation}
Likewise, each arrow in the cardinality cofiltration of factorization cohomology in~(\ref{card-compare-co}) belongs to a canonical fiber sequence
\begin{equation}\label{ho-fiber}
\Map^{\Sigma_i\wr \sO(n)}\bigl(\conf^{\neg, {\sf fr}}_i(M_\ast),C^{\ot i}\bigr) \longrightarrow
\tau^{\leq i} \int^{M_\ast} C \longrightarrow \tau^{\leq i-1}\int^{M_\ast}C~.
\end{equation}

\end{theorem}

\begin{remark}\label{hodge}
In the $n=1$ case where the manifold $M$ is the circle and $A$ is an associative algebra, the cardinality filtration specializes to the Hodge filtration of Hochschild homology studied by Burghelea--Vigu-Poirrier \cite{burvi}, Feigin--Tsygan \cite{feigintsygan1}, Gerstenhaber--Schack \cite{gs}, and Loday \cite{loday}.
\end{remark}

\subsubsection{\textbf{Goodwillie cofiltration}}\label{sec:goodwillie-of-M}
We are about to apply Goodwillie's calculus to functors of the form $\Alg_n^{\sf aug}(\cV) \to \cV$.  This calculus was developed by Goodwillie in \cite{goodwillie} for application to functors from pointed spaces and has since been generalized; see \cite{HA}, \cite{kuhn1}, and \cite{kuhn2}.
We begin by recalling this formalism.

Let $i$ be a finite cardinality.  
Let $\cX$ and $\cY$ be presentable $\infty$-categories, each with a zero object.
The $\oo$-category of \emph{polynomial functors of degree $i$} is the full $\oo$-subcategory of reduced functors
\[
{\sf Poly}_i(\cX,\cY) ~\subset~ \Fun_0(\cX,\cY)
\]
consisting of those that send strongly coCartesian $(i+1)$-cubes to Cartesian $(i+1)$-cubes.  
This inclusion admits a left adjoint 
\[
P_i\colon  \Fun_0(\cX,\cY) \longrightarrow {\sf Poly}_i(\cX,\cY)
\]
implementing a localization.
There is the full $\oo$-subcategory of \emph{homogeneous functors of degree $i$} 
\[
{\sf Homog}_i(\cX,\cY)~\subset~{\sf Poly}_i(\cX,\cY)
\]
consisting of those polynomial functors $H$ of degree $i$ for which $P_j H \simeq 0$ is the zero functor provided $j<i$.  
Consequently, each reduced functor $F\colon\cX\to \cY$ canonically determines a cofiltration of reduced functors
\[
F \to P_\infty F \to \cdots \to P_{i}F \to P_{i-1} F \to \cdots~
\]
with each composite arrow $F\to P_i F$ the unit for the above adjunction, and with each kernel $\Ker\bigl(P_i F \to P_{i-1}F\bigr)$ homogeneous of degree $i$. Here we denote $P_\infty F := \underset{i}\limit ~P_i F$ for the inverse limit of the cofiltration.

For a fixed zero-pointed $n$-manifold $M_\ast$, we apply this discussion to the functor $\int_{M_\ast} \colon \Alg^{\sf aug}_n(\cV) \to \cV_{\uno//\uno}\simeq \cV$. Here we have used the identification of retractive objects over the unit of $\cV$ with $\cV$, as discussed in \S\ref{Free and trivial algebras}.  
In particular, there is a canonical arrow between functors
\begin{equation}\label{completion-tower}
\int_{M_\ast}~ \longrightarrow ~P_\infty \int_{M_\ast}~. 
\end{equation}

The next result identifies the layers of the Goodwillie cofiltration of factorization homology in terms of configuration spaces and the cotangent space.  
We give a proof of this result at the end of~\S\ref{sec.good-layers}.

\begin{theorem}\label{goodwillie-layers} 
There is a canonical fiber sequence among functors $\Alg_n^{\sf aug}(\cV) \to \cV$:
\[
\conf_i^{\sf fr}(M_\ast)\underset{\Sigma_i\wr \sO(n)}\bigotimes L(-)^{\ot i} \longrightarrow P_i \int_{M_\ast}\longrightarrow P_{i-1}\int_{M_\ast}
\] for every $i$ a finite cardinality.
\end{theorem}

Unlike the cardinality (co)filtration of factorization (co)homology of the previous subsection, the Goodwillie cofiltration of factorization homology does not always converge.
The next result provides an understood class of parameters for which the Goodwillie cofiltration converges.

We will reference the notion of a \emph{t-structure} on $\cV$, which is just a usual t-structure on the homotopy category of $\cV$ (see Definition 1.2.1.4 of ~\cite{HA}). That is, it consists of fully faithful inclusions
\[
\cV_{>0} \hookrightarrow \cV \hookleftarrow \cV_{\leq 0}~.
\]
These inclusions admit a right adjoint $\pi_{>0}:\cV \ra \cV_{>0}$ and a left adjoint $\pi_{\leq 0}:\cV \ra \cV_{\leq 0}$. For any object $V$, the units and counits of these adjunctions define a natural cofiber sequence
\[\pi_{>0} V \ra V\ra \pi_{\leq 0}V~\]
We say that a t-structure on $\cV$ is \emph{compatible with the symmetric monoidal structure} of $\cV$ if the restriction of the functor $\cV \times \cV \xra{\ot} \cV$ to $\cV_{\geq 0}\times\cV_{\geq 0}$ factors through $\cV_{\geq 0}$, and if the unit $\uno\in \cV_{\geq 0}$ is connective. A t-structure is {\it cocomplete} if the natural map $V \ra \varprojlim \pi_{\leq i} V$ is an equivalence for every object $V\in\cV$. Equivalently, cocompleteness means that the inverse limit $\varprojlim \pi_{> i}V$ is the zero object for every $V$.

We give a proof of this next result at the end of~\S\ref{sec.conv}.
\begin{theorem}\label{convergence}
Suppose there exists a cocomplete t-structure on $\cV$ that is compatible with the symmetric monoidal structure.
The value of the canonical arrow~(\ref{completion-tower}) evaluated on $A$
\[
\int_{M_\ast} A \longrightarrow P_\infty \int_{M_\ast} A
\]
is an equivalence in $\cV$ provided either of the following criteria is satisfied:
\begin{itemize}
\item the augmentation ideal $\Ker\bigl(A \to \uno\bigr)$ is connected (with respect to the given t-structure);
\item the topological space $M_\ast$ is connected and compact and the augmentation ideal $\Ker\bigl(A\to \uno\bigr)$ is connective (with respect to the given t-structure).
\end{itemize}
\end{theorem}

\begin{remark}
A similar result is treated by Matsuoka in \cite{matsuoka}.
\end{remark}

\subsubsection{\textbf{Comparing cofiltrations}}
We compare the cardinality and Goodwillie cofiltrations.   Recall the Poincar\'e duality arrow~(\ref{PD-map}) of Theorem~\ref{thm.PD-map}.  

We give a proof of the next result at the end of~\S\ref{sec.compare}.  
\begin{theorem}\label{compare-towers}
The Poincar\'e/Koszul duality arrow $\int_{(-)} \to \int^{(-)^\neg} \bBar$ extends to an equivalence of cofiltrations of functors $\Fun(\Alg_n^{\sf aug},\cV)$ for each zero-pointed $n$-manifold $M_\ast$ with $M$ connected:
\begin{equation}\label{good-to-card}
P_\bullet \int_{M_\ast} \xra{~\simeq~} \tau^{\leq \bullet} \int^{M_\ast^\neg} \bBar ~.
\end{equation}
In particular, for each compact smooth $n$-manifold $\ov{M}$ with partitioned boundary $\partial \ov{M} = \partial_{\sL} \sqcup \partial_{\sR}$, each augmented $n$-disk algebra $A$ in $\cV$, and each finite cardinality $i$, there is an equivalence in $\cV$
\[
P_i\int_{\ov{M}\smallsetminus \partial_{\sR}} A \xra{~\simeq~} \tau^{\leq i}\int^{\ov{M}\smallsetminus \partial_{\sL}} \bBar(A)~.
\]  

\end{theorem}

\begin{cor}\label{goodwillie-limit}
The Poincar\'e/Koszul duality arrow $\int_{M_\ast} A \to \int^{M_\ast^\neg}\bBar A$ of~(\ref{PD-map}) canonically factors through an equivalence in $\cV$
\[
P_\infty \int_{M_\ast}A\xra{~{}~\simeq~{}~} \int^{M_\ast^\neg} \bBar A
\]
from the limit of the Goodwillie cofiltration.
This factorization is functorial in $M_\ast$ and $A$.  

\end{cor}

\subsection{Non-unital algebras}
In an ambient stable situation, there is a convenient equivalence between augmented algebras and non-unital algebras.

Recall the connected component functor $\Disk_{n,+} \xra{\pi_0} \Fin_\ast$, which is symmetric monoidal.
We denote $\Fin_\ast^{\sf surj}\subset \Fin_\ast$ for the subcategory of based finite sets and surjections among them; this is a symmetric monoidal subcategory.
\begin{definition}\label{def:non-unital}
$\Disk_{n,+}^{\sf surj}$ is the pullback $\infty$-category
\[
\Disk_{n,+}^{\sf surj} := (\Disk_{n,+})_{|\Fin_\ast^{\sf surj}}~.
\]
For $\cV$ a symmetric monoidal $\infty$-category, 
the $\infty$-category of \emph{non-unital} $n$-disk algebras in $\cV$ is 
\[
\Alg_n^{\sf nu}(\cV) := \Fun^\ot_0(\Disk_{n,+}^{\sf surj},\cV)
\]
the $\oo$-category of non-unital symmetric monoidal functors that respect terminal objects.
\end{definition}

\begin{prop}[Proposition 5.4.4.10 of~\cite{HA}]\label{not-augmented}
There is an equivalence of $\infty$-categories
\[
\Ker^{\sf aug}\colon \Alg_n^{\sf aug}(\cV) ~\simeq~ \Alg_n^{\sf nu}(\cV) \colon \uno \oplus (-)
\]
-- the values of the left functor are depicted as $\Ker^{\sf aug}A\colon \underset{I}\bigvee\RR^n_+ \mapsto  \Ker\bigl(A(\underset{I}\bigvee \RR^n_+) \to A(\ast) \simeq \uno\bigr)$.
This equivalence is natural among such $\cV$.  
\end{prop}

\subsection{Support for Theorem~\ref{truncation-quotients} (cardinality layers)}\label{reduced-LKan}

\subsubsection{{\bf Configuration space quotients}}
Here we draw on results from~\cite{aft1},~\cite{aft2}, and~\cite{ZP} to identify the cofiber of $\Disk^{<i}_+(M_\ast) \to \Disk^{\leq i}_+(M_\ast)$.  For this subsection, fix a finite cardinality $i$.

In~\S3 of~\cite{aft1} we construct, for each stratified space $X$, a stratified space $\Ran_{\leq i}(X)$ whose points are finite subsets $S\subset X$ for which the map to connected components $S\to[X]$ is surjective. We show that $\Ran_{\leq i}(-)$ is continuously functorial among conically smooth embeddings among stratified spaces which are surjective on connected components.  
In Lemma~2.20 of~\cite{aft2} we show that the resulting functor
\begin{equation}\label{aft2-EE-to-E}
\Ran_{\leq i}\colon \Disk^{\sf surj, \leq i}(\bsc)_{/X} \xra{~\simeq~} \bsc_{/\Ran_{\leq i}(X)}
\end{equation}
is an equivalence.

To state our next technical result, we introduce notation.
Let $V\hookrightarrow M_\ast$ be an object in $\Disk^{\leq i}_+(M_\ast)$.  
The $\infty$-subcategory
\[
\Disk^{\leq i}_+(M_\ast)_V~\subset~(\Disk(\Bsc)_{/M_\ast})^{V/}
\]
consists of: those objects $V\overset{f}\hookrightarrow V' \hookrightarrow M_\ast$ for which $f$ is injective on path components; those morphisms, which are embeddings $V'\to V''$ under $V$ and over $M_{\ast}$, that are surjective on path components.
We fix a section $[V]\hookrightarrow V$ of the quotient map $V\to [V]$ to its path components.  
The given embedding $V\hookrightarrow M_\ast$ then determines an injection between topological spaces $[V]\hookrightarrow V\hookrightarrow M_\ast$.  
The topological subspace 
\[
\Ran_{\leq i}(M_\ast)_{[V]}~\subset~\Ran_{\leq i}(M_\ast)
\]
consisting those pointed subsets $S_+\subset M_\ast$ that contain $[V]\subset V \hookrightarrow M_\ast$.

\begin{cor}\label{disks-to-ran}
The functor~(\ref{aft2-EE-to-E}) restricts as an equivalence of $\infty$-categories:
\[
\Ran_{\leq i}\colon \Disk^{\leq i}_+(M_\ast) \xra{~\simeq~} \Disk^{\leq 1}_+\bigl(\Ran_{\leq i}(M_\ast)\bigr)~.
\]
More generally, for any fixed embedding $V \hookrightarrow M_\ast$ in $\Disk^{\leq i}_+(M_\ast)$, there is a likewise equivalence
\[
\Ran_{\leq i}\colon \Disk^{\leq i}_+(M_\ast)_V \xra{~\simeq~} \Disk^{\leq 1}_+\bigl(\Ran_{\leq i}(M_\ast)_{[V]}\bigr)~.
\]
\end{cor}
\begin{proof}
Let $B\subset M_\ast$ be a basic neighborhood of $\ast$.  
In~\cite{aft1} it is shown that any conically smooth open embedding from a basic $U\hookrightarrow M_\ast$ whose image contains $\ast$ is canonically based-isotopic to one that factors through a based isomorphism $U\cong B\hookrightarrow M_\ast$.  
We conclude that the projection from the slice
\begin{equation}\label{of-M-over}
\bigl(\Disk(\bsc)_{/M_\ast}\bigr)^{B/} \xra{~\simeq~} \Disk_+(M_\ast)
\end{equation}
is an equivalence of $\infty$-categories.  
Likewise, we conclude that the projection from the slice
\[
\bigl(\Disk^{\leq 1}_+\bigl(\Ran_{\leq i}(M_\ast)\bigr)\bigr)^{\Ran_{\leq i}(B)/}\xra{~\simeq~}\Disk^{\leq 1}_+\bigl(\Ran_{\leq i}(M_\ast)\bigr)
\]
is an equivalences of $\infty$-categories.
The result follows from the equivalence~(\ref{aft2-EE-to-E}).  To conclude the assertion for the case of an embedding $V\hookrightarrow M_\ast$, we again apply Lemma~2.20 of~\cite{aft2} to the case of the stratification $X= (M_\ast, [V]\amalg \ast)$, and then pass to $\oo$-subcategories to conclude the result.
\end{proof}

This identification affords the following essential consequence.

\begin{lemma}\label{surj-final}
The functor
\[
\Disk^{\sf surj}_+(M_\ast)\longrightarrow \Disk_+(M_\ast)
\]
is final.

\end{lemma}

\begin{proof}

Both $\Disk_+^{\sf surj}(M_\ast) \ra\Disk^{\sf surj}(\bsc)_{/M_\ast}$ and $\Disk_+(M_\ast)\ra \Disk(\bsc)_{/M_\ast}$ are final. By the partial two-of-three property of finality, the assertion follows from showing that the functor \[\Disk^{\sf surj}(\bsc)_{/M_\ast} \longrightarrow \Disk(\bsc)_{/M_\ast}\] is final. Writing $M_\ast \cong \bigvee_{i\in[M_\ast]} M_{i,\ast}$ as a wedge over its components, this functor is expressible as a product \[\Disk^{\sf surj}(\bsc)_{/M_\ast}\cong \prod_{i\in[M_\ast]}\Disk^{\sf surj}(\bsc)_{/M_{i,\ast}}  \longrightarrow \prod_{i\in[M_\ast]}\Disk(\bsc)_{/M_{i,\ast}}\cong \Disk(\bsc)_{/M_\ast}~.\] Since a product of final functors is final, we can reduce to the case of a factor, i.e., the case in which $M_\ast$ is irreducible. 

By Quillen's Theorem A, it suffices to show that for each $V \hookrightarrow M_\ast$, the $\oo$-undercategory
\[
\disk^{\sf surj}(\Bsc)_{/M_\ast}\underset{\disk(\Bsc){/M_\ast}}\times\Bigl(\disk(\Bsc)_{/M_\ast}\Bigr)^{V/}
\]
has a contractible classifying space. There exists an adjunction
\[
 \Disk_+(M_\ast)_V
\leftrightarrows
\disk^{\sf surj}(\Bsc)_{/M_\ast}\underset{\disk(\Bsc){/M_\ast}}\times\Bigl(\disk(\Bsc)_{/M_\ast}\Bigr)^{V/}
\]
hence an equivalence between their classifying spaces. By Corollary~\ref{disks-to-ran}, we obtain
\[
\sB\Bigl( \Disk_+(M_\ast)_V
\Bigr)
\simeq
\sB{\sf Exit}\bigl(\ran(M_\ast)_{[V]}\bigr)
\simeq
\ran(M_\ast)_{[V]}~.
\]
The result now follows from the contractibility of this form of the Ran space---to see this, we use a now standard argument from \cite{bd}: the Ran space $\ran(M_\ast)_{[V]}$ carries a natural H-space structure, given by taking unions of subsets, for which the composition of the diagonal and the H-space multiplication is the identity. Consequently, its homotopy groups must be zero.
\end{proof}

The following is an essential result, a description of the layers in the cardinality filtration of the $\oo$-category $\Disk_+(M_\ast)$. 
Recall the zero-pointed manifolds with corners $C^{\neg}_i(M_\ast)$ and $C_i(M_\ast)$ from Lemma~\ref{conf.finitary}, which are based homotopy equivalent to the pointed topological spaces $\conf^{\neg}_i(M_\ast)$ and $\conf_i(M_\ast)$
from Notation~\ref{def.conf}.

\begin{lemma}\label{disk-config}
Let $i$ be a finite cardinality, and let $M_\ast$ be a zero-pointed manifold with $M$ connected.
There is a canonical cofiber sequence among $\infty$-categories
\[
\Disk^{<i}_+(M_\ast)\longrightarrow \Disk^{\leq i}_+(M_\ast) \longrightarrow  \Disk^{\leq 1}_+\bigl(C^{\neg}_i(M_\ast)_{\Sigma_i}\bigr)~.
\]
Likewise, there is a canonical equivalence between pointed $\infty$-categories:
\[
\Disk_+^{=i}(M_\ast)~\simeq~ \Disk_+^{=1}\bigl( C_i(M_\ast)_{\Sigma_i} \bigr)~.
\]

\end{lemma}
\begin{proof}
Consider the diagram among $\infty$-categories
\[
\xymatrix{
\Disk^{<i}_+(M_\ast)  \ar[rr]  \ar[d]^-{\underset{(\rm Lem~\ref{disks-to-ran})}\simeq}
&&
\Disk^{\leq i}_+(M_\ast)  \ar[d]_-{\underset{(\rm Lem~\ref{disks-to-ran})}\simeq}
\\
\Disk^{\leq 1}_+\bigl(\Ran_{<i}(M_\ast)\bigr) \ar[rr]
&&
\Disk^{\leq 1}_+\bigl(\Ran_{\leq i}(M_\ast)\bigr)  ~.
}
\]
The top horizontal functor is the one whose cofiber we seek to identify.
The vertical functors are equivalences by Lemma~\ref{disks-to-ran}, as indicated.  
We are therefore reduced to identifying the cofiber of the bottom horizontal functor.

The inclusion $\Ran_{<i}(M_\ast)\hookrightarrow \Ran_{\leq i}(M_\ast)$ is a proper constructible embedding.
Blowing-up along this proper constructible embedding results in the diagram of stratified spaces and proper constructible bundles among them:
\[
\xymatrix{
\sL  \ar[rr]  \ar[d]
&&
{\sf Bl}_{\Ran_{<i}(M_\ast)}\bigl( \Ran_{\leq i}(M_\ast)\bigr)   \ar[d]
\\
\Ran_{<i}(M_\ast)  \ar[rr]
&&
\Ran_{\leq i}(M_\ast)
}
\]
where here and in what follows, $\sL$ is the link $\sL := {\sf Link}_{\Ran_{<i}(M_\ast)}\bigl( \Ran_{\leq i}(M_\ast)\bigr)$.
Through Lemma~3.7.6 of~\cite{aft1}, this is a pushout diagram among stratified spaces.  
Note that the connectivity assumption on $M$ implies that the $\pi_0$-surjectivity condition in the definition of $\ran(M_\ast)$ is vacuously satisfied. Consequently, the complement of the inclusion $\Ran_{<i}(M_\ast)\hookrightarrow \Ran_{\leq i}(M_\ast)$ is the open subspace $\conf_i(M)_{\Sigma_i}$.
Consider the pullback diagram among stratified spaces:
\[
\xymatrix{
{\sf Bl}_{\Ran_{<i}(M_\ast)}\bigl( \Ran_{\leq i}(M_\ast)\bigr)_{|\conf_i(M)_{\Sigma_i}}  \ar[rr]  \ar[d]_-{\cong}
&&
{\sf Bl}_{\Ran_{<i}(M_\ast)}\bigl( \Ran_{\leq i}(M_\ast)\bigr)   \ar[d]
\\
\conf_i(M)_{\Sigma_i}  \ar[rr]
&&
\Ran_{\leq i}(M_\ast)   ~.
}
\]
The left vertical map is an isomorphism, and the top horizontal map is an open inclusion of the interior of a smooth manifold with corners.  
In particular, there is a collar-neighborhood of the faces:
\[
\xymatrix{
&
\sL  \ar[dl]_-{0}  \ar[dr]
&
\\
\sL  \times [0,1)
\ar@{-->}[rr]
&
&
{\sf Bl}_{\Ran_{<i}(M_\ast)}\bigl( \Ran_{\leq i}(M_\ast)\bigr)~,
}
\]
where the bottom horizontal map is a refinement onto its image, which is open -- here, we regard $[0,1) = \bigl( [0,1)\to \{0<1\} \bigr)$ as a stratified space for which the $0$-stratum is $\{0\}\subset [0,1)$.  
We conclude that the canonical diagram among stratified space
\[
\xymatrix{
&&
\sL\times (0,1)  \ar[rr]  \ar[d]
&&
\conf_i(M)_{\Sigma_i}  \ar[dd]
\\
\sL\times \{0\}  \ar[rr]  \ar[d]
&&
\sL\times [0,1)  \ar[drr]  
&&
\\
\Ran_{<i}(M_\ast)  \ar[rrrr]
&&&&
\Ran_{\leq i}(M_\ast)
}
\]
is a colimit diagram.
This colimit diagram among stratified spaces yields a colimit diagram among enter-path $\infty$-categories:
\[
\xymatrix{
&
{\sf Entr}( \sL )  \ar[r]  \ar[d]_-1
&
{\sf Entr}\bigl(\Ran_{<i}(M_\ast) \bigr)   \ar[dd]
\\
{\sf Entr}( \sL  ) \ar[r]^-0  \ar[d]
&
{\sf Entr}( \sL)   \times [1]  \ar[dr]  
&
\\
{\sf Entr}\bigl(\conf_i(M)_{\Sigma_i} \bigr)  \ar[rr]
&&
{\sf Entr}\bigl(  \Ran_{\leq i}(M_\ast)  \bigr)  ~.
}
\]
There results a cofiber sequence among $\infty$-categories
\[
{\sf Entr}\bigl(  \Ran_{< i}(M_\ast)  \bigr) 
\longrightarrow
{\sf Entr}\bigl(  \Ran_{\leq i}(M_\ast)  \bigr) 
\longrightarrow
{\sf Entr}( \sL  )^{\tl}\underset{ {\sf Entr}( \sL  )} \coprod {\sf Entr}\bigl(\conf_i(M)_{\Sigma_i} \bigr)   .
\]
Now, Proposition~1.2.13 of~\cite{aft1} states that refinements among stratified spaces induce localizations between enter-path $\infty$-categories.
Using this, we reidentify the above cofiber $\infty$-category as the pushout among $\infty$-categories,
\[
\sL^{\tl} \underset{\sL}\coprod \conf_i(M)_{\Sigma_i}   ~,
\]
involving $\infty$-groupoids associated to the underlying topological spaces of the stratified spaces involved.  
By direct inspection of spaces of objects and of morphisms, the canonical functor from this pushout 
\[
\sL^{\tl} \underset{\sL}\coprod \conf_i(M)_{\Sigma_i}  
\xra{~\simeq~}
\Disk_+^{\leq 1}\bigl(C^{\neg}_i(M_\ast)\bigr)
\]
is an equivalence between $\infty$-categories.
With the definitional identification~${\sf Entr}(X):=\Bsc_{/X}$ for each stratified space $X$, we conclude the desired cofiber sequence.

Consider the pointed $\Sigma_i$-topological space which is the $i$-fold product $(M_\ast)^{\times i}$.  
Consider the pointed $\Sigma_i$-equivariant subspace 
\[
B ~:=~ \Bigl\{ \{1,\dots,i\}\xra{c} M_\ast  \mid   c^{-1}(\ast) \neq \emptyset \Bigr\}    ~   \subset   ~   (M_\ast)^{\times i}  ~.
\]
Consider the $\Sigma_i$-equivariant subspace 
\[
D ~:=~ \Bigl\{ \{1,\dots,i\}\xra{c} M_\ast  \mid    c^{-1}(\ast) = \emptyset \text{ and } c \text{ is not injective} \Bigr\}  ~  \subset   ~    (M_\ast)^{\times i}~.
\]
The complement 
\[
C~:=~(M_\ast)^{\times i}\smallsetminus D~  \subset ~  (M_\ast)^{\times i}
\]
is a pointed $\Sigma_i$-equivariant subspace. 
Denote the $\Sigma_i$-equivariant subspace $C^{\partial}:= B \cap C$.

By assumption, the inclusion of the zero-point $\ast \hookrightarrow M_\ast$ admits a neighborhood deformation retraction.
It follows that the inclusion $B \hookrightarrow (M_\ast)^{\times i}$ also admits a $\Sigma_i$-equivariant neighborhood deformation retraction.  
Likewise, the inclusion $C^{\partial} \hookrightarrow C$ admits a $\Sigma_i$-equivariant neighborhood deformation retraction.
Therefore, the inclusion $C^{\partial}_{\Sigma_i} \hookrightarrow C_{\Sigma_i}$ admits a neighborhood deformation retract.  
Furthermore, there is a canonical identification of the complement:
\[
C_{\Sigma_i}\smallsetminus C^{\partial}_{\Sigma_i}
~=~
\conf_i(M)_{\Sigma_i}~.
\]
Using this, inspection of spaces of objects and of morphisms reveals canonical equivalences between $\infty$-categories
\[
\Disk_+^{=i}(M_\ast)
\xla{~\simeq~}
(C^{\partial}_{\Sigma_i})^{\tl} \underset{C^{\partial}_{\Sigma_i}} \coprod  \conf_i(M)_{\Sigma_i}
\xra{~\simeq~}
\Disk_+^{=1}\bigl( C_i(M_\ast) \bigr)~.
\]
\end{proof}

\subsubsection{{\bf Reduced extensions}}\label{sec.reduced}
We explain a couple general maneuvers concerning left Kan extensions in the presence of zero objects. For ${f}:\cK\ra \cK'$ functor among small $\infty$-categories, and $\cV$ a presentable $\infty$-category, there is an adjunction among functor $\infty$-categories
\[
f_{!} \colon \cV^\cK  \rightleftarrows \cV^{\cK'} \colon f^\ast
\]
with the right adjoint $f^\ast$ given by precomposing with $f$, and with the left adjoint $f_{!}$ given by left Kan extension.  In the following, let $\cV$ be a presentable $\infty$-category with a zero object. For $\ast \xra{x} \cK$ a pointed $\oo$-category (i.e., an $\infty$-category with a distinguished object---which need not be zero), then we define the full $\oo$-subcategory
\[
\Fun_0(\cK,\cV)\subset \cV^\cK
\]
as the fiber over the zero object in the sequence $\Fun_0(\cK,\cV) \to \cV^\cK \xra{x^\ast} \cV$.

\begin{lemma}\label{reduced}
Under the hypotheses above, the inclusion $\Fun_0(\cK,\cV) \ra \cV^\cK$ admits a left adjoint
$(-)^{\sf red}\colon \cV^\cK \to \Fun_0(\cK,\cV)$. Further, this left adjoint fits into a cofiber sequence in $\cV^\cK$
\[
x_{!} x^\ast \longrightarrow \id_{\cV^\cK} \longrightarrow (-)^{\sf red}~.
\]

\end{lemma}

\begin{proof}
Because $x^\ast$ preserves colimits, the universal morphism $\coker\bigl(x^\ast x_{!} x^\ast \to x^\ast\bigr)\xra{\simeq}x^\ast\coker\bigl(x_{!} x^\ast \to \id\bigr)$ is an equivalence in $\cV$.  
Because the inclusion $\ast \xra{x} \cK$ is fully faithful, the morphism $x^\ast x_{!} x^\ast \xra{\simeq} x^\ast$ is an equivalence in $\cV$.  
So the values of the cofiber $\coker\bigl(x_{!} x^\ast \to \id\bigr)$ lie in $\Fun_0(\cK,\cV)$.  
Because $0\in \cV$ is zero-pointed, then the value $x_{!}(0) \in \cV^\cK$ is the zero functor.  
The endofunctor $x_{!} x^\ast$ restricts to the zero functor on $\Fun_0(\cK,\cV)$. Consequently, $(-)^{\sf red}$ restricts to the identity functor on $\Fun_0(\cK,\cV)$. 
Finally, the arrow $\id\to (-)^{\sf red}$ witnesses $(-)^{\sf red}$ as a left adjoint as claimed.  
\end{proof}

\begin{lemma}\label{quotient} 
Let $i\colon \cK_0 \to \cK$ be a fully faithful functor among $\infty$-categories, and consider the functor $j\colon \cK \to \cK/\cK_0 :=  \ast \underset{\cK_0}\coprod \cK$ to the cone, regarded as a pointed $\infty$-category.
There is a cofiber sequence in the functor $\infty$-category $\cV^\cK$
\[
i_{!} i^\ast \longrightarrow \mathsf{id}_{\cV^\cK} \longrightarrow j^\ast j_{!}^{\sf red}
\] for any $\cV$ a stable presentable $\infty$-category.  

\end{lemma}

\begin{proof} 
There is a canonical fiber sequence of $\infty$-categories
\[
\Fun_0(\cK/\cK_0,\cV) \xra{~j^\ast~}\cV^\cK \xra{~i^\ast~} \cV^{\cK_0}~.
\]
The inclusion of the constant zero functor $\{0\}\subset\cV^{\cK_0}$ is fully faithful, therefore the functor $j^\ast$ is fully faithful since fully faithfulness is preserved by the formation of pullbacks. We therefore identify $\Fun_0(\cK/\cK_0,\cV)$ with the corresponding full $\oo$-subcategory of $\cV^\cK$ in the following.

Because $i^\ast$ preserves colimits, the universal morphism $\coker\bigl(i^\ast i_{!} i^\ast \to i^\ast\bigr) \xra{\simeq} i^\ast \coker\bigr(i_{!} i^\ast \to \id\bigr)$ is an equivalence in $\cV^{\cK_0}$.  
Because $i\colon \cK_0\to \cK$ is fully faithful, the morphism $i^\ast i_{!} i^\ast \xra{\simeq}i^\ast$ is an equivalence in $\cV^{\cK_0}$. 
It follows that the values of $\coker\bigl(i_{!} i^\ast\to \id\bigr)$ lie in $\Fun_0(\cK/\cK_0,\cV)$.
Because $0\in \cV$ is a zero object, the value $i_{!}(0)$ is the zero functor.  It then follows that $i_{!} i^\ast$ restricts to the zero functor on $\Fun_0(\cK/\cK_0,\cV)$; therefore $\coker\bigl(i_{!} i^\ast \to \id\bigr)$ restricts to the identity functor on $\Fun_0(\cK/\cK_0,\cV)$. 
In summary, the endofunctor $\coker\bigl(i_{!} i^\ast \to \id\bigr)\colon \cV^\cK \to \cV^\cK$ factors through $\Fun_0(\cK/\cK_0,\cV) \xra{j^\ast} \cV^\cK$. Since the factorizing functor $j_!^{\sf red}:\cV^\cK \to \Fun_0(\cK/\cK_0,\cV)$ is the composite of left adjoints, $j_!$ and $(-)^{\sf red}$, therefor $j_!^{\sf red}$ is again left adjoint to $j^\ast$.
\end{proof}

\begin{proof}[Proof of Theorem~\ref{truncation-quotients}]
We explain the cofiber sequence~(\ref{ho-cofiber}).The argument concerning the fiber sequence~(\ref{ho-fiber}) is dual.

From Lemma~\ref{quotient}, with that notation, there is a cofiber sequnce
\[
\tau^{\leq i-1}\int_{M_\ast} A \longrightarrow  \tau^{\leq i}\int_{M_\ast} A \longrightarrow \colim\Bigl(\bigl(\Disk_+^{\leq i}(M_\ast)\bigr)\big/\bigl(\Disk_{+}^{\leq i-1}(M_\ast)\bigr) \xra{j^\ast j^{\sf red}_{!} A } \cV\Bigr)
\]
Through Lemma~\ref{disk-config} the colimit expression is canonically identified as the colimit
\[
\colim\Bigl((\Disk^{\leq 1}_+\bigl(C^\neg_i(M_\ast)\bigr) \xra{~A^{\ot i}~} \cV\Bigr)~.
\] 
Lemma~\ref{conf.finitary} gives the based $\Sigma_i$-equivariant homotopy equivalence $C^{\neg}_i(M_\ast) \simeq \conf^{\neg}_i(M_\ast)$, which evidently lies over $\BO(n)^{\times i}$.  
This identifies this colimit as $\conf_i^{\neg, {\sf fr}}(M_\ast)\underset{\Sigma_i\wr \sO(n)} \bigotimes A^{\ot i}$, the tensor over based modules in spaces.  

\end{proof}

\subsection{Free calculation}\label{sec:free-calculation}
Here we give the calculation of the factorization homology of a free algebra.  
This calculation is a fundamental input to a number of our arguments. To make this calculation, we assume throughout this subsection that the symmetric monoidal structure of $\cV$ distributes over colimits, and we will use the notation $\oplus$ for the coproduct on $\cV$.  For the next result, note that each $\sO(n)$-module $V$ in $\cV$ determines a $\Sigma_i\wr \sO(n)$-module $V^{\ot i}$ in $\cV$.
\begin{theorem}\label{free-calculation}
Let $\cV$ be a $\ot$-cocomplete symmetric monoidal $\infty$-category.
Let $V\in \Mod_{\sO(n)}(\cV_{\uno//\uno})$ be a $\sO(n)$-module in retractive objects over the unit of $\cV$.
There is a canonical identification of the factorization homology of the free augmented algebra on $V$ in terms of configuration spaces of $M$:
\[
\int_{M_\ast} \FF^{\sf aug} V
~
\simeq
~
\bigoplus_{i\geq 0} \Bigl(\conf_i^{\sf fr}(M_\ast) \underset{\Sigma_i\wr \sO(n)}\bigotimes V^{\ot i}\Bigr)~.
\]

\end{theorem}

\begin{proof}
The term on the righthand side depicts a functor $\ZMfld_n \to \cV$, naturally in $V$.   
This functor canonically extends as a symmetric monoidal functor, because of the distribution assumption on the symmetric monoidal structure of $\cV$.  
It is manifest that the restriction to $\Disk_{n,+}$ of the righthand side satisfies the universal property of the free functor $\FF^{\sf aug}$.   
This proves the theorem for the case of $M_\ast = \bigvee_J \RR^n_+$ a finite disjoint union of Euclidean spaces.

We explain the following sequence of equivalences in $\cV$:
\begin{eqnarray}\label{free-calc-lines}
\nonumber
\int_{M_\ast} \FF^{\sf aug} V
&
\underset{(1)}\simeq
&
\colim\Bigl(\Disk_+(M_\ast) \xra{\FF^{\sf aug} V} \cV\Bigr)
\\
\nonumber
&
\underset{(2)}\simeq
&
\bigoplus_{i\geq 0} \underset{U_+\in\Disk_+(M_\ast)}\colim \Bigl(\conf^{\sf fr}_i(U_+) \underset{\Sigma_i\wr \sO(n)}\bigotimes V^{\ot i}\Bigr)
\\\nonumber
&
\underset{(3)}\simeq
&
\bigoplus_{i\geq 0} \underset{U_+\in\Disk_+(M_\ast)}\colim \Bigl(\colim\bigl(\Disk^{=i}_+(U_+) \xra{V^{\ot i}} \cV\bigr) \Bigr)
\\
\nonumber
&
\underset{(4)}\simeq
&
\bigoplus_{i\geq 0}\colim\Bigl(\cX^i \xra{V^{\ot i}} \cV \Bigr)
\\
\nonumber
&
\underset{(5)}\simeq
&
\bigoplus_{i\geq 0} \colim\Bigl(\Disk^{= i}_+(M_\ast) \xra{V^{\ot i}} \cV\Bigr)
\\
\nonumber
&
\underset{(6)}\simeq
&
\bigoplus_{i\geq 0} \colim\Bigl((\Disk^{= 1}_+\bigl(C_i(M_\ast)_{\Sigma_i}\bigr) \xra{V^{\ot i}} \cV\Bigr)
\\
\nonumber
&
\underset{(7)}\simeq
&
\bigoplus_{i\geq 0} \colim\Bigl( \conf_i(M_\ast)_{\Sigma_i} \xra{V^{\ot i}} \cV\Bigr)
\end{eqnarray}
The equivalence (1) is the finality of the functor $\Disk_+(M_\ast) \to \bigl(\Disk_{n,+/M_\ast}\bigr)/\Disk_{n,+}$ (Theorem~\ref{thank-god}).
The equivalence~(2) follows from the first paragraph, combined with the distribution of $\bigoplus$ over sifted colimits (using Theorem~\ref{thank-god}).   
The equivalences (3) and (6) are the identifications 
$\Disk^{=i}_+(-) \simeq \Disk^{=1}_+\bigl(C_i(-)_{\Sigma_i}\bigr)$ 
extracted from
Lemma~\ref{disk-config}.
The equivalence~(7) is a consequence of the based homotopy equivalence $C_i(M_\ast)\simeq \conf_i(M_\ast)$ of Lemma~\ref{conf.finitary}.  

Consider the $\infty$-category of arrows $\Fun\bigl([1],\Disk_+(M_\ast)\bigr)$.  Evaluation at $0$ gives a functor \[{\sf ev}_0\colon \Fun\bigl([1],\Disk_+(M_\ast)\bigr) \longrightarrow \Disk_+(M_\ast)~.\]  
We denote the pullback $\infty$-category 
\[
\cX^i := \Disk^{= i}_+(M_\ast) \underset{\Disk_+(M_\ast)}\times \Fun\bigl([1],\Disk_+(M_\ast)\bigr)~.
\]
There is thus a functor $\cX^i\xra{{\sf ev}_0}\Disk^{= i}_+(M_\ast)$, which is a Cartesian fibration.  
For each object $e\colon \underset{i}\bigsqcup \RR^n\to M$ in $\Disk^{= i}_+(M_\ast)$, the object $(e=e)$ is initial in the fiber $\infty$-category ${\sf ev}_0^{-1}e$.  
It follows that the functor $\cX^i\xra{{\sf ev}_0} \Disk^{= i}_+(M_\ast)$ is final.
This explains the equivalence~(5).
The equivalence $(4)$ is formal, because nested colimits agree with colimits.
\end{proof}

\begin{cor}\label{free-greater-k}
Let $\cV$ be a $\ot$-stable-presentable symmetric monoidal $\infty$-category, and let $V$ be a $\sO(n)$-module in $\cV$. For each finite cardinality $i$, the diagram in $\cV$
\[
\tau^{>i}\int_{M_\ast} \FF^{\sf aug} V~\longrightarrow ~\int_{M_\ast} \FF^{\sf aug} V~\longrightarrow P_i\int_{M_\ast} \FF^{\sf aug}V
\]
can be canonically identified with the diagram
\[
\bigoplus_{j>i} \Bigl(\conf_j^{\sf fr}(M_\ast) \underset{\Sigma_j\wr \sO(n)}\bigotimes V^{\ot j}\Bigr)
~\longrightarrow~
\bigoplus_{j\geq 0} \Bigl(\conf_j^{\sf fr}(M_\ast) \underset{\Sigma_j\wr \sO(n)}\bigotimes V^{\ot j}\Bigr)
~\longrightarrow~
\bigoplus_{j\leq i} \Bigl(\conf_j^{\sf fr}(M_\ast) \underset{\Sigma_j\wr \sO(n)}\bigotimes V^{\ot j}\Bigr)
\]
given by inclusion and projection of summands.

\end{cor}

\begin{proof}

The identification of the left term can be seen by inspecting the definition of $\tau^{>i}\int_{M_\ast}$ and tracing the proof of Theorem~\ref{free-calculation}.
The identification of the right term follows by induction on $i$, for which both the base case and the inductive step are supported by Corollary~\ref{conf-as-homog}, using the equivalence $L\FF V \xra{\simeq} V$ of Lemma~\ref{LF}.
That the arrows are as claimed is manifest.  
\end{proof}

A basic feature of Koszul duality in general is that it sends free algebras to trivial algebras.
We are about to reference the notion of a trivial augmented $n$-\emph{co}algebra, the definition of which is exactly dual to that of trivial algebras~(see Definition~\ref{def:cotangent-space}).  

\begin{lemma}\label{free-trivial}
Let $\cV$ be a $\ot$-presentable symmetric monoidal $\infty$-category.
Let $V$ be a $\sO(n)$-module in $\cV_{\uno//\uno}$ and consider the diagonal $\sO(n)$-module $(\RR^n)^+\ot V$ in $\cV_{\uno//\uno}$.
There is a canonical equivalence of augmented $n$-disk coalgebras in $\cV$
\[
\bBar(\FF^{\sf aug} V) ~ \simeq ~  {\sf t^{aug}_{\cAlg}}\bigl((\RR^n)^+ \ot V\bigr)
\]
from the bar construction of the free augmented algebra to the trivial augmented \emph{co}algebra.

\end{lemma}
\begin{proof}
Let $J$ be a finite set and consider the action
\[
{\sf ZEmb}\bigl((\RR^n)^+,\underset{J}\bigvee (\RR^n)^+\bigr)\bigotimes \int_{(\RR^n)^+} \FF^{\sf aug} V \longrightarrow \int_{\bigvee_{J} (\RR^n)^+} \FF^{\sf aug} V ~\simeq~\Bigl(\int_{(\RR^n)^+}\FF^{\sf aug} V\Bigr)^{\otimes J}~.
\]
Through Theorem~\ref{free-calculation}, using the pigeonhole principle, for each $i_1,\dots,i_j>1$ the restriction of this morphism to $\conf_1\bigl((\RR^n)^+\bigr) \ot V$ followed by the projection to $\underset{j\in J}\bigotimes \conf_{i_j}\bigl((\RR^n)^+\bigr) \underset{\Sigma_{i_j}}\ot V^{\ot i_j}$ is canonically equivalent to the zero morphism.

For each $\epsilon >0$ consider the subspace $\conf_i^\epsilon\bigl((\RR^n)^+\bigr)\subset \conf_i\bigl((\RR^n)^+\bigr)$ consisting of those based maps $f\colon \{1,\dots,i\}_+\to (\RR^n)^+$, whose restriction $f_|\colon f^{-1}\RR^n\to \RR^n$ is injective and $x=y\in \RR^n$ whenever $\lVert f_|(x) - f_|(y)\rVert <\epsilon$ -- the inclusion of this subspace is a based weak homotopy equivalence.  
Flowing the vector field $x\mapsto x$ on $\RR^n$ for infinite time witnesses a deformation retraction of $\conf_i^\epsilon\bigl((\RR^n)^+\bigr)$ onto $\ast$ provided $i>1$.
We conclude that $\conf_i\bigl((\RR^n)^+\bigr)$ is contractible for $i>1$.  
Through Theorem~\ref{free-calculation} we arrive at a canonical identification
\[
\uno\oplus \conf_1\bigl((\RR^n)^+\bigr) \ot V \xra{~\simeq~} \int_{(\RR^n)^+} \FF^{\sf aug} V~.
\]
Combined with the above paragraph, we conclude that $\bBar(\FF^{\sf aug} V)$ is canonically identified as the trivial augmented $n$-disk coalgebra on the $\sO(n)$-module $\conf_1\bigl((\RR^n)^+\bigr) \ot V = \bigl((\RR^n)^+ \ot V\bigr)$. 
\end{proof}

\subsection{Support for Theorem~\ref{convergence} (convergence)}\label{sec.conv}

\begin{lemma}\label{sifted-resolutions} 
Let $\cV$ be a $\ot$-sifted cocomplete symmetric monoidal $\infty$-category. The functors $\Alg_n^{\sf aug}(\cV) \to \cV$
\[
\tau^{>i}\int_{M_\ast}\qquad\text{ and }\qquad \int_{M_\ast}\qquad\text{ and }\qquad P_i\int_{M_\ast}
\]
preserve sifted colimits for any $i$ a finite cardinality.
\end{lemma}

\begin{proof}  
We first prove the statement for $\int_{M_\ast}$.  
Let $A\colon \cJ\to \Alg_n^{\sf aug}(\cV)$ be a diagram of augmented $n$-disk algebras in $\cV$, indexed by a sifted $\infty$-category.
The canonical arrow $\underset{j\in \cJ}\colim \int_{M_\ast} A_j \longrightarrow \int_{M_\ast} \underset{j\in \cJ}\colim A_j$ in $\cV$ is a composition
\[
\colim_{j\in \cJ} \colim_{U_+\in \Disk_+(M_\ast)} A_j(U_+)~ \simeq ~\colim_{U_+\in \Disk_+(M_\ast)} \colim_{j\in \cJ} A_j(U_+) \longrightarrow \colim_{U_+\in \Disk_+(M_\ast)} (\colim_{j\in \cJ} A_j)(U_+) 
\]
where the outer objects are in terms of the defining expression for factorization homology, the left equivalence is through commuting colimits, and the right arrow is a colimit of canonical arrows.  
Because of the hypotheses on $\cV$, each arrow $\underset{j\in \cJ}\colim A_j(U_+) \to (\underset{j\in \cJ}\colim A_j)(U_+)$ is an equivalence if and only if it is for $U$ connected.  
This is the case provided the forgetful functor ${\sf ev}_{\RR^n_+}\colon \Alg_n^{\sf aug}(\cV) \to \cV$ preserves sifted colimits.  
This assertion is Proposition 3.2.3.1 of~\cite{HA}.

Because $P_i$ is a left adjoint, it commutes with sifted colimits.  The statement is thus true for $P_i\int_{M_\ast}$ after the first paragraph.  

The functor $\tau^{>i}\int_{M_\ast}\colon \Alg_n^{\sf aug}(\cV) \to \Fun(\Disk_+^{>i}(M_\ast) , \cV) \xra{\colim} \cV$ is a composition of two functors, the latter of which exists on the image of the first and it commutes with those sifted colimits on which it is defined.
Colimits in the middle $\infty$-category are given objectwise, and so it is enough to show that the restriction $\Alg_n^{\sf aug}(\cV) \to \Fun(\Disk_+^{=i}(M_\ast), \cV)$ preserves sifted colimits for each finite cardinality $i$.  
For $i=1$, this follows from the first paragraph as the case $M_\ast = \RR^n_+$.
For general $i$ this follows because the functor $\bigotimes\colon \cV^{i} \to \cV$ preserves sifted colimits, by assumption.  
\end{proof}

\begin{lemma}[Free resolutions]\label{free-resolution1}
Let $\cV$ be a $\ot$-sifted cocomplete symmetric monoidal $\infty$-category.
Every augmented $n$-disk algebra in $\cV$ is a sifted colimit of free augmented $n$-disk algebras.   
In particular, there are no proper full $\infty$-subcategories of $\Alg_n^{\sf aug}(\cV)$ that contain the image of the free functor $\FF^{\sf aug}\colon \Mod_{\sO(n)}(\cV_{\uno//\uno}) \to \Alg_n^{\sf aug}(\cV)$ and that are closed under the formation of sifted colimits.

\end{lemma}

\begin{proof}
Proposition 3.2.3.1 of~\cite{HA}) gives that the forgetful functor $\Alg_n^{\sf aug}(\cV) \to \cV_{\uno//\uno}$ is conservative and preserves sifted colimits.
The $\infty$-categorical Barr-Beck theorem (Theorem~4.7.4.5 of~\cite{HA}) gives that the adjunction $\Mod_{\sO(n)}(\VV_{\uno//\uno}) \rightleftarrows \Alg_n^{\sf aug}(\cV)$ is monadic.  
Consequently, each augmented $n$-disk algebra $A$ is the geometric realization of its functorial free resolution: $|\FF^{\bullet +1}A|\simeq A$.
\end{proof}

\begin{lemma}\label{tau-and-P} 
Let $A$ be an augmented $n$-disk algebra in $\cV$. The canonical sequence of arrows among functors $\Alg_n^{\sf aug}(\cV) \to \cV$
\[
\tau^{>i}\int_{M_\ast}~  \longrightarrow~ \int_{M_\ast}~ \longrightarrow ~P_i\int_{M_\ast} 
\]
is a cofiber sequence for any $i$ a finite cardinality.

\end{lemma}

\begin{proof} 
Corollary~\ref{free-greater-k} immediately gives the result for the case that $A$ is free.  
Every augmented $n$-disk algebra is a sifted colimit of free augmented $n$-disk algebras.  
Lemma~\ref{sifted-resolutions} states that the left two functors commute with sifted colimits.
Because $P_i$ is a left adjoint, the functor $P_i\int_{M_\ast}$ commutes with sifted colimits.
\end{proof}

\begin{proof}[Proof of Theorem~\ref{convergence}]
Choose such a t-structure as in the statement of the theorem. 
It suffices to show that the kernel of the canonical arrow is the zero object.  
There is a pair of equivalences in $\cV$
\[
\Ker\Bigl(\int_{M_\ast}A\ra P_\infty \int_{M_\ast}A\Bigr)~\underset{(1)}\simeq~\underset{i}\varprojlim\Ker\Bigl(\int_{M_\ast} A \to P_i\int_{M_\ast} A\Bigr)~\underset{(2)}\simeq~ \underset{i}\varprojlim~ \tau^{>i}\int_{M_\ast}A~.
\] 
The equivalence labeled (1) follows because formation of kernels commutes with sequential limits. Equivalence (2) follows from Lemma~\ref{tau-and-P}, which identifies the term of cardinality cofiltration $\tau^{>i}\int_{M_\ast}A$ as the kernel of $\int_{M_\ast}A\ra P_i\int_{M_\ast}A$.
We claim that, under either of two criteria, the object $\tau^{>i} \int_{M_\ast} A$ is $i$-connected with respect to the t-structure.  This claim then implies $\varprojlim ~\tau^{>i}\int_{M_\ast} A$ is infinitely connected under the named criteria; this completes the proof because $\cV$ being cocomplete with respect to the given t-structure implies that only the zero object is infinitely connected.

We now prove the above mentioned claim. Since connectivity is preserved under colimits, it suffices to resolve the algebra $A$ by free $n$-disk algebras. We can thus reduce to showing that $\tau^{> i}\int_{M_\ast}\FF^{\sf aug}V$ is $i$-connected, the case in which $A\simeq \FF^{\sf aug} V$ is the free augmented $n$-disk algebra on a $\sO(n)$-module $V$. By Corollary \ref{free-greater-k}, we have a calculation of this truncation as 
\[
\tau^{> i}\int_{M_\ast}\FF^{\sf aug}V~\simeq~ \bigoplus_{\ell> i} \bigl(\conf^{\fr}_\ell(M_\ast)\underset{\Sigma_\ell\wr\sO(n)}\bigotimes V^{\ot \ell}\bigr)~.
\]
Since the inclusion $\cV_{\geq i}\ra\cV$ is a left adjoint, the essential image is closed under colimits -- see \S1.2.1 of~\cite{HA}. So it is enough to show that $\conf_i^{\sf fr}(M_\ast) \underset{\Sigma_i\wr \sO(n)}\bigotimes V^{\ot i}$ is $i$-connected under either of the named criteria, for all $i$.  
Apply this fact to the functor $\Disk_+^{\leq 1}\bigl(\conf_i(M_\ast)\bigr) \xra{j^\ast j^{\sf red}_{!} V^{\ot i}} \cV$ from the proof of Theorem~\ref{truncation-quotients} (using the notation of~\S\ref{reduced-LKan}) supposing one of the criteria is satisfied: under the first criterion, $V^{\ot i}$ is $i$-connected and $\conf_i(M_\ast)$ is connected; under the second criterion, $\conf_i(M_\ast)$ is $i$-connected and $V^{\ot i}$ is connected. The claim follows since colimits preserve connectivity.
\end{proof}

\subsection{Support for Theorem~\ref{goodwillie-layers} (Goodwillie layers)}\label{sec.good-layers}

For the next result we denote the diagonal functor as ${\sf diag}_i\colon \Mod_{\sO(n)}(\cV) \to \Mod_{\sO(n)}(\cV)^{i}$, where $\cV$ is a $\ot$-cocomplete stable symmetric monoidal $\oo$-category.

\begin{lemma}\label{identify-homogeneous}
There is an equivalence of $\infty$-categories
\[
{\sf Poly}_1\bigl(\Mod_{\sO(n)}(\cV)^{i}_{\Sigma_i},\cV\bigr)\xra{~{}~\simeq~{}~}{\sf Homog}_i\bigl(\Alg_n^{\sf aug}(\cV),\cV\bigr)
\]
for each $i$ a finite cardinality. It assigns to $D$ the functor $D\bigl({\sf diag}_i \circ L(-)\bigr)$.
\end{lemma}

\begin{proof} 
Theorem 6.1.4.7 of~\cite{HA} gives the following identification:
there is a canonical fully faithful functor
\[
{\sf Homog}_i\bigl(\Alg^{\sf aug}_n(\cV), \cV\bigr)\hookrightarrow \Fun\bigl(\stab(\Alg^{\sf aug}_n(\cV))^{i}_{\Sigma_i},\cV\bigr)
\]
whose essential image consists of the $\Sigma_i$-invariant functors that preserve colimits in each variable.
For $\cO$ an operad, there is the general canonical equivalence of $\infty$-categories $\stab\bigl(\Alg^{\sf aug}_\cO(\cV)\bigr)\simeq \m_{\cO(1)}(\cV)$ -- see \S7.3.4 of \cite{HA}.  Through the general equivalence above, the named expression depicts an inverse to this fully faithful functor on its essential image.
\end{proof}

\begin{cor}\label{conf-as-homog}
Let $M_\ast$ be a zero-pointed $n$-manifold, and let $i$ be a finite cardinality.
The functor
\[
\conf_i^{\sf fr}(M_\ast) \underset{\Sigma_i\wr \sO(n)}\bigotimes L(-)^{\ot i} \colon \Alg_n^{\sf aug}(\cV) \longrightarrow \cV
\]
is homogeneous of degree $i$.  

\end{cor}

\begin{proof}
The coend $\conf_i^{\sf fr}(M_\ast)\underset{\Sigma_i\wr \sO(n)}\bigotimes -\colon \Mod_{\sO(n)}(\cV)^{i} \to \cV$ preserves colimits and is $\Sigma_i$-invariant, so we can apply Lemma~\ref{identify-homogeneous}.  
\end{proof}

\begin{proof}[Proof of Theorem~\ref{goodwillie-layers}]
Through Lemma~\ref{identify-homogeneous}, the $i$-homogeneous layer of the Goodwillie cofiltration of the functor $\int_{M_\ast}$ is a symmetric functor of $i$-variables of $\sO(n)$-modules. 
To identify this multi-variable functor we evaluate on free augmented $n$-disk algebras. 
Through Corollary~\ref{free-greater-k}, there is a canonical identification
\[
\Ker\Bigl(P_i\int_{M_\ast}\FF V\to P_{i-1}\int_{M_\ast}\FF V\Bigr)~\simeq~ 
\conf_i^{\sf fr}(M_\ast)\underset{\Sigma_i\wr \sO(n)}\bigotimes V^{\ot i}
\]
functorially in the $\sO(n)$-module $V$.  
This verifies the theorem in this free case because of the canonical equivalence $L\FF V \simeq V$ of Lemma~\ref{LF}.
\end{proof}

\subsection{Support for Theorem~\ref{compare-towers} (comparing cofiltrations)}\label{sec.compare}
In this subsection we fix a $\ot$-stable-presentable symmetric monoidal $\infty$-category $\cV$.

We use the following result, which generalizes Lemma~\ref{free-trivial} away from the case of free algebras, at the level of $\sO(n)$-modules.
\begin{theorem}[Corollary 2.29 of~\cite{cotangent}]\label{L-bar}
There is a canonical equivalence between functors $\Alg_n^{\sf aug}(\cV) \to \Mod_{\sO(n)}(\cV)$
\[
(\RR^n)^+ \ot L(-)~\simeq~ {\sf cKer}^{\sf aug}\bigl(\bBar(-)\bigr)~.
\]

\end{theorem}

\begin{cor}\label{conf-PD}
For each zero-pointed $(ni)$-manifold $P_\ast$, equipped with a $\Sigma_i\wr\sO(n)$-structure, there is a canonical equivalence of functors $\Alg_n^{\sf aug}(\cV) \to \cV$
\[
{\sf Fr}_{P_\ast} \underset{\Sigma_i\wr \sO(n)}\bigotimes L(-)^{\ot i} \xra{~\simeq~} \Map^{\Sigma_i\wr \sO(n)}({\sf Fr}_{P_\ast^\neg} , \bBar(-)^{\ot i}\bigr)~.
\]

\end{cor}
\begin{proof}
Theorem~\ref{L-bar} gives a canonical equivalence among $\Sigma_i\wr \sO(n)$-modules in $\cV$:
\[
(\RR^{ni})^+ \ot L(-)^{\ot i} \simeq \bigl((\RR^n)^+ \ot L(-)\bigr)^{\ot i} \simeq  \bigl({\sf cKer}^{\sf aug}(\bBar(-))\bigr)^{\ot i}~.
\]
The result is a direct corollary of Theorem~\ref{linear-PD}, according to the $\sB\bigl(\Sigma_i\wr \sO(n)\bigr)$-structured version of Remark~\ref{PD-B}.  
\end{proof}

\begin{cor}\label{tau-polonomial}
For each finite cardinality $i$, the functor $\tau^{\leq i}\int_{M_\ast^\neg} \bBar\colon \Alg_n^{\sf aug}(\cV) \to \cV$ is polynomial of degree $i$.  

\end{cor}
\begin{proof}
Through Corollary~\ref{conf-PD}, Corollary~\ref{conf-as-homog} implies $\Map^{\Sigma_i\wr \sO(n)}({\sf Fr}_{P_\ast^\neg} , \bBar(-)^{\ot i}\bigr)$ is homogeneous of degree $i$.  
The result then follows by induction using the fibration sequence of Theorem~\ref{truncation-quotients}.
\end{proof}

\begin{proof}[Proof of Theorem~\ref{compare-towers}]
Let $M_\ast$ be a zero-pointed $n$-manifold with $M$ connected.
Let $i$ be a finite cardinality.
Corollary~\ref{tau-polonomial} asserts that the functor $\tau^{\leq i}\int^{M_\ast^\neg} \bBar \colon \Alg_n^{\sf aug}(\cV) \to \cV$ is polynomial of degree $i$.  
The morphism of cofiltrations $P_\bullet \int_{M_\ast} \to \tau^{\leq \bullet}\int^{M_\ast^\neg} \bBar$ follows through the universal property of the Goodwillie cofiltration.
There results a morphism of $i$-homogeneous layers:
\[
\Ker\Bigl(P_i \int_{M_\ast} \to P_{i-1}\int_{M_\ast}\Bigr) 
~{}~\longrightarrow~{}~
\Ker\Bigl(\tau^{\leq i}\int^{M_\ast^\neg} \bBar \to \tau^{\leq i-1}\int^{M_\ast^\neg} \bBar \Bigr)~.
\]
Through Theorem~\ref{truncation-quotients} and Theorem~\ref{goodwillie-layers}, this morphism is canonically equivalent to the morphism of functors
\begin{equation}\label{conf-comparison}
\conf_i^{\sf fr}(M_\ast)\underset{\Sigma_i\wr \sO(n)}\bigotimes L(-)^{\ot i}  ~{}~\longrightarrow~{}~ \Map^{\Sigma_i\wr \sO(n)}\bigl(\conf_i^{\neg, {\sf fr}}(M_\ast^\neg),\bBar(-)^{\ot i}\bigr) ~.
\end{equation}
This is the arrow of Corollary~\ref{conf-PD} applied to the $\sB\bigl(\Sigma_i\wr \sO(n)\bigr)$-structured zero-pointed manifold $\conf_i(M_\ast)$ of Proposition~\ref{conf-stuff}, and so the arrow is an equivalence.  
\end{proof}

\section{Factorization homology of formal moduli problems}

In this section, we will be concerned with a generalization of factorization homology which allows for a more general coefficient system, a moduli functor of $n$-disk algebras (as in the works on derived algebraic geometry \cite{dag10}, \cite{toenvezzosi}, and \cite{thez}).

\begin{definition}[Linear dual]\label{def:linear-dual}
Let $\cV$ be a symmetric monoidal $\infty$-category with $\cV$ presentable.
There is a functor $\cV^{\op} \to \Psh(\cV)$ given by $c\mapsto \cV(-\ot c, \uno)$, morphisms to the symmetric monoidal unit.  
Should $\cV$ be $\ot$-presentable, this functor canonically factors through the Yoneda embedding as a functor
\[
(-)^\vee \colon \cV^{\op} \to \cV
\]
which we refer to as \emph{linear dual}.  
\end{definition}

We choose to simplify and matters by restricting our generality to a familiar setting.
\begin{convention}[$\Ch_\Bbbk^\ot$]
Henceforward, we fix a field $\Bbbk$ and, unless otherwise stated, work over the background symmetric monoidal $\infty$-category 
$
\Ch_\Bbbk^\ot 
$
of chain complexes over $\Bbbk$ -- its equivalences are quasi-isomorphisms.  
With tensor product it becomes a $\ot$-stable-presentable symmetric monoidal $\infty$-category, and it is endowed with a standard t-structure.  
Our choice to work over a field is for the basic but fundamental property that the duality functor exchanges $i$-connected and $(-i)$-coconnected objects.

\begin{notation}
Let $i$ be an integer.
The full $\oo$-subcategories $\Ch_\Bbbk^{\geq i}\subset \Ch_\Bbbk \supset \Ch_\Bbbk^{\leq i}$ consist of those chain complexes whose non-zero homology has the indicated degree bounds.

We simplify the notation
\[
\Alg_n^{\sf nu} := \Alg_n^{\sf nu}(\Ch_\Bbbk)\underset{\mathrm{Prop}~\ref{not-augmented}}\simeq  \Alg_n^{\sf aug}(\Ch_\Bbbk) =: \Alg_n^{\sf aug}~.
\]
We denote the full $\oo$-subcategories 
$
\Alg_n^{{\sf nu},\geq i} \subset \Alg_n^{\sf nu}\supset \Alg_n^{{\sf nu},\leq i}
$
consisting of those non-unital $n$-disk algebras whose underlying chain complex lies in $\Ch_\Bbbk^{\geq i}$ and $\Ch_\Bbbk^{\leq i}$, respectively. We likewise denote the full $\oo$-subcategories 
$
\Alg_n^{{\sf aug},\geq i} \subset \Alg_n^{\sf aug}\supset \Alg_n^{{\sf aug},\leq i}
$
consisting of those augmented $n$-disk algebras $A$ whose associated non-unital algebra lies in $\Alg_n^{\sf nu,  \geq i}$ and $\Alg_n^{\sf nu, \leq i}$, respectively.

\end{notation}

\end{convention}

\subsection{Formal moduli}
We begin with a few essential notions from derived algebraic geometry of $n$-disk algebras.

\begin{definition}[$\Perf_\Bbbk$]\label{perf-triv}
The full $\oo$-subcategory $\Perf_\Bbbk \subset \Ch_\Bbbk$ consisting of those chain complexes $V$ over $\Bbbk$ for which the $\Bbbk$-module $\underset{q\in \ZZ} \bigoplus \sH_qV$ is finite rank over $\Bbbk$.  
The intersection $\Perf_\Bbbk^{\geq i}:= \Perf_\Bbbk \cap \Ch_\Bbbk^{\geq i}$ consists of those perfect complexes whose homology vanishes below dimension $i$.
\end{definition}

\begin{definition}[${\sf Triv}_n$ and $\Artin_n$]\label{triv-n}
The full $\oo$-subcategory $\Triv_n\subset \Alg_n^{\sf nu,\geq 0}$ is the essential image of $\Mod_{\sO(n)}(\Perf_\Bbbk^{\geq 0})$ under the functor that assigns a complex the associated trivial algebra -- it consists of trivial connective non-unital $n$-disk algebras whose underlying $\sO(n)$-module is a perfect chain complex.
The $\oo$-category $\Artin_n$ of non-unital Artin $n$-disk algebras is the smallest full $\oo$-subcategory of $\Alg_n^{\sf nu,\geq 0}$ that contains $\Triv_n$ and is closed under small extensions. That is:
\begin{itemize}
\item[~]
If $B$ is in $\Artin_n$, $V$ is in $\Triv_n$, and the following diagram
\[
\xymatrix{
A   \ar[d]     \ar[r]
&
B   \ar[d]
\\
0   \ar[r] 
&
V
}
\] 
forms a pullback square in $\Alg_n^{\sf nu, \geq 0}$, then $A$ is in $\Artin_n$.

\end{itemize}
\end{definition}

\begin{remark} One might also justifiably calls these {\it local} non-unital Artin algebras, but for economy we omit this extra adjective. In \cite{dag10}, Lurie uses the terminology {\it small} for an equivalent definition in the case of $n$-disk algebras with trivializations of the tangent bundles, i.e., $\cE_n$-algebras.
\end{remark}

\begin{definition}[Moduli functor]  
The $\infty$-category of formal $n$-disk moduli functors 
\[
\FModuli_n:= \Fun\bigl(\Artin_n, \Spaces\bigr)
\]
is the $\infty$-category of copresheaves on non-unital Artin $n$-disk algebras.  
\end{definition}

\begin{remark}
To obtain a workable geometric theory of formal moduli functors such as that of \cite{dag10}, \cite{toenvezzosi}, or \cite{hinich}, one should restrict to those presheaves that satisfy some gluing condition, such as preserving limits of small extensions (after \cite{schlessinger}). We will not use these conditions, so for simplicity of presentation we omit them.
\end{remark}

\begin{example}[Formal spectrum]\label{def:spf} 
The formal spectrum ${\sf Spf}$ is the composite functor
\[
\bigl(\Alg_n^{ \geq 0, {\sf aug}}\bigr)^{\op} \longrightarrow \Psh(\Alg_n^{{\sf aug},\geq 0}) \longrightarrow \FModuli_n
\]
of Yoneda followed by restriction.
Its values are given by ${\sf Spf}(A) \colon R\mapsto \Alg_n^{\sf nu}(A,R)$, which we refer to as the \emph{formal spectrum of $A$}.  

\end{example}

We choose to conceptually simplify our duality formalism by using linear duals and never consider coalgebras.
\begin{definition}\label{def:D}
The Koszul duality functor is the composite
\[
\DD^n \colon (\Alg_n^{\sf aug})^{\op} \xra{\bBar^{\op}} (\cAlg_n^{\sf aug})^{\op} \xra{~(-)^\vee~} \Alg_n^{\sf aug}
\]
which is $\bBar$ followed by linear dual.  

\end{definition}

\begin{definition}\label{def:MC}
For an Artin $n$-disk algebra $R$ and an augmented $n$-disk algebra $A$ over a field $\Bbbk$, the Maurer--Cartan space \[\MC_{A}(R) :=\Alg_n^{\sf aug}(\DD^nR, A)\] is the space of maps from the Koszul dual of $R$ to $A$. The Maurer--Cartan functor $\MC \colon \Alg_n^{\sf aug} \to \FModuli_n$ is the adjoint of the pairing
\[
\Artin_n \times\Alg_n^{\sf aug} \longrightarrow \Alg_{n}^{\sf aug,\op} \times \Alg_n^{\sf aug} \longrightarrow \Spaces
\]
where the first functor is $\DD^n \times \id$ and the second functor is the mapping space; $\MC$ sends $A$ to the functor $\MC_A$.
\end{definition}

\begin{remark} 
Our definition of the moduli functor $\MC_A$ has the same form as that given by Lurie in~\cite{dag10}, there denoted $\Psi A$.
Our construction of the functor $\DD^n$, on the other hand, is somewhat different: we use the geometry of zero-pointed manifolds, whereas Lurie uses twisted arrow categories. 
To verify that these two constructions agree requires a relationship between twisted arrow categories and zero-pointed manifolds. 
We defer this problem to a separate work.
\end{remark}

The formal spectrum functor ${\sf Spf}$ has a right adjoint.
\begin{definition}[Algebra of functions]\label{def:global-sections}
The \emph{$n$-disk algebra of functions} functor
$
\cO\colon \FModuli_n \ra (\Alg_n^{\sf aug})^{\op}
$
is given as
\[
\cO(X) :=~{}~ \limit_{({\sf Spf}(R) \to X)\in \bigl((\Artin_n^{\op})_{/X}\bigr)^{\op}} R~{}~ =~{}~ \limit\Bigl(\bigl((\Artin_n^{\op})_{/X}\bigr)^{\op} \to \Artin_n \to \Alg_n^{\sf aug}\Bigr)~.
\]
Equivalently, the functor $\cO$ is the right Kan extension of the inclusion $(\Artin_n)^{\op} \hookrightarrow (\Alg_n^{\sf aug})^{\op}$ along the functor $\Spf$:
\[
\xymatrix{
(\Artin_n)^{\op}    \ar[d]_{\Spf}   \ar[rr]
&&
(\Alg_{n}^{\sf aug})^{\op}
\\
\FModuli_n  \ar@{-->}[urr]_-{\cO}
&
&
.
}
\]

\end{definition}

The Maurer--Cartan functor is a lift of the duality functor $\DD^n$.
Namely, there is a canonical equivalence 
\[
\DD^n A ~\simeq ~\cO(\MC_A)
\] 
between the Koszul dual of $A$ and the augmented $n$-disk algebra of functions on the Maurer--Cartan functor of $A$. This can be seen as a special case of Theorem \ref{main} for $\ov{M}=\DD^n$, the closed $n$-disk.

\subsection{Factorization homology with formal moduli coefficients}
We have a notion of factorization homology with coefficients in a formal $n$-disk moduli functor.

\begin{definition}[Factorization homology with formal moduli]\label{def:fact-moduli}
We extend factorization homology to formal moduli functors 
\[
\xymatrix{
\Artin_n   \ar[d]_{\Spf}    \ar[rr]^-{\int}
&&
\Fun(\ZMfld_n, \Ch)
\\
\FModuli_n^{\op}   \ar@{-->}[urr]_-{\int}
}
\]
as the right Kan extension of $\int$ along ${\sf Spf}$, denoted with the same symbol.
Explicitly, the factorization homology of a zero-pointed $n$-manifold $M_\ast$ with coefficients in a formal $n$-disk moduli functor $X$ is
\[
\int_{M_\ast} X :=~{}~ \limit_{({\sf Spf}(R) \to X)\in \bigl((\Artin_n^{\op})_{/X}\bigr)^{\op}} \int_{M_\ast} R~{}~ =~{}~ \limit\Bigl(\bigl((\Artin_n^{\op})_{/X}\bigr)^{\op} \to \Artin_n \xra{\int_{M_\ast}} \Ch_\Bbbk\Bigr)~.
\]

\end{definition}

\begin{remark} 
One can think of $\int_{M_\ast}X$ as $\Gamma(X, \int_{M_\ast}\cO)$, the global sections of the sheaf on $X$ given by calculating the factorization homology of $M_\ast$ with coefficients in the structure sheaf of $X$.
The canonical arrow $\int_{M_\ast} X \to \int_{M_\ast} \cO(X)$ is typically not an equivalence unless $X$ is affine; the arrow can be regarded as a type of completion.
Likewise, unless a formal moduli functor $X$ is affine, factorization homology $\int_{M_\ast} X$ will typically fail to satisfy $\ot$-excision and the Goodwillie tower for factorization homology will typically fail to converge.  
\end{remark}

We now state our main theorem.

\begin{theorem}[Poincar\'e/Koszul duality for formal moduli]\label{main} 
For $\Bbbk$ a field, there is a canonical equivalence of functors $\Alg_n^{\sf aug} \to\Fun\bigl(\ZMfld_n,\Ch_\Bbbk)$,
\[
\Bigl(\int_{(-)}~{}~\Bigr)^\vee ~{}~\simeq~{}~ \int_{(-)^\neg} \MC~.
\]
In particular, for each augmented $n$-disk algebra $A$ in chain complexes over $\Bbbk$, and each $n$-dimensional cobordism $\ov{M}$ with boundary $\partial \ov{M} = \partial_{\sL}\coprod \partial_{\sR}$, there is a canonical equivalence of chain complexes over $\Bbbk$
\[
\Bigl(\int_{\ov{M}\smallsetminus \partial_{\sL}} A\Bigr)^\vee~ \simeq~ \int_{\ov{M}\smallsetminus \partial_{\sR}} \MC_A
\]
between the linear dual of the factorization homology with coefficients in $A$, and the factorization homology with coefficients in the Maurer--Cartan moduli functor of $A$.

\end{theorem}

\begin{remark} 
Let $A$ be an augmented $n$-disk algebra in chain complexes over $\Bbbk$, and let $M$ be a closed $(n-d)$-manifold.
Taking products with $d$-dimensional Euclidean spaces defines augmented $d$-disk algebras $\int_{M_+\wedge-} A\colon \Disk_{d,+}\to \Ch_\Bbbk^\ot$ and $\int_{M_+\wedge -} \MC_A\colon \Disk_{d,+} \to \Ch_\Bbbk^\ot$. Combining factorization homology and our Koszul dual functor $\DD^d$, there is an equivalence
\[
\DD^d\Bigl(\int_{M_+\wedge\RR^d_+}A\Bigr)~\simeq~ \Bigl(\int_{(\RR^d)^+}\int_{M_+\wedge\RR^d_+}A\Bigr)^\vee~.
\]
Our result then specializes to a canonical equivalence of augmented $d$-algebras
\[
\DD^d\Bigl(\int_{M_+\wedge\RR^d_+}A\Bigr)~\simeq~\int_{M_+\wedge\RR^d_+} \MC_A~.
\]

\end{remark}

We now prove our main theorem, making use of the three results which will be developed in the coming subsections.

\begin{proof}[Proof of Theorem \ref{main}]
Let $A$ be an augmented $n$-disk algebra in chain complexes over $\Bbbk$, and let $M_\ast$ be a zero-pointed $n$-manifold.
We establish the following diagram of canonical equivalences in $\Ch_\Bbbk$, each natural in all of their arguments:
\[
\xymatrix{
\Bigl(\displaystyle\int_{M_\ast} A\Bigr)^\vee  \ar[d]_-{\text{Prop}~\ref{resolve}}^\simeq
&&
\displaystyle\int_{M_\ast^\neg} \MC_A  \ar@{=}[d]^-{\text{Def}~\ref{def:fact-moduli}}   \ar@{-->}[ll]_-{\simeq}
\\
\Bigl(\underset{F\in (\fpres_n^{\leq -n})_{/A}}\colim \displaystyle\int_{M_\ast} F\Bigr)^\vee  \ar[d]^-{\simeq}
&&
\underset{R\in \bigl((\Artin_n^{\op})_{/\MC_A}\bigr)^{\op}}\limit  \displaystyle\int_{M_\ast^\neg} R   \ar[d]^-{\text{Thm}~\ref{fpres}}_-\simeq
\\
\underset{F\in \bigl((\fpres_n^{\leq -n})_{/A}\bigr)^{\op}}\limit \Bigl(\displaystyle\int_{M_\ast} F\Bigr)^\vee      
&&
\underset{F\in \bigl((\fpres_n^{\leq -n})_{/A}\bigr)^{\op}}  \limit \displaystyle\int_{M_\ast^\neg} \DD^n F~. \ar[ll]^-{\text{Thm}~\ref{affine}}_-\simeq
}
\]
By Proposition~\ref{resolve}, we can calculate factorization homology with coefficients in $A$ as a colimit over finitely presented $(-n)$-coconnective $n$-disk algebras: $\underset{F\in (\fpres_n^{\leq -n})_{/A}}\colim \int_{M_\ast} F \xra{\simeq} \int_{M_\ast}A$. 
By Theorem~\ref{affine}, for a finitely presented $(-n)$-coconnective $n$-disk algebra $F$, there is natural equivalence $\Bigl(\int_{M_\ast} F\Bigr)^\vee\xla{\simeq} \int_{M_\ast^\neg} \DD^n F$.
By Theorem~\ref{fpres}, Koszul duality restricts to an equivalence $\DD^n: \fpres^{\leq -n}_n \simeq (\Artin_n)^{\op}:\DD^n$ between finitely presented and Artin $n$-disk algebras.
By definition $\Map({\sf Spf}(R),\MC_A) \simeq \Map(\DD^n R, A)$ for $R$ Artin, and so this last equivalence gives an equivalence of slice categories $(\fpres^{\leq -n}_n)_{/A} \simeq (\Artin_n^{\op})_{/\MC_A}$.
\end{proof}

\subsection{Finitely presented coconnective algebras}
We first prove the main result for the special case in which $A$ is a finitely presented $(-n)$-coconnective augmented $n$-disk algebra. The definition of finite presentation is exactly dual to our definition of Artinian.

Recall from Definition~\ref{free} the free algebra functor $\FF\colon \Mod_{\sO(n)}(\Ch_\Bbbk) \to \Alg_{n}^{\sf aug}$.
\begin{notation}\label{def.alg-i}
For $i>0$, we will ongoingly make use of the following full $\infty$-subcategories
\[
\Alg_n^{\sf aug}~\supset~\Alg_n^{\leq -i}~\supset~\Free_n^{{\sf all},\leq -i}~\supset~\Free_n^{\leq -i}~\supset~\Free_n^{{\sf perf},\leq -i}
\]
respectively consisting of: augmented $n$-disk algebras whose augmentation ideal is $(-i)$-coconnective as a $\Bbbk$-module; free augmented $n$-disk algebras on $\sO(n)$-modules whose underlying $\Bbbk$-module is $(-i)$-coconnective; 
free augmented $n$-disk algebras on $\sO(n)$-modules whose underlying $\Bbbk$-module is $(-i)$-coconnective and finite; 
and free augmented $n$-disk algebras on $\sO(n)$-modules which are $(-i)$-coconnective truncations of perfect $\sO(n)$-modules.
\end{notation}
While the outer inclusions in Notation~\ref{def.alg-i} are manifest, the inner inclusion is a particular feature of $n$-disk algebras, which we shall see in the proof of Theorem~\ref{affine}.
We make use of another class of augmented $n$-disk algebras which is dual to Artin algebras.  
\begin{definition}[$\FPres_n^{\leq -n}$]\label{def:fpres}
We denote the intermediate full $\oo$-subcategory
\[
\Free_n^{{\sf perf},\leq -n}\subset\FPres_n^{\leq -n} \subset  \Alg_{n}^{\sf aug, \leq -n}~,
\]
which is the smallest among all such that satisfy the following property:
\begin{itemize}
\item[~]
Let $A$ and $B$ be $(-n)$-coconnective $n$-disk algebras with
\[
\xymatrix{
\FF V \ar[r]   \ar[d]
&
\Bbbk   \ar[d]
\\
A   \ar[r]
&
B
}
\]
a pushout square in $\Alg_n^{\sf aug}$ in which $V$ is a $\sO(n)$-module in $\Perf_\Bbbk^{\leq -n}$; then if $A$ is in $\FPres_n^{\leq -n}$, then $B$ is in $\FPres_n^{\leq -n}$.

\end{itemize}
\end{definition}

This next result verifies that cofibers of augmented algebras can inherit coconnectivity, which assures the existence of many $(-N)$-coconnective algebras that are not free.

\begin{lemma}\label{cocon-cofib}
Fix $N\geq n$.
Let $F\to A$ be a morphism between augmented $n$-disk algebras, each of whose augmentation ideals is $(-N)$-coconnective.
If the induced map on degree-$(-N)$ homology of augmentation ideals $\sH_{-N}\bigl(\Ker(F\to \Bbbk)\bigr) \to \sH_{-N}\bigl( \Ker( A \to \Bbbk) \bigr)$ is injective, then the augmentation ideal $\Ker(A \underset{F} \amalg \Bbbk \to \Bbbk)$ of the pushout among augmented $n$-disk algebras is also $(-N)$-coconnective.  

\end{lemma}

\subsection{Proof of Lemma~\ref{cocon-cofib}}
Through the equivalence $\Alg_n^{\sf aug} \simeq \Alg_n^{\sf nu}$ of Proposition~\ref{not-augmented}, Lemma~\ref{cocon-cofib} is equivalent to the following.
\begin{lemma}\label{0}
Fix $N\geq n$.
Let $F\to A$ be a morphism between non-unital $n$-disk algebras, each of which are $(-N)$-coconnective.
If the map $\sH_{-N}(F) \to \sH_{-N}(A)$ is injective, then the pushout among non-unital $n$-disk algebras $A \underset{F} \amalg 0$ is also $(-N)$-coconnective.  

\end{lemma}

We prove Lemma~\ref{0} using a simple criterion for coconnectivity for a $\Bbbk$-module $Y$: if there exists a filtration
\[
0=: 
Y_{\leq -1}
\longrightarrow
Y_{\leq0}
\longrightarrow
Y_{\leq1}
\longrightarrow
\cdots 
\longrightarrow 
\underset{q\mapsto \oo} \colim ~Y_{\leq q}
\simeq
Y
\]
whose associated graded
\[
\bigoplus_{q\geq 0}\coker\bigl(Y_{\leq q-1}\ra Y_{\leq q}\bigr)
\]
is $(-N)$-coconnective, then $Y$ is $(-N)$-coconnective. For Lemma~\ref{cocon-cofib}, $Y_{\leq q}$ will be the $q$-skeleton ${\sf Sk}_q(X_\dagger)$ of a simplicial $\Bbbk$-module $X_\dagger$ whose realization is $A \underset{F} \amalg 0$. Our proof unfolds as follows:

\begin{itemize}

\item[{\bf\S\ref{sec.a}}] constructs the simplicial $\Bbbk$-module $X_\dagger$ (see (\ref{05})) with $q$-simplices the free algebra on the cokernel of the map of iterated free algebra $\FF^qF\ra \FF^qA$,
\[
X_q \simeq \FF\bigl(\coker( \FF^q F \to \FF^q A)\bigr)~,
\] and identifies its geometric realization: 
$|X_\dagger|\simeq A\underset{F}\amalg 0$. See Lemma~\ref{01}.
\item[{\bf \S\ref{sec.b}}] identifies the $q$-simplices $X_q$ as a direct sum indexed by height-$q$ rooted trees. See Lemma~\ref{l2}.

\item[{\bf \S\ref{sec.c}}] computes the associated graded of the skeletal filtration of $|X_\dagger|$. First, Lemma~\ref{a} gives a splitting
\[
X_q \simeq \sL_q(X_\dagger) \oplus X^{\sf nd}_q
\] as a direct sum of the $q$th latching object, indexed by the degenerate height-$q$ rooted trees, and a module of non-degenerate $q$-simplices, indexed by the non-degenerate height-$q$ rooted trees. Second, Corollary~\ref{lb} gives an explicit expression for the associated graded of the skeletal filtration as the $q$th suspension of the non-degenerate $q$-simplices $X_q^{\sf nd}$.

\item[{\bf \S\ref{sec.d}}] concludes the $(-N-q)$-coconnectivity of the non-degenerate $q$-simplices $X_q^{\sf nd}$, given the assumed $(-N)$-coconnectivity of the cokernel $\coker(F\to A)$. See Lemma~\ref{l7}.
This implies the sought $(-N)$-coconnectivity of $A\underset{F} \amalg 0$.

\end{itemize}

\subsubsection{\bf The pushout as a geometric realization}\label{sec.a}
The main result of this subsection is Lemma~\ref{01}, which
 identifies the underlying $\Bbbk$-module of the relevant pushout among non-unital $n$-disk algebras,  $A\underset{F}\amalg 0$, as the colimit of a bisimplicial $\Bbbk$-module. First, we recall in the following construction that any pushout can alternatively be expressed as a geometric realization of coproducts.

\begin{construction}\label{30}
Let $\cC$ be an $\infty$-category that admits finite colimits; let $\amalg$ denote the coproduct and $\emptyset\in\cC$ be the initial object.
Given a diagram $D\la C\to E$ in $\cC$, there exists a simplicial object
\begin{equation}\label{31}
\bBar_\bullet^{\amalg}( D , C , E )\colon \bDelta^{\op} \longrightarrow \cC
~,\qquad
[p]\mapsto D \amalg C^{\amalg p} \amalg E~,
\end{equation}
with the simplicial structure morphisms uniquely determined by the morphisms $D\la C\to E$, $C\amalg C \to C$, and $\emptyset \ra C$.
The canonical morphism from the $1$-skeleton,
\begin{equation}\label{36}
D\underset{C} \amalg E \xra{~\simeq~} \bigl| \bBar_\bullet^{\amalg}(D,C, E )\bigr| ~,
\end{equation}
is an equivalence in $\cC$.

\end{construction}

\begin{notation}\label{32}
We apply Construction~\ref{30} in two cases.  
\begin{itemize}
\item
Take $\cC=\Mod_{\sO(n)}(\Mod_{\Bbbk})$.
In this case, we denote $\oplus$ for coproduct in $\Mod_{\sO(n)}(\Mod_{\Bbbk})$.
For $E=0$ the zero $\sO(n)$-$\Bbbk$-module, the simplicial object~(\ref{31}) determined by a morphism $U\to V$ between $\sO(n)$-$\Bbbk$-modules is the simplicial $\sO(n)$-$\Bbbk$-module
\begin{equation}\label{33}
\bBar_\bullet^{\oplus}(V,U,0)\colon \bDelta^{\op} 
\longrightarrow
\Mod_{\sO(n)}(\Mod_{\Bbbk})
~,\qquad
[p]\mapsto V \oplus U^{\oplus p}~.
\end{equation}

\item
Take $\cC=\Alg_n^{\sf nu}$, where $\amalg$ is the coproduct in $\Alg_n^{\sf nu}$.
For $E=0$ the zero non-unital $n$-disk algebra, the simplicial object~(\ref{31}) determined by a morphism $C \to B$ is the simplicial non-unital $n$-disk algebra
\begin{equation}\label{34}
\bBar_\bullet^{\amalg}(B,C,0)\colon \bDelta^{\op} 
\longrightarrow
\Alg_n^{\sf nu}
~,\qquad
[p]\mapsto B \amalg C^{\amalg p}~.
\end{equation}
We use the same notation for the simplicial $\Bbbk$-module which is the composite functor:
\begin{equation}\label{35}
\bBar_\bullet^{\amalg}(B,C,0)\colon \bDelta^{\op} 
\xra{~\bBar_\bullet^{\amalg}(B,C,0)~}
\Alg_n^{\sf nu}
\xra{~\rm forget~}
\Mod_{\sO(n)}(\Mod_{\Bbbk})~.
\end{equation}

\end{itemize}

\end{notation}
Consider the free-forgetful adjunction
\begin{equation}\label{1}
\FF\colon \Mod_{\sO(n)}(\Mod_{\Bbbk})
~\rightleftarrows~
\Alg_n^{\sf nu}~.
\end{equation}
This adjunction determines a functor
\[
\FF^{\dagger+1}\colon \Alg_n^{\sf nu} \longrightarrow \Fun(\bDelta^{\op}, \Mod_{\sO(n)}(\Mod_{\Bbbk}))
~,\qquad
B  \mapsto 
\Bigl( [q]\mapsto \FF^{ q+1} B \Bigr)~,
\]
whose value on a non-unital $n$-disk algebra $B$ is the $(q+1)$-fold composition of the composite endo-functor
\[
\FF\colon \Mod_{\sO(n)}(\Mod_{\Bbbk}) \xra{~\FF~} \Alg_n^{\sf nu}\xra{~\rm forget~} \Mod_{\sO(n)}(\Mod_{\Bbbk}) ~,
\]
applied to the underlying $\sO(n)$-$\Bbbk$-modules of $B$. 
Therefore, each simplicial $\sO(n)$-$\Bbbk$-modules determines a bisimplicial $\sO(n)$-$\Bbbk$-modules.
In particular, the simplicial $\sO(n)$-$\Bbbk$-modules~(\ref{35}) applied to the case of the given morphism $F\to A$ results in the bisimplicial $\sO(n)$-$\Bbbk$-modules
\begin{equation}\label{02}
\bBar_\bullet^\amalg ( \FF^{\dagger+1} A , \FF^{\dagger+1} F , 0 ) \colon 
\bDelta^{\op}\times \bDelta^{\op}
\longrightarrow
\Mod_{\sO(n)}(\Mod_{\Bbbk})
~,\qquad
[p],[q]~\mapsto~
\bBar_p^{\amalg}(\FF^{q+1} A , \FF^{q+1} F , 0 )~.
\end{equation}
We will denote
\begin{equation}\label{05}
X_\dagger 
:=
\underset{\bullet\in \bDelta}\colim\bBar_\bullet^{\amalg}(\FF^{\dagger+1} A ,\FF^{\dagger+1} F,0)
\colon  \bDelta^{\op} \longrightarrow \Mod_{\sO(n)}(\Mod_{\Bbbk})
\xra{~\rm forget~}
\Mod_{\Bbbk}
\end{equation}
for the simplicial $\Bbbk$-module which is the geometric realization along the $(\bullet)$-simplicial coordinate of the underlying bisimplicial $\Bbbk$-module of~(\ref{02}); that is, $X_\dagger$ is the simplicial $\Bbbk$-module with $q$-simplices $X_q=  \bigl| \bBar_\bullet^{\amalg}(\FF^{q +1} A ,\FF^{q +1} F,0)\bigr|$.

\begin{lemma}\label{01}
The following statements concerning the simplicial $\sO(n)$-$\Bbbk$-modules~(\ref{05}) are true.
\begin{enumerate}

\item
There is a canonical equivalence between $\Bbbk$-modules,
\[
A\underset{F} \amalg 0
~\simeq~
\bigl|  X_\dagger \bigr|~,
\]
between the underlying $\Bbbk$-modules of the pushout among non-unital $n$-disk algebras and the underlying $\Bbbk$-module of the geometric realization of the bisimplicial $\sO(n)$-$\Bbbk$-modules~(\ref{02}).

\item
For each $q\geq 0$, there is a canonical identification between $\Bbbk$-modules,
\begin{equation}\label{90}
X_q
~\simeq~
\FF\bigl(  \coker( \FF^q F \to \FF^q A )\bigr)~,
\end{equation}
between the value of $X_\dagger$ on $[q]$ and the underlying $\Bbbk$-module of the free non-unital $n$-disk algebra on the cokernel.

\end{enumerate}

\end{lemma}

\begin{proof}

We prove statement~(1) by establishing the following string of equivalences among $\Bbbk$-modules.  
\begin{eqnarray}
\nonumber
A \underset{F} \amalg 0  
&
\underset{\rm (a)} {\xra{~\simeq~}}
&
\bigl| \bBar_\bullet^{\amalg}(A,F,0) \bigr| 
\\
\nonumber
&
\underset{\rm (b)} {\xla{~\simeq~}}
&
\Bigl| \bBar_\bullet^{\amalg}\bigl( | \FF^{\dagger+1}A| , |\FF^{\dagger+1} F| , 0 \bigr) \Bigr|
\\
\nonumber
&
\underset{\rm (c)} {\xla{~\simeq~}}
&
\underset{\bullet\in\bDelta}\colim \ \underset{\dagger \in \bDelta}\colim \bBar_\bullet^{\amalg}\bigl( \FF^{\dagger+1}A ,  \FF^{\dagger+1} F , 0 \bigr)
\\
\nonumber
&
\underset{\rm (d)} {~\simeq~}
&
\underset{\dagger\in\bDelta}\colim \ \underset{\bullet \in \bDelta}\colim \bBar_\bullet^{\amalg}\bigl( \FF^{\dagger+1}A ,  \FF^{\dagger+1} F , 0 \bigr) =: |X_\dagger|
~ .
\end{eqnarray}
The given morphism $F\to A$ between non-unital $n$-disk algebras determines the simplicial non-unital $n$-disk algebra $\bBar_\bullet^{\amalg}(A,F,0)$, as in~(\ref{34}).  
The equivalence~(\ref{36}) specializes to an equivalence between non-unital $n$-disk algebras: 
\begin{equation}\label{37}
A\underset{F} \amalg 0 \xra{~\simeq~} \bigl| \bBar_\bullet^{\amalg}(A,F,0)\bigr|~.
\end{equation}

Proposition~3.2.3.1 of~\cite{HA} gives that each of the forgetful functors $\Alg_n^{\sf nu} \to \Mod_{\sO(n)}(\Mod_{\Bbbk}) \to \Mod_{\Bbbk}$ preserves geometric realizations.  
This establishes~(a), in which $\bBar_\bullet^{\amalg}(A,F,0)$ is taken to be the simplicial $\Bbbk$-module of~(\ref{35}).

Recall that any non-unital $n$-disk algebra $B$ is the geometric realization of its \emph{functorial free resolution}, from Lemma \ref{free-resolution1}:
\begin{equation}\label{2}
\bigl|  \FF^{\dagger+1} B \bigr|
\xra{~\simeq~}
B~.
\end{equation}
By construction, the identification~(\ref{2}) is functorial in the argument $B$.
Forgetting to underlying $\Bbbk$-modules establishes~(b).

For each $[p]\in \bDelta^{\op}$, there is a canonical morphism
\[
\bigl| \FF^{\dagger+1}A  \amalg  (\FF^{\dagger+1} F)^{\amalg p} \bigr| 
\longrightarrow
| \FF^{\dagger+1}A|  \amalg  \bigl(|\FF^{\dagger+1} F| \bigr)^{\amalg p}
\]
between non-unital $n$-disk algebras.  
This morphism is functorial in $[p]\in \bDelta^{\op}$, which establishes the arrow of~(c).  
Using that the coproduct $\amalg \colon \Alg_n^{\sf nu}\times \Alg_n^{\sf nu} \to \Alg_n^{\sf nu}$ distributes over geometric realizations, this canonical morphism is, in fact, an equivalence.  
This establishes~(c).

The identification~(d) is the fact that colimits commute with one another.  
This proves~(1).

We now prove statement~(2).
Let $U \to V$ be a morphism between $\sO(n)$-$\Bbbk$-modules.
We establish the following string of equivalences between $\sO(n)$-$\Bbbk$-modules.
\begin{eqnarray}
\nonumber
\bigl| \bBar_\bullet^{\amalg}(\FF V,\FF U,0)\bigr|
&
\underset{\rm (a)}{~\simeq~}
&
\Bigl| \FF\bigl( \bBar_\bullet^{\oplus}(V,U,0)\bigr) \Bigr|
\\
\nonumber
&
\underset{\rm (b)}{~\simeq~}
&
\FF \Bigl(   \bigl| \bBar_\bullet^{\oplus}(V,U,0)\bigr| \Bigr)
\\
\nonumber
&
\underset{\rm (c)}{~\simeq~}
&
\FF \bigl(   \coker( U \to V )  \bigr)~.
\end{eqnarray}
Inspecting Construction~\ref{30}, there is a canonical morphism between simplicial $\sO(n)$-$\Bbbk$-modules:
\begin{equation}\label{42}
\bBar_\bullet^{\amalg}(\FF V,\FF U,0)
\longrightarrow
\FF\bigl( \bBar_\bullet^{\oplus}( V , U , 0 ) \bigr)~.
\end{equation}
Being a left adjoint between cocomplete $\infty$-categories, the free functor $\FF\colon \Mod_{\sO(n)}(\Mod_{\Bbbk}) \to \Alg_n^{\sf nu}$ carries direct sums to coproducts: $\FF\colon \oplus \mapsto \amalg$.  
Upon further inspection of Construction~\ref{30}, it follows from this that the canonical morphism~(\ref{42}) is an equivalence between simplicial $\sO(n)$-$\Bbbk$-modules.
Forgetting to underlying $\Bbbk$-modules establishes~(a).

From the universal property of geometric realizations as a colimit, there is a canonical morphism between $\sO(n)$-$\Bbbk$-modules
\begin{equation}\label{43}
\Bigl| \FF\bigl( \bBar_\bullet^{\oplus}(V,U,0)\bigr) \Bigr|
\longrightarrow
\FF \Bigl(   \bigl| \bBar_\bullet^{\oplus}(V,U,0)\bigr| \Bigr)~.
\end{equation}
Being a left adjoint between cocomplete $\infty$-categories, the free functor $\FF\colon \Mod_{\sO(n)}(\Mod_{\Bbbk}) \to \Alg_n^{\sf nu}$ carries geometric realizations to geometric realizations.
From Proposition~3.2.3.1 of~\cite{HA}, each of the forgetful functors $\Alg_n^{\sf nu} \to \Mod_{\sO(n)}(\Mod_{\Bbbk})\xra{~\rm forget~} \Mod_{\Bbbk}$ carries geometric realizations to geometric realizations as well.
We conclude that the morphism~(\ref{43}) is an equivalence between $\Bbbk$-modules.
This establishes~(b).

The identification~(c) follows from the general identification~(\ref{36}) in Construction~\ref{30}.
Statement~(2) follows from the above string of canonical equivalences, as the case that the morphism $(U\to V)$ is the induced morphism $(\FF^q F \to \FF^q A)$.  
\end{proof}

\subsubsection{\bf Explicating $X_\dagger$}\label{sec.b}
We now give a presentation of each $\Bbbk$-module $X_q$ of $q$-simplices, together with the face and degeneracy morphisms of $X_\dagger$.  
\begin{definition} Fix $q\geq 0$.
A \emph{height-$q$ rooted tree} is a functor from the opposite of the poset $[q]=\{0<\dots<q\}$ to the category of non-empty finite sets and surjections among them,
\[
T\colon [q]^{\op} \longrightarrow \Fin^{\sf surj}_{\neq \emptyset}
~,\qquad
i\mapsto T_i~,
\qquad\text{ and }\qquad
(i<j) \mapsto (T_j \xra{T_{i<j}} T_i)~,
\]
whose value on $0\in [q]$ is terminal: $T_0 = \ast$.  
The category of \emph{height-$q$ rooted trees} is the full subcategory ${\sf Tree}_q~\subset~\Fun\bigl([q]^{\op},\Fin^{\sf surj}_{\neq \emptyset}\bigr)$
consisting of the height-$q$ rooted trees. $[T]$ is the isomorphism-type of $T$.  

\end{definition}
Observe that these categories of rooted trees assemble as a functor
\[
{\sf Tree}_{\bullet+1}\colon \bDelta^{\op} \longrightarrow \Cat
~,\qquad
[q]\mapsto {\sf Tree}_{q+1} \subset \Fun\bigl(( [q]^{\tl} )^{\op} , \Fin^{\sf surj}_{\neq \emptyset}\bigr)~,
\]
using the isomorphism 
$[q]^{\tl} \cong [q+1]$ in $\bDelta$.  
In particular, each $\sigma\colon [p]\to [q]$ in $\bDelta$ determines a functor
\begin{equation}\label{e9}
{\sf Tree}_{q+1} 
\longrightarrow
{\sf Tree}_{p+1}
~,\qquad
T\mapsto \sigma^\ast T := \Bigl(  ([p]^{\tl})^{\op} 
\xra{~(\sigma^{\tl})^{\op}~}
([q]^{\tl})^{\op} \xra{T} \Fin^{\sf surj}_{\neq \emptyset} \Bigr)~.
\end{equation}

Let $T$ be a height-$q$ rooted tree. 
We now consider the finite group $\Aut(T)$ of automorphisms of $T$ in the category ${\sf Tree}_{q}$ of height-$q$ rooted trees.  
For each $0\leq i \leq q$, since the morphism $T_{i\leq q}\colon T_q \to T_i$ is surjective, the evident homomorphism,
\begin{equation}\label{e9'}
{\sf ev}_q\colon \Aut(T) 
\longrightarrow
\Sigma_{T_q}~,
\end{equation}
is injective.  
The canonical inclusion $\iota_{\leq i}\colon \{0<\dots<i\} \hookrightarrow [q]$ in $\bDelta$ determines the height-$i$ rooted tree $T_{\leq i}:= \iota^\ast_{\leq i} T$, as well as a homomorphism
\begin{equation}\label{e8}
\Aut(T) \longrightarrow \Aut(T_{\leq i})~.
\end{equation}
Also, each $t\in T_i$ determines the height-$(q-i)$ rooted tree $T_{/t}\colon \{i<i+1<\dots<q\}^{\op} 
\ra
\Fin^{\sf surj}_{\neq \emptyset}$  sending $j\mapsto T_{i\leq j}^{-1}(t)$.
Notice a canonical identification of the kernel of the homomorphism~(\ref{e8}): $\Ker(\ref{e8})~\cong~\underset{t\in T_i} \prod \Aut(T_{/t})$.
Lastly, the homomorphism~(\ref{e8}) admits a section defined on its image:
\begin{equation}\label{e4}
\Aut(T)
~\cong~
\Bigl(\underset{t\in T_i} \prod \Aut(T_{/t})\Bigr) \rtimes {\sf Im}(\ref{e8}) 
~.
\end{equation}
With this as the inductive step, and the injection~(\ref{e9'}) as the base case, there results a description of $\Aut(T)$ as an iterated semi-direct product of subgroups of symmetric groups.

Taking sets of connected components defines a functor
\begin{equation}\label{e30}
\cD_n~:=~\Bigl(  [-]\colon \Disk_{n,+}\longrightarrow
\Fin_\ast  \Bigr)~.
\end{equation}
This functor is an $\infty$-operad in the sense of~\S2.1 of~\cite{HA}, which we denote simply as $\cD_n$.
The symmetric monoidal envelope of this $\infty$-operad ${\sf Env}(\cD_n)\simeq \Disk_n$ is the symmetric monoidal $\infty$-category of $n$-disks.
There is a canonical equivalence between $\infty$-categories over $\Fin$:
\[
\Disk_n~\simeq~ (\Disk_{n,+})_{|\Fin}~,
\]
where the righthand term is the base change of~(\ref{e30}) along the monomorphism $\Fin \xra{(-)_+}\Fin_\ast$ given by adjoining a disjoint base point. 
For each rooted tree $T$, we denote the space of lifts to $\Disk_n$ of the composite functor $T\colon [q]^{\op}\to \Fin^{\sf surj}_{\neq \emptyset} \hookrightarrow \Fin$ as
\[
\cD_n(T)
~:=~
\Map_{/\Fin}\bigl([q]^{\op} , \Disk_n \bigr)~.
\]
Inspecting the functor~(\ref{e30}), for each height-$q$ rooted tree $T$, there is a canonical identification between $\Aut(T)$-spaces:
\begin{eqnarray}
\nonumber
\cD_n(T)
&
~\simeq~
&
\Bigl( \conf^{\sf fr}_{T_{0}}(\RR^n) \Bigr)
\underset{\sO(n)^{T_0}} \times
\Bigl( \prod_{t\in T_{0}} \conf^{\sf fr}_{T_{0<1}^{-1}(t)}(\RR^n) \Bigr)
\underset{\sO(n)^{T_1}} \times
\Bigl( \prod_{t\in T_1} \conf^{\sf fr}_{T_{1<2}^{-1}(t)}(\RR^n) \Bigr)
\underset{\sO(n)^{T_2}} \times
\\
\label{08}
&
&
\dots
\underset{\sO(n)^{T_{q-1}}} \times
\Bigl( \prod_{t\in T_{q-1}} \conf^{\sf fr}_{T_{q-1<q}^{-1}(t)}(\RR^n) \Bigr)
\underset{\sO(n)^{T_{q}}} \times  \ast 
~,
\end{eqnarray}
where, for $I$ a finite set, there is a natural equivalence 
\begin{equation}\label{e33}
\conf^{\sf fr}_I(\RR^n)~\simeq~ \sO(n)^I \times \conf_I(\RR^n)~.
\end{equation}
In particular, projecting onto the $T_0$- and $\ast$-factors defines a morphism between $\Sigma_{T_q}$-spaces
\[
\cD_n(T)
\longrightarrow
\BO(n)
\times
\BO(n)^{T_q}~.  
\]
We denote the fiber over the $\Sigma_{T_q}$-invariant point $\bigl(  \RR^n ,  (\RR^n)_{t\in T_q}\bigr) \in \BO(n)  \times \BO(n)^{T_q}$ as
\begin{equation}\label{e31}
\cD_n^{\sf fr}(T)~,
\end{equation}
and regard it as a $\Sigma_{T_q}\wr \sO(n)$-module in $\Mod_{\sO(n)}(\Spaces)$.  
In particular, tensoring defines a colimit preserving functor between $\infty$-categories
\begin{equation}\label{e32}
\Mod_{\Sigma_{T_q}\wr \sO(n)}(\Mod_{\Bbbk})
\longrightarrow
\Mod_{\sO(n)}(\Mod_{\Bbbk})
~,\qquad
U\mapsto
\cD_n^{\fr}(T) \underset{\Sigma_{T_q}\wr \sO(n)} \otimes U~,
\end{equation}
to that of $\sO(n)$-$\Bbbk$-modules.
Via the homomorphism~(\ref{e9'}), we define the group
\[
\Aut(T)\wr \sO(n)~\subset~ \Sigma_{T_q}  \ltimes \sO(n)^{T_q}  ~.
\]

We now use the preceding notions to describe the iterated free functor.
\begin{lemma}\label{l1}
Fix $q\geq 0$.
For each $\sO(n)$-$\Bbbk$-module $W$, there is a canonical identification of the $q$-fold composition of the free-forgetful functor applied to $W$ as the $\sO(n)$-$\Bbbk$-module
\begin{equation}\label{e3}
\FF^q(W)
~\simeq~
\underset{[T, \ \textrm{\rm ht-}q]}\bigoplus~
\cD^{\sf fr}_n(T) \underset{\Sigma_{T_q}\wr \sO(n)} \bigotimes W^{\otimes T_q}~,
\end{equation}
in which the direct sum is indexed by the set of isomorphism classes of height-$q$ rooted trees.  
\end{lemma}

\begin{proof}
We induct on $q$.
For $q=0$, the assertion is automatic. 
For $q=1$, there is a standard identification among $\sO(n)$-$\Bbbk$-modules
\begin{eqnarray}
\nonumber
\FF(W)
&
~\simeq~
&
\underset{k>0} \bigoplus \conf^{\fr}_k(\RR^n)\underset{\Sigma_k \wr \sO(n)}\otimes W^{\otimes k}
\\
\label{e1}
&
~\simeq~
&
\underset{k>0} \bigoplus ~ \cD^{\fr}_n(k)\underset{\Sigma_k}\otimes W^{\otimes k}
\\
\label{e2}
&
~\simeq~
&
\underset{[T, \text{ ht-}1 ]} \bigoplus \cD^{\fr}_n(T)\underset{\Aut(T)\wr \sO(n)}\otimes W^{\otimes T_1}    ~.
\end{eqnarray}
Here, the second identification is simply the identification of the $\Sigma_k\wr \sO(n)$-space $\cD^{\fr}_n(k)$ of multi-morphisms from $(\RR^n)_{i=1}^k$ to $\RR^n$ in the $\infty$-operad
$\cD_n$; the third identification is definitional.  

We now prove the inductive case: assume the statement is true for $q-1$, where $q$ is strictly positive.  
The lemma is proved upon establishing the following string of identifications among $\sO(n)$-$\Bbbk$-modules.
\begin{eqnarray}
\nonumber
\FF^q(W)
&
\underset{\rm (a)}{~\simeq~}
&
\FF\bigl( \FF^{q-1}(W)\bigr)
\\
\nonumber
&
\underset{\rm (b)}{~\simeq~}
&
\underset{k>0} \bigoplus ~\cD^{\fr}_n(k)\underset{\Sigma_k\wr \sO(n)}\otimes \bigl( \FF^{q-1}(W)\bigr)^{\otimes k}
\\
\nonumber
&
\underset{\rm (c)}{~\simeq~}
&
\underset{k>0} \bigoplus ~\cD^{\fr}_n(k)\underset{\Sigma_k\wr \sO(n)}\otimes \Bigl( \underset{[S, \text{ ht-}(q-1)]}\bigoplus~
\cD^{\fr}_n(S) \underset{\Aut(S)\wr \sO(n)} \bigotimes W^{\otimes S_{q-1}} \Bigr)^{\otimes k}
\\
\nonumber
&
\underset{\rm (d)}{~\simeq~}
&
\underset{k>0} \bigoplus 
\Bigl(    \underset{ \bigl\{ ( [S_1],\dots,[S_k] )\mid ~S_i \text{ ht-} (q-1) \bigr\} } \bigoplus  
~\Bigl(\cD^{\fr}_n(k)\times \prod_{i=1}^k \cD^{\fr}_n(S_i)\Bigr)  
\underset{\prod_{i=1}^k \Aut(S_i)\wr \sO(n) }\bigotimes 
\Bigl(  \bigotimes_{i=1}^k  W^{\otimes (S_i)_{q-1}} \Bigr)  \Bigr)_{\Sigma_k\wr \sO(n)}
\\
\nonumber
&
\underset{\rm (e)}{~\simeq~}
&
\underset{[T, \text{ ht-}q]}\bigoplus~
\cD^{\fr}_n(T) \underset{\Aut(T)\wr \sO(n)} \bigotimes W^{\otimes T_q}
\end{eqnarray}
The identification~(a) is definitional.
The identification~(b) is the case $q=1$, established above, with $W$ there replaced by $\FF^{q-1}(W)$.  
The identification~(c) is the inductive assumption.  
The identification~(d) is distributivity of tensor products among $\Bbbk$-modules over direct sums thereof.
The identification~(e) is a consequence of the following construction.
\begin{itemize}
\item[]
Given $k>0$ and height-$(q-1)$ rooted trees $S_1,\dots,S_k$, construct the height-$q$ rooted tree
\[
T\colon [q]^{\op} \longrightarrow  \Fin^{\sf surj}_{\neq\emptyset}
~,\qquad
i\mapsto \coprod_{i=1}^k S_{i-1}
\qquad \text{ for $i>0$~, 
\qquad
and } 0 \mapsto \ast~.
\]
\end{itemize}
Inspecting this construction, the $\Aut(T)$-equivariant identification~(\ref{08}) gives the $\prod_{i=1}^k \Aut(S_i)$-equivariant identification $\cD^{\fr}_n(k) \underset{\sO(n)^{k}}\times \prod_{i=1}^k \cD^{\fr}_n(S_i)  \simeq \cD^{\fr}_n(T)$.  
There is an evident $\prod_{i=1}^k \Aut(S_i)\wr \sO(n)$-equivariant identification $W^{\ot T_q} \simeq \bigotimes_{i=1}^k W^{\ot (S_i)_{q-1}}$.  
Through these equivariant identifications, the identification~(e) follows by reindexing the sum, then using the semi-direct product description~(\ref{e4}) of $\Aut(T)$.  
\end{proof}

Lemma~\ref{l1} facilitates the following explicit description of the $q$-simplices of $X_\dagger$.  
\begin{lemma}\label{l2}
For each $q\geq 0$, there is a canonical identfication between $\Bbbk$-modules:
\begin{equation}\label{88}
X_q
~\simeq~
\underset{[T, {\rm ht\text{-}} (q+1)]}\bigoplus~
\cD^{\fr}_n(T) \underset{\Aut(T)\wr \sO(n)} \bigotimes \Bigl( \underset{t\in T_1}\bigotimes \coker \bigl( F^{\otimes T_{1 \leq q+1}^{-1} (t)} \to A^{\otimes T_{1 \leq q+1}^{-1} (t)}  \bigr) \Bigr)~.
\end{equation}

\end{lemma}

\begin{proof}
We prove the statement by establishing the following string of equivalences among $\Bbbk$-modules:
\begin{eqnarray}
\nonumber
X_q
&
\underset{\rm (a)}{~\simeq~}
&
\FF\Bigl(\coker\bigl(\FF^q(F) \to \FF^q(A) \bigr) \Bigr)
\\
\nonumber
&
\underset{\rm (b)}{~\simeq~}
&
\FF\Bigl(
\coker\Bigl( \underset{[S, \text{ ht-}q]}\bigoplus
\cD^{\fr}_n(S) \underset{\Aut(S)\wr \sO(n)} \bigotimes F^{\otimes S_q} 
\longrightarrow
\underset{[S, \text{ ht-}q]}\bigoplus
\cD^{\fr}_n(S) \underset{\Aut(S)\wr \sO(n)} \bigotimes A^{\otimes S_q} 
\Bigr) \Bigr)
\\
\nonumber
&
\underset{\rm (c)}{~\simeq~}
&
\FF\Bigl(
\underset{[S, \text{ ht-}q]}\bigoplus
\cD^{\fr}_n(S) \underset{\Aut(S)\wr \sO(n)} \bigotimes \coker\bigl(F^{\otimes S_q} 
\to 
A^{\otimes S_q} 
\bigr) \Bigr)
\\
\nonumber
&
\underset{\rm (d)}{~\simeq~}
&
\underset{k>0}\bigoplus~ \cD^{\fr}_n(k)\underset{\Sigma_k\wr \sO(n)}\bigotimes \Bigl(
\underset{[S, \text{ ht-}q]}\bigoplus
\cD^{\fr}_n(S) \underset{\Aut(S)\wr \sO(n)} \bigotimes \coker\bigl(F^{\otimes S_q} 
\to 
A^{\otimes S_q} 
\bigr) \Bigr)^{\otimes k}
\\
\nonumber
&
\underset{\rm (e)}{~\simeq~}
&
\underset{k>0} \bigoplus 
\Bigl(    \underset{ \bigl\{ ( [S_1],\dots,[S_k] )\mid ~ S_i\text{ ht-}q \bigr\} } \bigoplus  
~\Bigl(\cD^{\fr}_n(k)\times \prod_{i=1}^k \cD^{\fr}_n(S_i)\Bigr)  
\underset{\prod_{i=1}^k \Aut(S_i)\wr\sO(n) }\bigotimes 
\Bigl(  \bigotimes_{i=1}^k  
\coker\bigl(F^{\otimes S_q} 
\to 
A^{\otimes S_q} 
\bigr)
\Bigr)_{\Sigma_k \wr \sO(n)}
\\
\nonumber
&
\underset{\rm (f)}{~\simeq~}
&
\underset{[T, \text{ ht-} (q+1)]}\bigoplus~
\cD^{\fr}_n(T) \underset{\Aut(T)\wr \sO(n)} \bigotimes \Bigl( \underset{t\in T_1}\bigotimes \coker \bigl( F^{\otimes T_{1 \leq q+1}^{-1} (t)} \to A^{\otimes T_{1 \leq q+1}^{-1} (t)}  \bigr) \Bigr)~.
\end{eqnarray}
The identification~(a) is Lemma~\ref{01}(2).
The identification~(b) is Lemma~\ref{l1}, applied to each of $W=F$ and $W=A$.
The identification~(\ref{e3}) of Lemma~\ref{l1} is functorial in the argument $W$.
So the arrow in the second line above is that induced by the morphism $F\to A$ between underlying $\sO(n)$-$\Bbbk$-modules.  
The identification~(c) follows because colimits commute with one another.  
The identification~(d) is the expression~(\ref{e1}).  
The identification~(e) is distribution of tensor products among $\sO(n)$-$\Bbbk$-modules over direct sums thereof.  
The argument justifying the identification~(f) is identical to that for the identification~(e) in the proof of Lemma~\ref{l1}.
\end{proof}

We now make explicit, for each morphism $\sigma\colon [p]\to [q]$ in $\bDelta$, the morphism between $\Bbbk$-modules
\[
\sigma^\ast \colon X_q \longrightarrow X_p~.
\]
Through the identification~(\ref{88}), the above morphism is a morphism between $\Bbbk$-modules,
\begin{equation}\label{e6}
\xymatrix{
\underset{[T]}\bigoplus~
\cD^{\fr}_n(T) \underset{\Aut(T)\wr \sO(n)} \bigotimes \Bigl( \underset{t\in T_{1}}\bigotimes \coker \bigl( F^{\otimes T_{1 \leq q+1}^{-1} (t)} \to A^{\otimes T_{1 \leq q+1}^{-1} (t)}  \bigr) \Bigr)  
\ar[d]^-{\sigma^\ast}
\\
\underset{[S]}\bigoplus~
\cD^{\fr}_n(S) \underset{\Aut(S)\wr \sO(n)} \bigotimes \Bigl( \underset{s \in S_1}\bigotimes \coker \bigl( F^{\otimes S_{1 \leq p+1}^{-1} (s)} \to A^{\otimes S_{1 \leq p+1}^{-1} (s)}  \bigr) \Bigr)~,
}
\end{equation}
in which the direct sum in the domain is indexed by the set of isomorphism classes of height-$(q+1)$ rooted trees, and the direct sum in the codomain is indexed by the set of isomorphism classes of height-$(p+1)$ rooted trees.  
Being a morphism between direct sum $\Bbbk$-module,~(\ref{e6}) is a matrix, whose $([S],[T])$-enrty is the precomposition of~(\ref{e6}) by the canonical morphism from the $[T]$-summand, postcomposed by the projection onto the $[S]$-summand:
\begin{equation}\label{e7}
\xymatrix{
\cD^{\fr}_n(T) \underset{\Aut(T)\wr \sO(n)} \bigotimes \Bigl( \underset{t\in T_{1}}\bigotimes \coker \bigl( F^{\otimes T_{1 \leq q+1}^{-1} (t)} \to A^{\otimes T_{1 \leq q+1}^{-1} (t)}  \bigr) \Bigr)  
\ar[d]^-{\sigma^{T}_{S}}
\\
\cD^{\fr}_n(S) \underset{\Aut(S)\wr \sO(n)} \bigotimes \Bigl( \underset{s \in S_1}\bigotimes \coker \bigl( F^{\otimes S_{1 \leq p+1}^{-1} (s)} \to A^{\otimes S_{1 \leq p+1}^{-1} (s)}  \bigr) \Bigr)~.
}
\end{equation}
The next observation, which follows by unwinding the constructions involved in its statement, identifies this $([S],[T])$-entry in explicit terms.
Recall the notation $\sigma^\ast T$ of~(\ref{e9}).
\begin{observation}\label{ld}
Let $\sigma\colon [p]\to [q]$ be a morphism in $\bDelta$.  
Let 
\[
S\colon ([p]^{\tl})^{\op}
=
[p+1]^{\op}
\longrightarrow
\Fin^{\sf surj}_{\neq \emptyset}
\qquad\text{ and }\qquad
T\colon ([q]^{\tl})^{\op}
=
[q+1]^{\op}
\longrightarrow
\Fin^{\sf surj}_{\neq \emptyset}
\]
be two rooted trees.
\begin{enumerate}
\item
If $S \ncong \sigma^\ast T$, then the morphism $\sigma^{T}_{S}$ of~(\ref{e7}) is $0$.

\item
Suppose $S = \sigma^\ast T$.
If $\sigma(p)=q$, the morphism $\sigma^T_S$ of~(\ref{e7}) is that induced by the map between spaces over $\sB\bigl(\Sigma_{T_q}\wr \sO(n)\bigr) \times \BO(n)= \sB\bigl(\Sigma_{S_p} \wr \sO(n)\bigr)\times \BO(n)$,
\[
\sigma^T_S  \colon 
\cD_n(T) 
\underset{(\ref{08})} {~\simeq~}
\Bigl( \prod_{0< i \leq q+1} \prod_{t\in T_{i-1}} \conf_{T_{i-1<i}^{-1}(t)}(\RR^n)  \Bigr)_{\sO(n)}
\xra{~{}~{}~\circ~{}~{}~}
\Bigl(  \prod_{0< j \leq  p+1} \prod_{s\in S_{j-1}} \conf_{S_{j-1<j}^{-1}(s)}(\RR^n)  \Bigr)_{\sO(n)}
\underset{(\ref{08})} {~\simeq~}
\cD_n(S)~,
\]
of the $\cD_n$-operad structure.

\item
Suppose $\sigma = \partial_q\colon [q-1]\to [q]$ is the last face map.
Suppose $S  = \sigma^\ast T = T_{|[q]^{\op}}$.
In this case, the morphism $\sigma^T_S$ of~(\ref{e7}) is that induced by the diagram among $\Sigma_{T_q}\wr \sO(n)$-$\Bbbk$-modules
\[
\xymatrix{
\underset{t\in T_{1}} \bigotimes \underset{ t'\in T_{1\leq q}^{-1}(t)}\bigotimes  
\Bigl( \conf_{T_{q<q+1}^{-1}(t')}(\RR^n)\underset{\Sigma_{T_{q<q+1}^{-1}(t')} \wr \sO(n)}\bigotimes F^{\otimes T_{q<q+1}^{-1}(t') }   \Bigr) \ar[rrr]^-{\underset{t\in T_{1}} \otimes~ \underset{T_{1\leq q}^{-1}(t)}\otimes 
({\rm act}) } \ar[d]
&&&
\underset{t\in T_{1}} \bigotimes 
F^{\otimes T_{1\leq q}^{-1}(t)}   \ar[d]
\\
\underset{t\in T_{1}} \bigotimes \underset{ t'\in T_{1\leq q}^{-1}(t)}\bigotimes  
\Bigl( \conf_{T_{q<q+1}^{-1}(t')}(\RR^n)\underset{\Sigma_{T_{q<q+1}^{-1}(t')}\wr \sO(n)}\bigotimes A^{\otimes T_{q< q+1}^{-1}(t') }  \Bigr)  \ar[rrr]^-{\underset{t\in T_{1}} \otimes~ \underset{ T_{1\leq q}^{-1}(t)}\otimes 
({\rm act}) } 
&&&
\underset{t\in T_{1}} \bigotimes 
A^{\otimes T_{1\leq q}^{-1}(t)} 
}
\]
in which the horizontal morphisms are induced by the action of the $\cD_n$-operad on $F$ and on $A$, and the diagram commutes because the morphism $F\to A$ is one between non-unital $\cD_n$-algebras.

\end{enumerate}

\end{observation}

\subsubsection{\bf Explicating the associated graded}\label{sec.c}
We give a presentation of the associated graded of the skeletal filtration of the geometric realization $|X_\dagger|$.  
We first review some features of geometric realizations resulting from the cardinality filtration on $\bDelta^{\op}$.

Let $Y_\bullet \colon \bDelta^{\op}\to \Mod_{\Bbbk}$ be a simplicial $\Bbbk$-module.
The geometric realization $|Y_\bullet|$ carries a \emph{skeletal filtration}:
\begin{equation}\label{6}
0=: 
{\sf Sk}_{-1}(Y_\bullet)
\longrightarrow
{\sf Sk}_0 (Y_\bullet)
\longrightarrow
{\sf Sk}_1 (Y_\bullet)
\longrightarrow
\cdots 
\longrightarrow 
\underset{q\mapsto \oo} \colim ~{\sf Sk}_q(Y_\bullet)
\simeq
| Y_\bullet |~,
\end{equation}
for ${\sf Sk}_{q}(Y_\bullet) := \colim(\bDelta^{\op}_{\leq q+1}\xra{Y_\bullet}\Ch_\Bbbk)$. For each $q\geq 0$, the \emph{$q$th latching object} is the $\Bbbk$-module
\[
\sL_q(Y_\bullet)~:=~ 
\underset{[q]\underset{\neq}{\xra{\rm surj}}[r]}
\colim Y_r~,
\]
where the colimit is indexed by non-identity surjections from $[q]$. By inspection of the indexing category, the $q$th latching object is a pushout among finite coproducts
\begin{equation}\label{induct-latch}
\sL_{q-1}(Y_\bullet)\underset{\sL_{q-1}(Y_\bullet)^{\oplus q}} \bigoplus Y_{q-1}^{\oplus q}  
\xra{~\simeq~}
\sL_q(Y_\bullet)
\end{equation}
involving the $(q-1)$st latching object.  
There is a canonical morphism $\sL_q(Y_\bullet) \ra Y_q$, and 
the $\Bbbk$-module of \emph{non-degenerate} $q$-simplices is the cokernel:
\begin{equation}\label{e5}
Y_q^{\sf nd}~:=~
\coker\bigl(  \sL_q(Y_\bullet) \to Y_q \bigr)~.
\end{equation}
For each $q\geq 0$, these $\Bbbk$-modules assemble as a pushout diagram:
\begin{equation}\label{latch-po}
\xymatrix{
\sL_q(Y_\bullet)  \underset{\partial \Delta^q\otimes \sL_q(Y_\bullet)}\bigoplus  \partial \Delta^q\ot Y_q  \ar[rr]  \ar[d]
&&
{\sf Sk}_{q-1}(Y_\bullet)  \ar[d]
\\
Y_q  \ar[rr]
&&
{\sf Sk}_q(Y_\bullet)~.
}
\end{equation}
This pushout diagram determines a canonical identification between cokernel $\Bbbk$-module:
\begin{equation}\label{3}
\coker\Bigl(\sL_q(Y_\bullet)  \underset{\partial \Delta^q\otimes \sL_q(Y_\bullet)}\bigoplus  \partial \Delta^q\ot Y_q
\longrightarrow
Y_q\Bigr)
\xra{~\simeq~}
\coker\bigl(  {\sf Sk}_{q-1}(Y_\bullet) \to {\sf Sk}_q(Y_\bullet)  \bigr)  ~.
\end{equation}

We identify the latching objects as they concern the simplicial $\Bbbk$-module $X_\dagger$ of~(\ref{05}).  First:\begin{terminology}\label{t1}
Fix $0\leq i\leq q$.  
Let $T\colon [q]^{\op}\to  \Fin^{\sf surj}_{\neq \emptyset}$ be a height-$q$ rooted tree.
\begin{itemize}
\item
$T$ is \emph{$i$-non-degenerate} if the restriction $T_{|} \colon \{q-i<\dots<q\}\to \Fin^{\sf surj}_{\neq \emptyset}$ is conservative.  

\item
$T$ is \emph{$i$-degenerate} if the restriction $T_{|} \colon \{q-i<\dots<q\}\to \Fin^{\sf surj}_{\neq \emptyset}$ is \emph{not} conservative.  

\end{itemize}

\end{terminology}

\begin{lemma}\label{a}
Fix $q\geq 0$.
The canonical morphism between $\Bbbk$-modules, 
$
\sL_q ( X_\dagger )
\to
X_q
$, 
from the $q$th latching object to the $q$-simplices,
is a section of a retraction:
\begin{equation}\label{06}
\sL_q ( X_\dagger )
\oplus
X_q^{\sf nd}
~\simeq~
X_q~.
\end{equation}
In explicit terms, 
for each $q\geq 0$, there is a canonical identification between $\Bbbk$-modules:
\begin{equation}\label{88'}
X_q^{\sf nd}
~\simeq~
\underset{[T]}\bigoplus~
\cD^{\fr}_n(T) \underset{\Aut(T)\wr \sO(n)} \bigotimes \Bigl( \underset{t\in T_1}\bigotimes \coker \bigl( F^{\otimes T_{1 \leq q+1}^{-1} (t)} \to A^{\otimes T_{1 \leq q+1}^{-1} (t)}  \bigr) \Bigr)~,
\end{equation}
the direct sum being indexed by isomorphism classes of $q$-non-degenerate height-$(q+1)$ rooted trees.

\end{lemma}

\begin{proof}
The result follows from the following two statements.
\begin{enumerate}
\item There is a canonical identification of the $q$th latching object:
\begin{equation}\label{e6'}
\sL_q(X_{\dagger})
~\simeq~
\underset{[T]}\bigoplus~
\cD^{\fr}_n(T) \underset{\Aut(T)\wr \sO(n)} \bigotimes \Bigl( \underset{t\in T_1}\bigotimes \coker \bigl( F^{\otimes T_{1 \leq q+1}^{-1} (t)} \to A^{\otimes T_{1 \leq q+1}^{-1} (t)}  \bigr) \Bigr)~
\end{equation}
in which the direct sum is indexed by the set of isomorphism classes of $q$-degenerate height-$(q+1)$ rooted trees.

\item
Through the identification~(\ref{e6'}) and the identification~(\ref{88}), the canonical morphism $\sL_q(X_\dagger) \to X_q$ is that induced by the evident inclusion of sets indexing the direct sums.  

\end{enumerate}
In this proof, we will use the simplified notation
\begin{equation}\label{e12}
V(T) ~:=~ 
\underset{t\in T_1}\bigotimes \coker \bigl( F^{\otimes T_{1 \leq q+1}^{-1} (t)} \to A^{\otimes T_{1 \leq q+1}^{-1} (t)}  \bigr)
\end{equation}
so that the expression~(\ref{e6'}) can be rewritten as 
\begin{equation}\label{e13}
\sL_q(X_\dagger) 
~\simeq~ 
\underset{[T]}\bigoplus~
\cD^{\fr}_n(T) \underset{\Aut(T)\wr \sO(n)} \bigotimes V(T)~.
\end{equation}

We prove these assertions by induction on $q\geq 0$.
By construction $\sL_0(X_\dagger)=0$.
By definition of non-degeneracy, there are no $0$-degenerate height-1 rooted trees. 
Therefore, both sides of~(\ref{e6'}) are $0$ in the case that $q=0$, which establishes the base case of our induction.  

Assume $q>0$.  
Using the inductive assumption on $q$, through the identification~(\ref{induct-latch}), it is enough to show that the canonical morphism between $\Bbbk$-modules,
\begin{equation}\label{e14}
\underset{\sigma\colon [q]\underset{\neq}{\xra{\rm surj}}[r]}
\bigoplus
\Bigl( \underset{\bigl[S, \text{ ht-$(r+1)$, $r$-non-deg}\bigr]}
\bigoplus~
\cD^{\fr}_n(S) \underset{\Aut(S)\wr \sO(n)} \bigotimes V(S)  \Bigr)
\xra{~{}~\simeq~{}~}
\underset{\bigl[T, \text{ ht-$(q+1)$, $q$-deg}\bigr]}
\bigoplus~
\cD^{\fr}_n(T) \underset{\Aut(T)\wr \sO(n)} \bigotimes V(T)~,
\end{equation}
is an equivalence.
Inspecting Observation~\ref{ld}, this morphism is induced by a map between sets indexing the direct sums.
So we show the above morphism is an equivalence by first showing the canonical morphism between sets indexing these direct sums is a bijection, then by showing that each morphism between corresponding direct-summands is an equivalence.

We are to show that the canonical map between sets
\begin{equation}\label{e15}
\Bigl\{ \Bigl(\sigma\colon [q]\underset{\neq}{\xra{\rm surj}}[r]~,~\bigl[S, \text{ ht-$(r+1)$, $r$-non-deg}\bigr] \Bigr) \Bigr\}
\longrightarrow
\Bigl\{ \bigl[T, \text{ ht-$(q+1)$, $q$-deg}\bigr] \Bigr\}
~,\qquad
(\sigma , [S])\mapsto [\sigma^\ast S]~,
\end{equation}
is a bijection.
Both surjectivity, and injectivity, follow from the following discussion.

Let $T =\bigl( [q+1]^{\op} \xra{T} \Fin^{\sf surj}_{\neq \emptyset}\bigr)$ be a height-$(q+1)$ rooted tree.
Consider the localization-conservative factorization of the restriction of $T$:
\[
\xymatrix{
\{1<\dots<q+1\}^{\op}    \ar[rr]^-{T_{|}}  \ar[dr]_-{\rm localization}  
&&
\Fin^{\sf surj}_{\neq \emptyset}
\\
&
L  \ar[ur]^-{S_{|}}_-{\rm conservative}
&
.
}
\]
Because every morphism in the category $\Fin^{\sf surj}_{\neq \emptyset}$ is an epimorphism, the category $L$ in the above diagram is the opposite of a finite non-empty linearly ordered set, and the localization is therefore a degeneracy morphism in the category $\bDelta$:
\[
\sigma^{\op}\colon \{1<\dots<q+1\}^{\op}
\xra{~\rm localization~}
\{1<\dots<r+1\}^{\op} \cong L~,
\]
for some $0\leq r \leq q$.  
The functor $S_{|}$ in the above diagram uniquely extends as a height-$(r+1)$ rooted tree 
\[
S\colon 
[r+1]^{\op}  
\longrightarrow 
\Fin^{\sf surj}_{\neq \emptyset}~,
\]
given by declaring the value on $0$ to be $\ast$. 
By construction, the functor $S_{|}$ is conservative, and therefore the height-$(r+1)$ rooted tree $S$ is $r$-non-degenerate.  
By construction, there is an isomorphism between rooted trees: $T \cong \sigma^\ast S$.  
Finally, direct from the definition of a rooted tree being degenerate, $T$ is $q$-degenerate if and only if the degeneracy morphism $\sigma$ is not an isomorphism.  
This immediately concludes the verification that the map~(\ref{e15}) is surjective.  
We also conclude that, if $\rho^\ast R \cong \sigma^\ast S$, then both $R$ and $S$ have the same height, $\rho = \sigma$, and $R$ and $S$ are isomorphic.
This gives injectivity of~(\ref{e15}).

It remains to show that restriction of the morphism~(\ref{e14}) to corresponding direct-summands is an equivalence of $\Bbbk$-modules.
So let $\sigma\colon [q]\to [r]$ be a surjective morphism in $\bDelta$, and let $S$ be a height-$(r+1)$ rooted tree.
We are to show that the induced morphism between $\Bbbk$-modules
\begin{equation}\label{e16}
\cD^{\fr}_n(S) \underset{\Aut(S)\wr \sO(n)} \bigotimes V(S)
\longrightarrow
\cD^{\fr}_n(\sigma^\ast S) \underset{\Aut(\sigma^\ast S)\wr \sO(n)} \bigotimes V(\sigma^\ast S)
\end{equation}
is an equivalence.  
This immediately follows from these three assertions.
\begin{itemize}

\item
The induced map between sets $(\sigma^\ast S)_r \to S_q$ is an isomorphism, and the induced homomorphism between groups
\[
\Aut(S) \longrightarrow \Aut(\sigma^\ast S)
\]
is an isomorphism over $\Sigma_{S_q}$.  
This follows by induction on the height of $S$.
As the base case of $r=0$, the result follows upon inspecting the semi-direct product description~(\ref{e4}), applied to $\Aut(\sigma^\ast S)$.  
The inductive step also follows from that semi-direct product description, applied to both $\Aut(S)$ and to $\Aut(\sigma^\ast S)$.

\item
The induced $\bigl(\Aut(S)\wr \sO(n)\bigr) \times \sO(n)$-equivariant map 
\[
\cD^{\fr}_n(S)
\longrightarrow
\cD^{\fr}_n(\sigma^\ast S)
\]
is an equivalence.
This follows by induction on the height of $S$.
As the base case of $r=0$, the result follows upon inspecting the $\bigl(\Aut(S)\wr \sO(n)\bigr)\times \sO(n)$-equivariant identification~(\ref{08}).
The inductive step also follows from that equivariant identification, applied to both $\cD^{\fr}_n(S)$ and to $\cD^{\fr}_n(\sigma^\ast S)$.

\item
The induced $\Aut(S)\wr \sO(n)$-equivariant morphism between $\sO(n)$-$\Bbbk$-modules
\[
V(S)
\longrightarrow
V(\sigma^\ast S)
\]
is an equivalence.  
By inspection of the definition~(\ref{e12}) of $V(-)$, the assertion is true provided the following two assertions are true.  
\begin{enumerate}
\item 
The induced map between sets,
$(\sigma^\ast S)_1 \to S_1$,
is a bijection.

\item
For each $s\in S_1 =  (\sigma^\ast S)_1$,
the induced map between sets,
$(\sigma^\ast S)_{1\leq q+1}^{-1}(s) \to S_{1\leq r+1}^{-1}(s)$,
is a bijection.

\end{enumerate}
Inspecting the construction of $\sigma^\ast S$, both of these assertions follow because $\sigma$ is a surjection.

\end{itemize}
This completes this proof of Lemma~\ref{a}.
\end{proof}

We now identify the associated graded of the skeletal filtration~(\ref{6}), as it concerns the simplicial $\Bbbk$-module $X_\dagger$ of~(\ref{05}).  
\begin{cor}\label{lb}
Fix $q\geq 0$.
There are canonical identifications among $\Bbbk$-modules:
\begin{eqnarray}
\nonumber
\coker\bigl( {\sf Sk}_{q-1}(X_\dagger) \to {\sf Sk}_q(X_\dagger)\bigr)
&
~\simeq~
&
S^q\otimes X_q^{\sf nd}
\\
\nonumber
&
~\simeq~
&
\underset{[T]}\bigoplus~
S^q\otimes \cD^{\fr}_n(T) \underset{\Aut(T)\wr \sO(n)} \bigotimes \Bigl( \underset{t\in T_1}\bigotimes \coker \bigl( F^{\otimes T_{1 \leq q+1}^{-1} (t)} \to A^{\otimes T_{1 \leq q+1}^{-1} (t)}  \bigr) \Bigr)~,
\end{eqnarray}
in which the direct sum is indexed by the set of isomorphism classes of $q$-non-degenerate height-$(q+1)$ rooted trees.  

\end{cor}

\begin{proof}
We establish the following string of canonical identifications among $\Bbbk$-modules.
\begin{eqnarray}
\nonumber
\coker\bigl( {\sf Sk}_{q-1}(X_\dagger) \to {\sf Sk}_q(X_\dagger)\bigr)
&
\underset{\rm (a)}{~\simeq~}
&
\coker\Bigl(\sL_q(X_\dagger)  \underset{\partial \Delta^q\otimes \sL_q(X_\dagger)}\bigoplus  \partial \Delta^q\ot X_q
\longrightarrow
X_q\Bigr)
\\
\nonumber
&
\underset{\rm (b)}{~\simeq~}
&
\coker\Bigl(\sL_q(X_\bullet)  \underset{\partial \Delta^q\otimes \sL_q(X_\bullet)}\bigoplus  \partial \Delta^q\ot ( \sL_q(X_\bullet) \oplus X_q^{\sf nd})
\longrightarrow
\sL_q(X_\bullet)\oplus X_q^{\sf nd}\Bigr)
\\
\nonumber
&
\underset{\rm (c)}{~\simeq~}
&
S(\partial \Delta^{q-1}) \otimes \coker ( \sL_q(X_\bullet) \to X_q )
\\
\nonumber
&
\underset{\rm (d)}{~\simeq~}
&
S^q\otimes X_q^{\sf nd}
\\
\nonumber
&
\underset{\rm (e)}{~\simeq~}
&
\underset{[T]}\bigoplus~
S^q\otimes \cD^{\fr}_n(T) \underset{\Aut(T)\wr \sO(n)} \bigotimes \Bigl( \underset{t\in T_1}\bigotimes \coker \bigl( F^{\otimes T_{1 \leq q+1}^{-1} (t)} \to A^{\otimes T_{1 \leq q+1}^{-1} (t)}  \bigr) \Bigr)~,
\end{eqnarray}
The identification~(a) is the expression~(\ref{3}).
The identification~(b) is the expression~(\ref{06}).
The identification~(c) follows from the fact that tensoring with $\partial \Delta^q$ distributes over direct sums, together with cancelation, where $S (\partial \Delta^q)$ is the suspension.
The identification~(d) is an identification of this suspension $S^q\simeq S(\partial \Delta^q)$, together with the defining expression~(\ref{e5}) of $X_q^{\sf nd}$.
The identification~(e) is the expression~(\ref{88'}) of Lemma~\ref{a}, together with the fact that tensoring with a space distributes over direct sums. 
\end{proof}

\subsubsection{\bf Coconnectivity of the associated graded of the skeletal filtration}\label{sec.d}
We now bound the coconnectivity of the associated graded of the skeletal filtration~(\ref{6}) of $|X_\dagger|$.

We make repeated use of the following formal result.
For $G$ a finite group, consider the smallest full $\infty$-subcategory 
\[
\sB G^{\op} \hookrightarrow \Mod^{\sf fin.free}_{G^{\op}}(\Spaces)  \subset \Mod_{G^{\op}}(\Spaces):=\Fun(\sB G^{\op} , \Spaces)
\]
that contains the image of the Yoneda functor and that is closed under finite colimits.
We say that a right $G$-space $Z$ is \emph{finite and free} if it is contained in this full $\infty$-subcategory.

\begin{prop}\label{l5}
Let $G$ be a finite group.
Let $Z$ be a finite and free right $G$-space.
Let $W\in \Mod_{G}(\Mod_{\Bbbk})$ be a left $G$-$\Bbbk$-module.  
If $Z$ is $z$-coconnective and $W$ is $w$-coconnective, then
\[
Z\underset{G}\otimes W
\]
is $(z+w)$-coconnective.

\end{prop}

\begin{proof}
Firstly, $Z\otimes W$ is $(z+w)$-coconnective. 
The fully faithful inclusion $\Mod_{\Bbbk}^{\leq z+w} \hookrightarrow \Mod_{\Bbbk}$ of the $(z+w)$-coconnective $\Bbbk$-modules is a right adjoint, and therefore it preserves limits.  
Therefore, the $\Bbbk$-module $Z\overset{G}\otimes W$ of diagonal $G$-invariants is $(z+w)$-coconnective.  
The result will then follow once it is established that the norm map
\begin{equation}\label{e11}
Z\underset{G}\otimes W 
\longrightarrow
Z\overset{G} \otimes W
\end{equation}
is an equivalence between $\Bbbk$-modules when $Z$ is finite and free---this is essentially Proposition~2.10 of~\cite{AhearnKuhn}.
By direct inspection, the norm map is an equivalence in the case $Z=G$ is in the image of the Yoneda functor. Since both $G$-invariants and $G$-coinvariants commute with finite colimits, the result follows.
\end{proof}

We record this minor enhancement of Proposition~\ref{l5}.
\begin{cor}\label{l10}
Let $i>0$ and $H$ be a topological group.
Let $Z$ be a finite and free $\Sigma_i$-space.
Let $W$ be a $(\Sigma_i\wr H)$-$\Bbbk$-module.
If the underlying space of $Z$ is $z$-coconnective, and if the underlying $\Bbbk$-module of $W$ is $w$-coconnective, then the $\Bbbk$-module
\[
( Z\times H^{i})  \underset{\Sigma_i\wr H} \otimes W
\]
is $(z+w)$-coconnective.  

\end{cor}

\begin{proof}
We establish the following sequence of equivalences among $\Bbbk$-modules.
\begin{eqnarray}
\nonumber
(Z\times H^{i}) \underset{\Sigma_i\wr H} \otimes W
&
~:=~
&
\bigl((Z\times H^{i})   \otimes W   \bigr)_{\Sigma_i\wr H}
\\
\nonumber
&
\underset{\rm (a)}{~\simeq~}
&
\Bigl(  (Z\times H^{i})  \otimes W   \bigr)_{H^{i}}  \Bigr)_{\Sigma_i}
\\
\nonumber
&
\underset{\rm (b)}{~\simeq~}
&
\bigl( (Z\times H^{i})  \underset{H^{i}} \otimes W    \bigr)_{\Sigma_i}
\\
\nonumber
&
\underset{\rm (c)}{~\simeq~}
&
\Bigl( Z \otimes W    \Bigr)_{\Sigma_i}~=:~Z \underset{\Sigma_i} \otimes W~.
\end{eqnarray}
The first, and last, identifications are definitional, as diagonal quotients.  
The identification~(a) is a formal consequence of the split short exact sequence of groups $H^{i} \hookrightarrow \Sigma_i\wr H \leftrightarrows \Sigma_i$.
The identification~(b) is definitional, as a diagonal quotient again.
The identification~(c) is a formal cancelation.
The result now follows from Proposition~\ref{l5}.
\end{proof}

We make essential use of Proposition~\ref{l5}, founded on the following well-known result.
\begin{prop}\label{l9}
For each $k>0$, the space $\conf_k(\RR^n)$ is $(k-1)(n-1)$-coconnective.

\end{prop}

\begin{proof}
We proceed by induction on $k>0$.  
The statement is immediate in the case $k=1$.
So assume $k>1$.
Consider the fibration sequence 
\[
(S^{n-1})^{\vee k-1} \longrightarrow \conf_k(\RR^n) \longrightarrow \conf_{k-1}(\RR^n)
\]
in which the second arrow is restriction along the standard inclusion $\{1,\dots,k-1\}\hookrightarrow \{1,\dots,k\}$.  
Inspecting the Serre spectral sequence associated to this fibration sequence, the coconnectivity of $\conf_k(\RR^n)$ is bounded above by the sum of the coconnectivities of the fiber and the base, which is $(n-1) + (k-2)(n-1) = (k-1)(n-1)$.  
\end{proof}

We record the following technical result.
\begin{lemma}\label{l8}
Let $U\to V$ be a morphism between $\Bbbk$-modules each of which are $r$-coconnective. If the map $\sH_{r}(U) \to \sH_{r}(V)$ is injective, then the cokernel
\[
\coker\bigl( U^{\ot k} \to V^{\ot k}\bigr)
\]
is $kr$-coconnective for all $k\geq 0$.
\end{lemma}

\begin{proof}

The coconnectivity assumptions on each of $U$ and $V$ implies each of $U^{\ot k}$ and $V^{\ot k}$ is $kr$-coconnective.  
From the long exact sequence associated to a cofibration sequence, the cokernel $\coker\bigl( U^{\ot k} \to V^{\ot k} \bigr)$ is $kr$-coconnective if and only if the map between $\Bbbk$-vector spaces
\begin{equation}\label{97}
\sH_{kr}(U^{\ot k}) 
\longrightarrow
\sH_{kr}(V^{\ot k}) 
\end{equation}
is injective.  
Since $\Bbbk$ is a field, for each $r$-coconnective $\Bbbk$-module $W$ there is an isomorphism $
\sH_{kr}(W^{\otimes k})
~\cong~
\sH_{r}(W)^{\otimes k}$.
Injectivity of the induced map~(\ref{97}) follows from the assumed injectivity of $\sH_{r}(U) \to \sH_{r}(V)$.
\end{proof}

\begin{lemma}\label{l7}
For each $q\geq 0$, the $\Bbbk$-module~(\ref{e5}) of non-degenerate simplices $X_q^{\sf nd}$ is $(-N-q)$-coconnective.  

\end{lemma}

\begin{proof}
Recall from Lemma~\ref{a}, the identification between $\Bbbk$-modules:
\[
X_q^{\sf nd} ~\underset{(\ref{88'})}{~\simeq~}~
\underset{[T]}\bigoplus~
\cD^{\fr}_n(T) \underset{\Aut(T)\wr \sO(n)} \bigotimes \Bigl( \underset{t\in T_1}\bigotimes \coker \bigl( F^{\otimes T_{1 \leq q+1}^{-1} (t)} \to A^{\otimes T_{1 \leq q+1}^{-1} (t)}  \bigr) \Bigr)~,
\]
in which the direct sum is indexed by the set of isomorphism classes of $q$-non-degenerate height-$(q+1)$ rooted trees.
It is therefore sufficient to show that each direct-summand is $(-N-q)$-coconnective.  
We do this using Corollary~\ref{l10}.
Fix a $q$-non-degenerate height-$(q+1)$ rooted tree $T$.

We first identify the coconnectivity of $\cD^{\fr}_n(T)$.  
Recall the identification as right $\Aut(T)\wr \sO(n)$-modules in $\Mod_{\sO(n)}(\Spaces)$:
\begin{eqnarray}
\nonumber
\cD^{\fr}_n(T)
&
\underset{(\ref{08})}{~\simeq~}
&
\Bigl( \prod_{t\in T_q} \conf^{\fr}_{T_{q<q+1}^{-1}(t)}(\RR^n)  \Bigr)
\times 
\Bigl( \prod_{0\leq  i < q} \prod_{t\in T_i} \conf_{T_{i<i+1}^{-1}(t)}(\RR^n) \Bigr)
\\
\nonumber
&
\underset{(\ref{e33})}{~\simeq~}
&
\Bigl( \prod_{0\leq  i \leq q} \prod_{t\in T_i} \conf_{T_{i<i+1}^{-1}(t)}(\RR^n) \Bigr) \times \sO(n)^{T_{q+1}}   ~.
\end{eqnarray}
This identification reveals that the $\sO(n)^{T_{q+1}}$-coinvariants 
\[
\bigl( \cD^{\fr}_n(T)\bigr)_{\sO(n)^{T_{q+1}}}
~\simeq~
\prod_{0\leq  i \leq q} \prod_{t\in T_i} \conf_{T_{i<i+1}^{-1}(t)}(\RR^n) 
\]
is finite and free 
as a $\Aut(T)$-space.
Using Proposition~\ref{l9}, which states that $\conf_k(\RR^n)$ is $(k-1)(n-1)$-coconnective, 
this identification also reveals that the coconnectivity of the underlying space 
of coinvariants $\cD^{\fr}_n(T)_{\sO(n)^{T_{q+1}}}$ 
is bounded above by
\begin{eqnarray}
\nonumber
\sum_{0\leq i \leq q} \sum_{t\in T_i} (|T_{i<i+1}^{-1}(t)|-1)(n-1)
&
~=~
&
\sum_{0\leq i \leq q} (|T_{i+1}|-|T_i|) (n-1)
\\
\nonumber
&
~=~
&
(|T_{q+1}|-|T_0|)(n-1)
\\
\nonumber
&
~\leq~
& 
( (q+1) - 1 ) (n-1) = q(n-1)~,
\end{eqnarray}
in which the inequality is obtained as follows.
By construction, for each $0\leq i\leq j\leq q+1$, the map between finite sets, $T_{i\leq j}\colon T_j \to T_i$ is surjective.
Therefore, the assumption that $T$ is $q$-non-degenerate precisely yields the inequalities among cardinalities:
\[
|T_{q+1}| > |T_q| > \dots > |T_1|\geq |T_0| ~.
\]
By construction, the cardinality $|T_0|=1$.
Therefore, 
\begin{equation}\label{72}
|T_{q+1}|~\geq ~q+1~.
\end{equation}
This concludes the above coconnectivity bound of the space 
of coinvariants 
$\cD^{\fr}_n(T)_{\sO(n)^{T_{q+1}}}$.

We now bound the coconnectivity of the $\Bbbk$-module
\begin{equation}\label{e10}
\underset{t\in T_1}\bigotimes \coker \bigl( F^{\otimes T_{1 \leq q+1}^{-1} (t)} \to A^{\otimes T_{1 \leq q+1}^{-1} (t)}  \bigr) ~.
\end{equation}
The coconnectivity of the underlying $\Bbbk$-modules of both $F$ and of $A$ are assumed to be at most $(-N)$.  
Therefore, for each $t\in T_1$, the coconnectivity of both of the $\Bbbk$-modules $F^{\ot T_{1\leq q+1}^{-1}(t)}$ and $A^{\ot T_{1\leq q+1}^{-1}(t)}$ is at most 
\[
| T_{1\leq q+1}^{-1}(t) | (-N)~.
\]
Lemma~\ref{l8} directly applies to the morphism $F\to A$ between underlying $\Bbbk$-modules to imply, for each $t\in T_1$, that the coconnectivity of the cokernel 
\[
\coker \bigl( F^{\otimes T_{1 \leq q+1}^{-1} (t)} \to A^{\otimes T_{1 \leq q+1}^{-1} (t)}  \bigr)
\]
is at most $| T_{1\leq q+1}^{-1}(t) | (-N)$.  
Therefore, the coconnectivity of~(\ref{e10}) is at most 
\[
\underset{t\in T_1}\sum   | T_{1\leq q+1}^{-1}(t) | (-N)
~=~
|T_{q+1}|(-N)
~\leq ~
(q+1)(-N)~,
\]
where the inequality is established in the previous paragraph.

Using these coconnectivity bounds, Corollary~\ref{l10}
gives that the coconnectivity of the $[T]$-direct-summand of $X_q^{\sf nd}$ is at most
\[
q(n-1) + (q+1)(-N) 
~=~
-N -q - q(N-n)
\leq -N-q~,
\]
where the inequality is due to the assumption that $N\geq n$.  
\end{proof}

We are now prepared to prove Lemma~\ref{cocon-cofib}
\begin{proof}[Proof of Lemma~\ref{cocon-cofib}]
From Lemma~\ref{01}(1), it is enough to show that the $\Bbbk$-module $|X_\dagger|$ is $(-N)$-coconnective.
Recall from~(\ref{6}), the equivalence between $\Bbbk$-modules
\[
\underset{q\geq 0} \colim~{\sf Sk}_q(X_\dagger)~\simeq~ |X_\dagger|~.
\]
For each $r$, the functor $\sH_{r}(-)\colon \Mod_{\Bbbk} \to {\sf Vect}_{\Bbbk}$ commutes with filtered colimits.  
It follows that $|X_\dagger|$ is $(-N)$-coconnective if and only if ${\sf Sk}_q(X_\dagger)$ is $(-N)$-coconnective for each $q\geq 0$.  
We do this by induction on $q\geq 0$.  
For $q=0$, there is a canonical identification between $\Bbbk$-modules,
\[
{\sf Sk}_0(X_\dagger) 
= 
X_0 = X_0^{\sf nd}~.
\]
In the case $q=0$, Lemma~\ref{l7} implies ${\sf Sk}_0(X_\dagger)$ is $(-N)$-coconnective.
This concludes the base case.  
So let $q>0$.  
Consider the cofiber sequence:
\[
{\sf Sk}_{q-1}(X_\dagger)
\longrightarrow
{\sf Sk}_q(X_\dagger)
\longrightarrow
\coker\bigl({\sf Sk}_{q-1}(X_\dagger)  \to {\sf Sk}_{q}(X_\dagger)  \bigr)~.
\]
From the long exact sequence in homology associated to a cofiber sequence, it is enough to show that the cokernel $\coker\bigl({\sf Sk}_{q-1}(X_\dagger)  \to {\sf Sk}_{q}(X_\dagger)  \bigr)$ is $(-N)$-coconnective.  
Through Corollary~\ref{lb}, it is enough to show that the $\Bbbk$-module
\[
S^q\otimes X_q^{\sf nd}
\]
is $(-N)$-coconnective.
We are therefore reduced to showing that $X_q^{\sf nd}$ is $(-N-q)$-coconnective, which is given by Lemma~\ref{l7}.  
\end{proof}

\subsection{Finitely presented coconnective algebras, continued}

Before addressing this affine case, we require the following easy, but important, comparison of factorization homology and factorization cohomology in the stable setting.

\begin{prop}\label{switch} 
Let $\cV$ be a $\ot$-stable-presentable symmetric monoidal $\infty$-category.  
Let $A$ be an augmented $n$-disk algebra in $\cV$ for which the natural map 
\[
(A^\vee)^{\ot i}   \xra{~\simeq~}    (A^{\ot i})^\vee
\]
is an equivalence for all $i\geq 0$. 
The composition $\Disk_{n}^+ \cong \Disk_{n,+}^{\op}\xra{A^\vee} \cV^{\op} \xra{(-)^\vee} \cV$ canonically extends to an augmented $n$-disk coalgebra $A^\vee \colon \Disk_n^+ \to \cV$.
Furthermore, for any zero-pointed $n$-manifold $M_\ast$ the canonical arrow
\[
 \int^{M_\ast} A^\vee \xra{~\simeq~}   \Bigl(\int_{M_\ast} A\Bigr)^\vee
\]
is an equivalence.
\end{prop}

\begin{proof} 
The first statement is clear, by hypothesis.  
Furthermore, the hypothesis on $A$ supports the leftward arrow in the string of canonical equivalences in $\cV$:
\[
\scriptstyle
\Bigl(\int_{M_\ast} A\Bigr)^\vee =  \Bigl(\underset{U_+ \in \Disk_+(M_\ast)}\colim A(U_+) \Bigr)^\vee \xra{~\simeq~} \underset{U_+\in (\Disk_+(M_\ast))^{\op}}\limit A(U_+)^\vee \xla{~\simeq~}  \underset{U^+\in ((\Disk_n^+)_{M_\ast^\neg/} }\limit A^\vee(U^+) = \int^{M_\ast} A^\vee~.
\]
\end{proof}

We now turn to the affine case in the proof of our main theorem.

\begin{theorem}\label{affine} 
Let $\Bbbk$ be a field.
There is an equivalence between functors
\[
\Bigl(\int_{(-)} ~{}~\Bigr)^\vee~\simeq ~\int_{(-)^\neg} \DD^n~\colon~\FPres_n^{\leq -n} \to \Fun(\ZMfld_n,\Ch_\Bbbk)~.
\]
Specifically, for each finitely presented $(-n)$-coconnective augmented $n$-disk algebra $A$ in chain complexes over $\Bbbk$, and each $n$-dimensional cobordism $\ov{M}$ with boundary $\partial \ov{M} = \partial_{\sL}\coprod \partial_{\sR}$, there is a canonical equivalence of chain complexes over $\Bbbk$
\[
\Bigl(\int_{\ov{M}\smallsetminus \partial_{\sL}} A\Bigr)^\vee~ \simeq~ \int_{\ov{M}\smallsetminus \partial_{\sR}} \DD^nA~.
\]

\end{theorem}  
\begin{proof}
Fix a finitely presented $(-n)$-coconnective augmented $n$-disk algebra $A$.
We first consider the case of a zero-pointed $n$-manifold $M_\ast$ for which $M$ is connected.
We will establish the zig-zag of arrows in $\Ch_\Bbbk$
\[
\Bigl(\int_{M_\ast} A\Bigr)^\vee \xla{(1)} \Bigl(P_\infty \int_{M_\ast} A\Bigr)^\vee  \xla{(2)} \Bigl(\int^{M_\ast^\neg} \bBar(A)\Bigr)^\vee  \xra{(3)} \Bigl(\int_{M_\ast^\neg} \DD^n A\Bigr)^{\vee\vee} \xla{(4)} \int_{M_\ast^\neg}\DD^n A
\]
and such that each arrow in it is an equivalence.
The arrow~(1) is the linear dual of the canonical map to the limit of the Goodwillie cofiltration.
The arrow~(2) is the linear dual of that from Theorem~\ref{compare-towers}, which factors the Poincar\'e/Koszul duality morphism and is an equivalence by that result.  
The arrow~(3) is the linear dual of that of Proposition~\ref{switch}.  
The arrow~(4) is the standard one.  All of these maps are natural in $M_\ast$.

We verify that~(3) is an equivalence.  
Recall that since $A$ is finitely presented, therefore $LA$ is perfect.  
The equivalence of chain complexes $\DD^n A\simeq \Bbbk\oplus\bigl((\RR^n)^+ \ot LA\bigr)^\vee$ given by Theorem \ref{L-bar} implies $\DD^nA$ too is perfect.  
We can thus apply Proposition~\ref{switch} to conclude that~(3) is an equivalence.

We next show~(1) is an equivalence.  
For this we will show that the successive layers in the Goodwillie cofiltration become contractible through a range.  
Specifically, through Theorem~\ref{compare-towers}, we will show
\[
\conf_i^{\sf fr}(M_\ast) \underset{\Sigma_i\wr \sO(n)} \bigotimes (LA)^{\ot i}
\]
is $(-i+1)$-coconnective.

First, we show that $LA$ is a $(-n)$-coconnective $\sO(n)$-module given that $A$ is in $\fpres_n^{\leq -n}$. We prove this by induction on a resolution of $A$. 
The base case is when $A$ is a free augmented $n$-disk algebra: for $V\in \Mod_{\sO(n)}(\Perf_\Bbbk^{\leq -n})$ then $L\FF V \simeq V$ is $(-n)$-coconnective.
To prove the inductive step, consider a pushout of $(-n$)-coconnective algebras, $A\cong \Bbbk\underset{\FF V}\amalg B$, where $B$ is finitely presented $(-n)$-coconnective and subject to the assumption that $LB$ is $(-n)$-coconnective. If the map $\FF V\ra B$ is was not injective on homology in degree $-n$, then the pushout would have homology in degree $1-n$. 
By assumption $A$ is $(-n)$-coconnective; consequently, the map $\FF V\ra B$ must be injective on homology in degree $-n$.
Being a left adjoint, the cotangent space functor satisfies $L\bigl(\Bbbk\underset{\FF V}\amalg B) \simeq L\Bbbk \underset{L\FF V}\amalg LB \simeq \coker(V \to LB)$, and the cokernel is again $(-n)$-coconnective. Consequently, $LA$ is $(-n)$-coconnective and, further $(LA)^{\ot i}$ is a $(-ni)$-coconnective $\Sigma_i\wr \sO(n)$-module.

Via factorization homology, the $\Sigma_i\wr \sO(n)$-module $(LA)^{\ot i}$ determines a $\oplus$-excisive symmetric monoidal functor $\ZMfld_{\Sigma_i\wr \sO(n)}\to \Ch_\Bbbk^\oplus$ 
from zero-pointed $(ni)$-manifolds equipped with a $\sB\bigl(\Sigma_i\wr \sO(n)\bigr)$-structure on their tangent bundle.
Explicitly, this functor is given by the associated bundle construction: $W_\ast \mapsto {\sf Fr}_{W_\ast}\underset{\Sigma_i\wr \sO(n)}\bigotimes (LA)^{\ot i}$.  
Through $\oplus$-excision, the value of this functor on $W_\ast$ is $(s-ni)$-coconnective whenever $W_\ast$ is $s$-coconnective.  
The zero-pointed $(ni)$-manifold $\conf_i(M_\ast)$ is equipped with a $\sB\bigl(\Sigma_i\wr \sO(n)\bigr)$-structure on its tangent bundle; and Proposition~\ref{conf-stuff} implies its coconnectivity is equal to $n + (n - 1)(i - 1)$. The coconnectivity range implying equivalence (1) now follows by addition.

Lastly, we show (4) is an equivalence. It suffices to show that $\int_{M_\ast^\neg}\DD^nA$ is connective and has finite-rank homology groups over $\Bbbk$ in all dimensions. To show this, we apply the cardinality filtration of factorization homology: since $\int_{M_\ast^\neg}\DD^nA$ is a sequential colimit of the filtration $\tau^{\leq i} \int_{M_\ast^\neg}\DD^nA$, we further reduce to showing that the layers of the filtration are connective and grow in connectivity with $i$. By Theorem~\ref{compare-towers}, these cardinality layers are shifts of the duals of the Goodwillie layers of $\int_{M_\ast} A$. Consequently, their connectivities follows from the coconnectivities of the Goodwillie layers, which were computed in the proof of equivalence (1).

We conclude the theorem for general $M_\ast$ by showing that each side transforms wedges in the manifold variable to duals of tensor products. For the rightand side, we have the equivalence
\[
\int_{(M_\ast\vee N_\ast)^\neg} \DD^nA\simeq \int_{M_\ast^\neg\vee N_\ast^\neg} \DD^nA\simeq \int_{M_\ast^\neg} \DD^nA \ot \int_{N_\ast^\neg}\DD^nA
\]
since negation commutes with wedging and factorization homology is symmetric monoidal. For the lefthand side, we have the equivalences
\[
\Bigl(\int_{M_\ast\vee N_\ast} A\Bigr)^\vee\simeq \Bigl(\int_{M_\ast}A\ot\int_{N_\ast} A\Bigr)^\vee
\simeq
\Bigl(\Bigl(\int_{M_\ast^{\neg}}\DD^nA\Bigr)^\vee\ot\Bigl(\int_{N_\ast^\neg} \DD^nA\Bigr)^\vee\Bigr)^\vee
\]
again using that factorization homology is symmetric monoidal, as well as the previously established case of the theorem for $M_\ast$ and $N_\ast$ individually. We then have the further equivalences using the proof of equivalence (4),
\[
\Bigl(\Bigl(\int_{M_\ast^{\neg}}\DD^nA\Bigr)^\vee\ot\Bigl(\int_{N_\ast^\neg} \DD^nA\Bigr)^\vee\Bigr)^\vee
\simeq
\Bigl(\int_{M_\ast^{\neg}}\DD^nA\Bigr)^{\vee\vee}\ot\Bigl(\int_{N_\ast^\neg} \DD^nA\Bigr)^{\vee\vee}
\simeq
\int_{M_\ast^{\neg}}\DD^nA\ot\int_{N_\ast^\neg} \DD^nA~,
\]
using the coconnectivity and degreewise finiteness of $(\int_{M_\ast^\neg} \DD^nA)^\vee$, when $A$ is finitely presented $(-n)$-coconnective, to commute duals and tensor products. 
\end{proof}

\subsection{Comparing $\Artin_n$ and $\FPres_n$}
We show that Koszul duality implements an equivalence between finitely presented $(-n)$-coconnective algebras and connective Artin algebras.
These connectivity conditions generalize classical considerations in differential graded algebra, evident in the connectivity conditions in works such as \cite{moore}, \cite{milnormoore}, and \cite{quillen}.
It is well known that, in general, a free coalgebra is more complicated than a free algebra due to the failure of tensor products to commute with infinite products.
This failure means that the infinite product $\prod_{i\geq 0}V^{\ot i}$ lacks a coalgebra structure in general.
However, the infinite direct sum $\bigoplus_{i\geq 0}V^{\ot i}$ does have a natural coalgebra structure: it is the free ind-nilpotent coalgebra generated by $V$.
Consequently, whenever the natural completion map
\[\bigoplus_{i\geq 0} V^{\ot i} \longrightarrow \prod_{i\geq 0} V^{\ot i}\]
is an equivalence, then the infinite product does obtain a natural coalgebra structure; this is then easily seen to be free.
This completion map is an equivalence under (co)connectivity hypotheses. For instance, if $V$ is 1-connective, then $V^{\ot i}$ is $i$-connective, from which the equivalence follows.
Given the calculation of the bar construction on an augmented trivial associative algebra, $\bBar({\sf t}V)\simeq \bigoplus_{i\geq 0} (\Sigma V)^{\ot i}$, one can summarize this discussion for Koszul duality purposes as saying that the bar construction sends augmented trivial algebras to ind-nilpotent free coaugmented coalgebras, and it sends sends trivial {\it connective} algebras to free coalgebras (which happen to ind-nilpotent).

The following lemma generalizes the preceding to $n$-disk algebra.

\begin{lemma}\label{trivial-free} Let $V\in \m_{\sO(n)}\bigl(\Perf_\Bbbk^{\geq 0}\bigr)$ be an $\sO(n)$-module in perfect connective $\Bbbk$-modules, where $\Bbbk$ is a field. The Koszul dual of the trivial algebra on $V$, $\DD^n (\Bbbk\oplus V)\simeq \bBar({\sf t}V)^\vee$, is equivalent to the free augmented $n$-disk algebra on $V^\vee[-n]$. Likewise, $\bBar\bigl({\sf t} V\bigr)$ is the free augmented $n$-disk coalgebra on $(\RR^n)^+\ot V$. 
\end{lemma}
\begin{proof}
The argument is essentially that for Lemma~\ref{free-trivial}, but replacing the Goodwillie cofiltration with the cardinality filtration. For convenience, in this proof we will denote $V[n]:= (\RR^n)^+\ot V$ and $V[-n]: = \Map\bigl((\RR^n)^+,V\bigr)$. 
The reader should bear in mind that these notations carry the $\sO(n)$-action.

Consider the cardinality cofiltration for factorization cohomology with coefficients in the trivial coalgebra ${\sf t}V^\vee$.
There is an equivalence
\[
\int_{(\RR^n)^+}\FF(V^\vee[-n])  ~  \simeq  ~   {\sf t}V^\vee~.
\]
Consequently, by Theorem \ref{compare-towers}, we have for each $M_\ast$ an equivalence between the Goodwillie cofiltration of factorization homology and the cardinality cofiltration of factorization cohomology
\[
P_{\leq \bullet}\int_{M_\ast^\neg}\FF(V^\vee[-n])~  \simeq   ~    \tau^{\leq \bullet} \int^{M_\ast}{\sf t}V^\vee~.
\]
The 0-connectivity of $V$ implies that $V^\vee[-n]$ is $(-n)$-coconnective.
It follows from this that each fiber of the Goodwillie cofiltration, $\conf^{\fr}_i(M_\ast^\neg) \underset{\Sigma_i\wr \sO(n)}\ot V^\vee[-n]^{\ot i}$, is $(1-i)$-coconnective, since $\conf_i(M_\ast^\neg)$ has nonvanishing homology in degrees at most $n+(n-1)(i-1)$, each of which is finite dimensional.
Consequently, the Goodwillie cofiltration converges.
Using the comparison of the Goodwillie cofiltration and the cardinality cofiltration (Theorem~\ref{compare-towers}), since the Goodwillie cofiltration for a free algebra splits (Theorem~\ref{free-calculation}), we obtain that the cardinality cofiltration for factorization cohomology of a trivial coalgebra also splits.
Applying this result in the case of $M_\ast =(\RR^n)^+$ and $M_\ast^\neg =\RR^n_+$, we obtain the equivalence
\[
\FF(V^\vee[-n])~   \simeq~   \int^{(\RR^n)^+}{\sf t}V^\vee~.
\]
As established, the finiteness condition on $V$ gives that this $\Bbbk$-module is concentrated in negative homological degrees and is finite in each dimension.
Proposition~\ref{switch} therefore gives the equivalence
\[
\int^{(\RR^n)^+}{\sf t}V^\vee~ \simeq ~\Bigl(\int_{(\RR^n)^+}\Bbbk\oplus V\Bigr)^\vee~.
\]
The righthand side is exactly the definition of $\DD^n(\Bbbk\oplus V)$, the Koszul dual of the trivial $n$-disk algebra on $V$, so the result follows.
\end{proof}

\begin{remark}
For convenience, we have assumed $\Bbbk$ is a field, but this was only necessary for the dual $\DD^n(\Bbbk\oplus V)$ to be a free algebra. The coalgebra $\bBar({\sf t}V)$ is free provided $V$ and $\Bbbk$ are connective as complexes, without requirement that $\Bbbk$ is a field.

\end{remark}

We now have the following important duality result.
\begin{theorem}\label{fpres} Koszul duality restricts to a contravariant equivalence 
\[
\DD^n\colon \fpres^{\leq -n}_n ~\simeq ~(\Artin_n)^{\op}\colon \DD^n
\] 
between finitely presented $(-n)$-coconnective augmented $n$-disk algebras and Artin $n$-disk algebras in chain complexes over a field $\Bbbk$. 

\end{theorem}

\begin{proof} 
Lemma~\ref{free-trivial} shows that $\DD^n$ sends free algebras to trivial algebras.  
Inspecting further, $\DD^n$ restricts as a functor $\free_n^{{\sf perf},\leq -n} \to \Triv_n^{\geq 0}$ -- this uses that the $\Bbbk$-linear dual of a coconnective object is connective, since $\Bbbk$ is a field. 
Because $\DD^n = (\bBar)^\vee\colon \Alg_n^{\sf aug} \to (\Alg_n^{\sf aug})^{\op}$ is the composite of left adjoints, it preserves colimits.  
Inspecting the definitions of $\fpres_n^{\leq -n}$ and $\Artin_n$, we conclude that $\DD^n \colon \fpres_n^{\leq -n} \to (\Alg_n^{\sf aug})^{\op}$ canonically factors through $(\Artin_n)^{\op}$:
\begin{equation}\label{DD}
\DD^n \colon \fpres_n^{\leq -n} \longrightarrow (\Artin_n)^{\op}~.
\end{equation}
We will now explain that this restricted functor is fully faithful.  

For each pair of objects $A,B\in\fpres_n^{\leq -n}$ we must show that map of spaces
\begin{equation}\label{D-on-maps}
\Map_{\Alg_n^{\sf aug}}(A,B) \xra{~\DD^n~} \Map_{\Alg_n^{\sf aug}}(\DD^nB,\DD^nA)
\end{equation}
is an equivalence.  
Again, because $\DD^n$ preserves colimits, it is enough to consider the case that $A = \FF V$ is free on a $\sO(n)$-module $V$ in $\perf_\Bbbk^{\leq -n}$.  
For this case will explain the canonical factorization of the map~(\ref{D-on-maps}) through equivalences:
\begin{eqnarray}
\nonumber
\Map_{\Alg_n^{\sf aug}}(\FF V,B)  
&
\underset{(1)}\simeq
&
\Map^{\sO(n)}(V, B)
\\
\nonumber
&
\underset{(2)}\simeq
&
\Map^{\sO(n)}\bigl(((\RR^n)^+\ot B)^\vee , ((\RR^n)^+\ot V)^\vee\bigr)
\\
\nonumber
&
\underset{(3)}\simeq
&
\Map^{\sO(n)}\Bigl(\bigl((\RR^n)^+ \ot \bBar(\DD^n B)\bigr)^\vee, ((\RR^n)^+\ot V)^\vee\Bigr)
\\
\nonumber
&
\underset{(4)}\simeq
&
\Map^{\sO(n)}\Bigl(L\DD^n B, ((\RR^n)^+\ot V)^\vee\Bigr)
\\
\nonumber
&
\underset{(5)}\simeq
&
\Map_{\Alg_n^{\sf aug}}\Bigl(\DD^nB, {\sf t}\bigl(((\RR^n)^+\ot V)^\vee\bigr)\Bigr)
\\
\nonumber
&
\underset{(6)}\simeq
&
\Map_{\Alg_n^{\sf aug}}(\DD^n B, \DD^n \FF V)
\end{eqnarray}
The equivalence~(1) is by the free-forgetful adjunction.
The equivalence~(2) is an application of the functor $\bigl((\RR^n)^+\ot - \bigr)^\vee \colon \Mod_{\sO(n)}(\Perf_\Bbbk) \to \Mod_{\sO(n)}(\Perf_\Bbbk)$, which is an equivalence, because $\Ch_\Bbbk$ is stable and linear dual implements an equivalence on finite chain complexes.
The equivalence~(3) is because of the canonical equivalence $B\simeq \DD^n\DD^n B$, which is Theorem~\ref{affine} applied to the case of $M_\ast = \RR^n_+$ and $M_\ast^\neg = (\RR^n)^+$.  
The equivalence~(4) is from Theorem \ref{L-bar}.
The equivalence~(5) is the cotangent space-trivial adjunction.
The equivalence~(6) is the identification of the dual of a free algebra from Lemma~\ref{free-trivial}.  
It is straightforward to check that this composite equivalence agrees with the map~(\ref{D-on-maps}).

With Lemma~\ref{free-trivial}, inspection of the definitions of $\fpres_n^{\leq -n}$ and $\Artin_n$ gives that this restricted functor~(\ref{DD}) is essentially surjective. 
We conclude that~(\ref{DD}) is an equivalence of $\infty$-categories.  
From Lemma~\ref{trivial-free}, $(\DD^n)^{\op}$ restricts as an inverse to $\DD^n$ on $\Triv_n^{\geq 0}$.
Because the forgetful functor $\Alg_n^{\sf nu}(\Ch_\Bbbk) \to \Mod_{\sO(n)}(\Ch_\Bbbk)$ is conservative, it follows that $(\DD^n)^{\op}\colon (\Artin_n)^{\op} \to \fpres_n^{\leq -n}$ is inverse to $\DD^n$.  
\end{proof}

\begin{cor}
For $A\in \FPres_n^{\leq -n}$, the moduli problem $\MC_{A}$ is affine. Moreover, there is an equivalence
\[
\MC_{A} ~\simeq~ {\sf Spf}(\DD^n A)~.
\] 
\end{cor}

\subsection{Resolving by $\FPres_n^{\leq -n}$}

The rest of this section is devoted to proving the last required result used in the proof of our main theorem.
This is Proposition~\ref{resolve}, which states that factorization homology with general coefficients is given by left Kan extension of factorization homology restricted to finitely presented augmented $n$-disk algebras whose augmentation ideal is $(-n)$-coconnective. In order to establish this, we make use of the following notion of one $\oo$-category being generated under colimits by another.

\begin{definition}[Strongly generates]\label{def.stronggenerates}
A functor $g\colon \cB\ra \cC$ between $\infty$-categories {\it strongly generates} if the diagram among $\infty$-categories
\[
\xymatrix{
\cB  \ar[rr]^-g  \ar[d]_-g
&&
\cC  
\\
\cC  \ar[urr]_-{\id_\cC}
}
\]
exhibits the functor $\id_\cC$ as a left Kan extension of $g$ along $g$.
In other words, the natural transformation $\LKan_g(g) \xra{\simeq}\id_\cC$ between endofunctors on $\cC$ is by equivalences.  
\end{definition}

The following lemma is the main technical result of this section. The completion of its proof occupies~\S\ref{sec.gen} below.
Recall the full $\infty$-subcategories of $\Alg_n^{\sf aug}$ from Notation~\ref{def.alg-i}.

\begin{lemma}[Highly coconnected free resolutions]\label{free-resolution}
The inclusion from the $\infty$-category of free augmented $n$-disk algebras on $\sO(n)$-modules, whose underlying $\Bbbk$-module is $(-n)$-coconnective and finite, into augmented $n$-disk algebras in $\Bbbk$-modules,
\[
\Free^{\leq -n}_{n}~\hookrightarrow~ \Alg_n^{\sf aug}~,
\] 
strongly generates.
\end{lemma}
\begin{proof}
Because composition of left Kan extensions along composable functors is the left Kan extension along the composite, it suffices to show for that, for any full $\oo$-subcategory $\cF\subset \Free^{\leq -n}_n$, the inclusion $g_{|}\colon \cF\hookrightarrow \Alg_n^{\sf aug}$ strongly generates.
We do this for $\cF = \Free^{{\sf perf}, \leq -n}_n$, the $\oo$-category consisting of augmented $n$-disk algebras that are free on $(-n)$-coconnective truncations of perfect $\sO(n)$-modules. 
(Note, for example, that the trivial $\sO(n)$-module $\Bbbk[-n]\simeq \tau^{\leq -n}\sC_\ast(\sO(n),\Bbbk[-n])$ is the $(-n)$-coconnective truncation of a perfect $\sO(n)$-module, but it is {\it not} itself a perfect $\sO(n)$-module.)

We consider the sequence of left Kan extensions
\[
\xymatrix{
\Free^{{\sf perf}, \leq -n}_n\ar[rr]^{g_2\circ g_1\circ g_0}\ar[d]_{g_0}&&\Alg_n^{\sf aug}\\
\Free^{{\sf all},\leq -n}_n\ar[d]_{g_1}\ar@{-->}[urr]\\
\Alg_n^{\leq -n}\ar[d]_{g_2}\ar@{-->}[uurr]\\
\Alg_n^{\sf aug}\ar@{-->}[uuurr]}
\]
where each $g_i$ a fully faithful inclusion with $\Alg_n^{\leq -n}$ consisting of those augmented $n$-disk algebras whose augmentation ideal is a $(-n)$-coconnective $\Bbbk$-module.
Here $\Free^{{\sf all},\leq -n}_n$ consists of all free augmented $n$-disk algebras on $(-n)$-coconnective $\Bbbk$-modules.
To prove the canonical natural transformation
\[
\lkan_g(g) \longrightarrow  {\sf id}_{\Alg}
\]
is an equivalence it is sufficient to show that each of the canonical natural transformations
\begin{equation}\label{eq.0}
 \lkan_{g_0}(g_2\circ g_1\circ g_0)\ra g_2\circ g_1~,
\end{equation}

\begin{equation}\label{eq.1} \lkan_{g_1}(g_2\circ g_1) \ra g_2~,
\end{equation}
and
\begin{equation}\label{eq.2} \lkan_{g_2}(g_2) \ra {\sf id}_{\Alg}
\end{equation}
are equivalences.

We first prove that the arrow~(\ref{eq.0}) is an equivalence, namely that the natural map
\[
\colim_{\FF V\in \Free^{{\sf perf}, \leq -n}_{n/A}}\FF V \longrightarrow A
\]
is an equivalence for any $A\simeq \FF(W_0)$ a free $(-n)$-coconnective $n$-disk algebra.
Since $\m_{\sO(n)}$ is a compactly generated by the perfect $\sO(n)$-modules, we can choose an expression for $W_0$ as filtered colimit $W_\bullet:\cJ \ra \perf_{\sO(n)}$ of perfect $\sO(n)$-modules.
Since the truncation functor $\tau^{\leq -n}$ preserves colimits, we obtain an expression
\[
W_0\simeq \tau^{\leq -n}W_0 \simeq \colim_{\cJ} \tau^{\leq-n}W_\bullet
\]
for $W_0$ as a filtered colimit of $(-n)$-coconnective truncations of perfect $\sO(n)$-modules.
Applying the free functor $\FF$, which again preserves filtered colimits, we obtain a filtered diagram $A_\bullet :=\FF(\tau^{\leq -n}W_\bullet): \cJ\ra \Free^{{\sf perf}, \leq -n}_n$ together with an identification $\colim\bigl( \cJ \xra{g_2\circ g_1\circ g_0\circ A_\bullet} \Alg_n^{\leq -n}\bigr) \simeq A$.
Recall that the inclusion $\m_{\sO(n)}^{\leq -n} \hookrightarrow \m_{\sO(n)}$ preserves filtered colimits.
Therefore if $V$ is a compact of $\m_{\sO(n)}$ (i.e., a perfect $\sO(n)$-module) then the truncation $\tau^{\leq -n}V$ is a compact object of $\m_{\sO(n)}^{\leq -n}$.
(Note that the truncation need not be a compact object of $\m_{\sO(n)}$. For instance, the trivial module $\Bbbk[-n] \simeq \tau^{\leq -n} \sC_\ast(\sO(n),\Bbbk[-n])$ is compact as an object of $\m_{\sO(n)}^{\leq -n}$ but not compact as an object of $\m_{\sO(n)}$.)
Consequently, every object of $\Free^{{\sf perf},\leq -n}_n$ is a compact object of $\Alg_n^{\leq -n}$, which is to say that the map $\underset{j\in \cJ} \colim \Map(\FF V, A_j) \xra{\simeq} \Map(\FF V,A)$, involving spaces of morphisms in $\Alg_n^{\leq -n}$, is an equivalence. 
Consequently, we have that the natural functor between $\infty$-categories
\[
\colim_{j\in \cJ}\Free^{{\sf perf}, \leq -n}_{n/A_j} \longrightarrow \Free^{{\sf perf}, \leq -n}_{n/A}
\]
is an equivalence.
The result is now a formal consequence of commuting colimits:
\[
\colim_{\FF V\in \Free^{{\sf perf}, \leq -n}_{n/A}}\FF V
~ \simeq ~
\colim_{\FF V\in \underset{j\in\cJ}\colim \Free^{{\sf perf}, \leq -n}_{n/A_j}}\FF V
~ \simeq ~
\colim_{j\in\cJ} \colim_{\FF V\in\Free_{n/A_j}^{{\sf perf}, \leq -n}} \FF V
~\simeq ~
\colim_{j\in \cJ} A_j\simeq A~.
\]

We next prove that the arrow (\ref{eq.1}) is an equivalence. That is, we show that the canonical morphism
\[
\LKan_{g_2}(g_2\circ g_1)(A)~\simeq~\underset{\FF  V\in\Free^{\leq -n}_{n/A}}\colim \FF V \longrightarrow A
\]
is an equivalence for all $A\in \Alg_n^{\leq -n}$.
This arrow fits as the lower horizontal arrow a solid diagram among augmented $n$-disk algebras
\[
\xymatrix{
\underset{\FF  V\in\Free^{\leq -n}_{n/A}}\colim~ \underset{\bullet\in \bDelta^{\op}}\colim ~ \FF^{\bullet+2} V  \ar[d]_-{\simeq}  \ar[rr]   
&&
\underset{\bullet\in \bDelta^{\op}}\colim ~\FF^{\bullet+1} \Ker(A\to \Bbbk)  \ar[d]^-{\simeq}  \ar@{-->}[dll]
\\
\underset{\FF  V\in\Free^{\leq -n}_{n/A}}\colim \FF V  \ar[rr] 
&&
A,
}
\]
which we now establish by way of functorial free resolutions.  
Namely, for each augmented $n$-disk algebra $B$ there is a simplicial object
\[
\FF^{\bullet+1}\Ker(B\to \Bbbk)\colon \bDelta^{\op} \to (\Free_n^{\leq -n})_{/B}
\]
given by the functorial free resolution of $B$; it has the property that the canonical morphism from its colimit 
\[
|\FF^{\bullet+1}\Ker(B\to \Bbbk)| \xra{~\simeq~}B
\]
is an equivalence.  
This simplicial object is evidently functorial in the argument $B$, which is to say that it defines a functor $\FF^{\bullet+1}\Ker(-\to \Bbbk)\colon \Alg_n^{\leq -n} \to \Fun(\bDelta^{\op},(\Cat_{\infty})_{/\Free_n^{\leq -n}})$ to the $\infty$-category of simplicial objects in $\infty$-categories over $\Free_n^{\leq -n}$.  
Taking colimits establishes the square diagram and its commutativity, as well as the fact that the vertical arrows are equivalences.
This also establishes the dashed arrow making the filled diagram manifestly commute.

Lastly, that (\ref{eq.2}) is an equivalence is exactly the assertion that the inclusion $\Alg_n^{\leq -n} \hookrightarrow \Alg_n^{\sf aug}$ strongly generates.
This assertion is precisely Corollary \ref{cor.atlast}.
\end{proof}

\begin{lemma}\label{coconn-sifted}
Let $A$ be an augmented $n$-disk algebra in $\Ch_\Bbbk^\ot$.  
Consider the $\infty$-category $(\free^{\leq -n}_{n})_{/A}$ of augmented $n$-disk algebras over $A$ which are free on $(-n)$-coconnective $\sO(n)$-modules in finite chain complexes over $\Bbbk$. 
This $\infty$-overcategory
\[
(\free^{\leq -n}_{n})_{/A}
\]is sifted.
\end{lemma}
\begin{proof}
First, note that this $\infty$-category $(\free^{\leq -n}_{n})_{/A}$ is nonempty.  
We show the diagonal map $(\free^{{\sf perf},\leq -n}_{n})_{/A} \ra (\free^{{\sf perf},\leq -n}_{n})_{/A}\times (\free^{{\sf perf},\leq -n}_{n})_{/A}$ is final.
For this, we use Quillen's Theorem A and argue that the iterated slice $\infty$-category
\[
\bigl((\free_n^{{\sf perf},\leq -n})_{/A}\bigr)_{(\FF V \to A , \FF W \to A)/}
\]
has contractible classifying space for each pair of objects $(\FF V\to A)$ and $(\FF W\to A)$ of $(\free_n^{\leq -n})_{/A}$.
This iterated slice $\infty$-category has an initial object, given by the universal arrow $\FF(V\oplus W) \to A$ determined by the map of underlying $\sO(n)$-modules $V \oplus W \to \FF V \oplus \FF W \to A \oplus A \to A$.
\end{proof}

Through Lemma~\ref{sifted-resolutions}, the combination of Lemma~\ref{free-resolution} and Lemma~\ref{coconn-sifted} gives the following result.
\begin{cor}\label{free-res-fact}
For $A$ an augmented $n$-disk algebra in $\Ch_\Bbbk$, the canonical arrow 
\[
\colim_{\FF V\in (\free^{\leq -n}_{n})_{/A}}\int_{M_\ast} \FF V \xra{~\simeq~} \int_{M_\ast}A
\] 
is an equivalence in $\Ch_\Bbbk$.  

\end{cor}

\begin{prop}\label{resolve} 
Let $A$ be an augmented $n$-disk algebra in $\Ch_\Bbbk$.
The canonical arrow
\[
\colim_{F\in (\fpres^{\leq -n}_{n})_{/A}}\int_{M_\ast} F \xra{~\simeq~} \int_{M_\ast}A
\] 
is an equivalence in $\Ch_\Bbbk$.

\end{prop}

\begin{proof} 
We explain the sequence of equivalences in $\Ch_\Bbbk$:
\begin{equation}\label{free-fp-A}
\underset{\FF V\in (\free_n^{\leq -n})_{/A}}\colim \int_{M_\ast} \FF V \longrightarrow \underset{F \in (\fpres_n^{\leq -n})_{/A}}\colim\int_{M_\ast} F \longrightarrow \int_{M_\ast} A~.
\end{equation}
The arrows are the canonical ones, from restricting colimits.
Corollary~\ref{free-res-fact} states that the composite arrow is an equivalence.  
The result is verified upon showing that the left arrow is an equivalence.

Consider the commutative triangle of $\infty$-categories
\[
\xymatrix{
(\free_n^{\leq -n})_{/A}  \ar[rr]^-{\int_{M_\ast}}   \ar@{^{(}->}[d]
&&
\Ch_\Bbbk
\\
(\fpres_n^{\leq -n})_{/A}  \ar[rru]_-{\int_{M_\ast}} 
&&
.
}
\]
Let $F \to A$ be an object of $(\fpres_n^{\leq -n})_{/A}$.
From Corollary~\ref{free-res-fact}, the canonical map 
\[
\underset{\FF V\in (\free_n^{\leq -n})_{/F}}\colim \int_{M_\ast} \FF V \longrightarrow \int_{M_\ast} F
\] 
is an equivalence.
It follows that the above triangle is a left Kan extension, and thereafter that the left arrow of~(\ref{free-fp-A}) is an equivalence.  
\end{proof}

\subsubsection{\bf Conditions for strong generation}\label{sec.gen}
The following sequence of results, leading up to Corollary~\ref{cor.atlast}, is toward establishing the remaining step in the proof of Lemma~\ref{free-resolution} above, that the inclusion $\Alg^{\leq -n}_n\hookrightarrow \Alg_n^{\sf aug}$ strongly generates in the sense of Definition~\ref{def.stronggenerates}.

The next formal results gives examples of functors that strongly generate.

\begin{observation}\label{yon-gen}
For each $\infty$-category $\cC$, the Yoneda functor $j\colon \cC\to \Psh(\cC)$ strongly generates.  
This follows from the universal property of the Yoneda functor being the free colimit completion.
For instance, the canonical natural transformation $\LKan_j(j) \to \id$ is by equivalences since its value on a presheaf $\cF$ on $\cC$ is the standard equivalence
\[
\colim\bigl(\cC_{/\cF} \to \cC \xra{j} \Psh(\cC)\bigr) \simeq \underset{(jc \to \cF)\in \cC_{/\cF}} \colim jc \xra{\simeq~} \cF~.
\]

\end{observation}

\begin{prop}\label{prop.finalgen}
Each localization $\cB\ra \cC$ between $\infty$-categories strongly generates.  
\end{prop}
\begin{proof}
Let $g\colon \cB \to \cC$ be a localization.
We must show that the canonical natural transformation $\LKan_g(g)  \to \id_\cC$ is by equivalences in $\cC$.  
From the standard formula for the values of a left Kan extension, this is to show that, for each object $c\in \cC$, the morphism in $\cC$
\[
\colim\bigl(\cB_{/c} \xra{g} \cC_{/c} \to \cC\bigr) \longrightarrow \colim\bigl(\cC_{/c} \to \cC\bigr)\simeq c
\]
is an equivalence.  
To show this it is sufficient to show that the functor $\cB_{/c} \to \cC_{/c}$ is final for each object $c\in \cC$.  
Through Quillen's Theorem~A, this is the problem of showing that, for each morphism $c'\xra{f} c$ in $\cC$, the iterated slice $\infty$-category $(\cB_{/c})^{f/}$ has contractible classifying space: $\sB\bigl((\cB_{/c})^{f/})\bigr)\simeq \ast$.  

Let $c'\xra{f} c$ be a morphism in $\cC$.  
Because $\cB \xra{g} \cC$ is a localization, the restriction functor $\Fun(\cC,\Spaces) \xra{g^\ast} \Fun(\cB,\Spaces)$ is fully faithful. 
Therefore, for each functor $\cC\xra{\cF} \Spaces$, the value of the counit of the left Kan extension-restriction adjunction 
\[
g_!  g^\ast \cF \xra{~\simeq~} \cF
\]
is an equivalence in $\Fun(\cC,\Spaces)$.  
Applying this to the representable copresheaf $\cF:= \cC(c',-)$, then evaluating on $c$, we obtain an equivalence of spaces
\[
\sB\Bigl(\cB_{/c}\underset{\cC}\times \cC^{c'/} \Bigr)~\simeq~\colim\bigl(\cB_{/c} \to \cC\xra{\cC(c',-)} \Spaces\bigr) ~\simeq~g_! g^\ast \cC(c',-)(c)\xra{~\simeq~} \cC(c',c)
\]
where the second equivalence is the standard formula computing the values of a left Kan extension and the first equivalence is through the straightening-unstraightening construction of~\S2.2 of~\cite{HTT}.  
Quillen's Theorem A identifies the fiber over $f\in \cC(c',c)$ of this map as the classifying space $\sB\bigl((\cB_{/c})^{f/})\bigr)$.  
We conclude that this classifying space is contractible, as desired.  
\end{proof}

\begin{remark}
A functor $\cB\ra \cC$ that strongly generates is final with respect any functor $\cC \ra \cE$ which preserves colimits, in that there is a natural equivalence $\colim(\cC\ra \cE)\simeq \colim(\cB\ra \cE)$.
Strong generation of $\cB \ra \cC$ is thus a weakening of requiring both the conditions that $\cB$ generate $\cC$ under colimits {\it and} that the functor $\cB \ra \cC$ being final.
\end{remark}

\begin{lemma}\label{ff-gen}
A functor $g:\cB \ra \cC$ between $\infty$-categories strongly generates if and only if the restricted Yoneda functor
\[
\cC\longrightarrow \PShv(\cB)
\]
is fully faithful.
\end{lemma}
\begin{proof}
By definition, $g$ strongly generates if and only if the canonical natural transformation 
\[
\LKan_{g}(g) \longrightarrow  \id_\cC
\]
between endofunctors on $\cC$ is by equivalences.
This is to say that, for each object $c\in \cC$, the canonical morphism $\LKan_{g}(g)(c) \to c$ in $\cC$ is an equivalence.
Via the standard formula computing the values of left Kan extensions, this is to say that, for each object $c\in \cC$, the canonical morphism 
\[
\colim\bigl(\cB_{/c} \longrightarrow  \cB \xra{g} \cC\bigr) \to c
\]
in $\cC$ is an equivalence.  
By the defining universal property of colimits in an $\infty$-category, this is equivalent to the assertion that, for each pair of objects $c,c'\in \cC$, the canonical map between spaces
\[
\cC(c,c') \longrightarrow \limit\bigl( (\cB_{/c})^{\op} \to \cB^{\op} \xra{g^{\op}} \cC^{\op} \xra{\cC(-,c')} \Spaces \bigr)
\]
is an equivalence.  
Through the straightening-unstraightening equivalence (\S2.2 of~\cite{HTT}), the righthand limit space is canonically identified an end; namely, as the space of functors over $\cB$ from $\cB_{/c}$ to $\cB_{/c'}$.  
As so, the above map of spaces is canonically identified as the map on morphism spaces
\[
\cC(c,c') \longrightarrow \Map_{/\cB}\bigl(\cB_{/c},\cB_{/c'}\bigr)~\simeq~\Map_{\Psh(\cB)}\bigl(\cC(g-,c),\cC(g-,c')\bigr)
\]
induced by the restricted Yoneda functor $\cC \to \Psh(\cB)$.  
We conclude that $g$ strongly generates if and only if this restricted Yoneda functor is fully faithful.
\end{proof}

\begin{lemma}\label{select}
Let $g\colon \cC_0 \hookrightarrow \cC$ be a fully faithful functor between $\infty$-categories, and suppose $\cC$ is presentable.
Let $\cK$ be an $\infty$-category whose classifying space $\sB\cK\simeq\ast$ is terminal.
If the composite functor $\cC_0^{\cK} \xra{g^\cK} \cC^{\cK} \xra{\colim}\cC$ is a localization, then $g$ strongly generates.

\end{lemma}

\begin{proof}

Through Lemma~\ref{ff-gen}, we seek to prove that the restricted Yoneda functor $\cC\to \Psh(\cC_0)$ is fully faithful.
This restricted Yoneda functor is accessible and preserves limits. 
Because $\cC$ is assumed presentable, this restricted Yoneda functor admits a left adjoint $L\colon \Psh(\cC_0) \to \cC$.
Through universal properties of presheaf $\infty$-categories as free colimit completions, this left adjoint fits into a commutative diagram among $\infty$-categories
\[
\xymatrix{
\cC_0^{\cK}  \ar[r]^-{g^\cK}  \ar[dr]_-{j^\cK}
&
\cC^\cK  \ar[r]^-{\colim}  
&
\cC
\\
&
\Psh(\cC_0)^\cK  \ar[r]^-{\colim}
&
\Psh(\cC_0)  \ar[u]_-{L}
}
\]
in which $j\colon \cC_0 \to \Psh(\cC_0)$ denotes the Yoneda functor.
The restricted Yoneda functor $\cC\to \Psh(\cC_0)$ being fully faithful is equivalent to its left adjoint $L$ being a localization.  
We prove that $L$ is a localization by proving the following two points.
\begin{enumerate}

\item 
The composite functor $\cC_0^\cK \xra{j^\cK} \Psh(\cC_0)^\cK \xra{\colim} \Psh(\cC_0)$ strongly generates.

\item 
Let
\[
\xymatrix{
\cX \ar[rr]^-{C}  \ar[dr]_-{J}  
&&
\cZ  
\\
&
\cY  \ar[ur]_-{L}
&
}
\]
be a diagram among $\infty$-categories.
Suppose both $\cY$ and $\cZ$ are presentable.
Suppose $C$ is a localization, $J$ strongly generates, and $L$ is accessible and preserves colimits.
Supposing these conditions, the functor $L$ is a localization.

\end{enumerate}
Our result follows by applying~(2) to the case of $\cX:=\cC_0^\cK$, $\cZ:= \cC$, and $\cY:=\Psh(\cC_0)$, with $C:=\colim \circ g^\cK$, $J:= \colim \circ j^\cK$, and $L:=L$.  

We prove~(1).  We must prove that the canonical natural transformation $\LKan_{\colim\circ  j^\cK}(\colim\circ  j^\cK) \to \id$ between endofunctors on $\Psh(\cC_0)$ is by equivalences.  
We explain that this natural transformation factors as the sequence of natural transformations
\begin{eqnarray}
\nonumber
\LKan_{\colim \circ j^\cK}(\colim \circ j^\cK)   
&
\xra{\rm (a)}
&
\LKan_{\colim}\bigl( \LKan_{j^\cK}(\colim \circ j^\cK)\bigr)
\\
\nonumber
&
\xra{\rm (b)}
&
\LKan_{\colim}\bigl(\colim\circ \LKan_{j^\cK}(j^\cK)\bigr)
\\
\nonumber
&
\xra{\rm (c)}
&
\LKan_{\colim}\bigl(\colim \circ \LKan_{j}(j)^\cK  \bigr)
\\
\nonumber
&
\xra{\rm (d)}
&
\LKan_{\colim}(\colim)
\\
\nonumber
&
\xra{\rm (e)}
&
\sB \cK \ot \id
\\
\nonumber
&
\xra{\rm (f)}
&
\id
\end{eqnarray}
each of which is by equivalences.
The universal property of left Kan extensions gives that a left Kan extension of a composition is an iteration of left Kan extensions, which gives the identity~(a).
The identity~(b) follows from the fact that, from the universal property of colimits, the colimit functor preserves colimits.
The identity~(c) is valid because colimits in functor categories are computed pointwise.  
The identity~(d) is valid because the Yoneda functor strongly generates (Observation~\ref{yon-gen}).
The identity~(e), involving tensoring with the classifying space of $\cK$, is verified from the standard expression computing the values of left Kan extensions.  
The identity~(f) invokes the assumption that the classifying space of $\cK$ is terminal.  
This completes the proof of~(1).

We now prove~(2).  
For $W:= L^{-1}(\cZ^\sim)\subset \cY$, we must verify that the canonical functor $\cY[W^{-1}] \to \cZ$ under $\cY$ is an equivalence.  
The assumption that $L$ preserve colimits implies this functor $\cY[W^{-1}] \to \cZ$ preserves colimits.  
The assumption that $J$ strongly generates implies composition with the localization $\cX \xra{J} \cY \to \cY[W^{-1}]$ also strongly generates.  
We are therefore reduced to the case that the functor $\cY \xra{L} \cZ$ is conservative; for this case the problem is to verify that this functor $\cY \xra{L} \cZ$ is an equivalence.

Invoking the universal property of localizations, conservativity of $L$ implies the existence of a unique functor $\cZ\xra{!} \cY$ under $\cX$.  
Because localizations are epimorphisms, the resulting composite functor $\cZ\xra{!} \cY \xra{L}\cZ$ under $\cX$ is canonically equivalent to the identity functor on $\cZ$; in symbols $L\circ ! \simeq \id_\cZ$. 
It remains to verify that the composite functor $\cY\xra{L}\cZ\xra{!} \cY$ is equivalent to the identity functor on $\cY$.  
To do this we explain the equivalences among endofunctors on $\cY$
\begin{eqnarray}
\nonumber
!\circ L
&
\xla{\rm (i)}
&
!\circ L \circ \LKan_J(J)
\\
\nonumber
&
\xla{\rm (ii)}
&
\LKan_J(!\circ L \circ J)
\\
\nonumber
&
\overset{\rm (iii)}\simeq
&
\LKan_J(!\circ L \circ !\circ C)
\\
\nonumber
&
\overset{\rm (iv)}\simeq
&
\LKan_J(!\circ C)
\\
\nonumber
&
\overset{\rm (v)}\simeq
&
\LKan_J(J)
\\
\nonumber
&
\xra{\rm (vi)}
&
\id_\cY~.
\end{eqnarray}
The identifications~(i) and~(vi) follow directly from the assumption that the functor $\cX\xra{J}\cY$ strongly generates.
The equivalences~(iii) and~(v) follow from the definition of the factorization $J\colon \cX \xra{C} \cZ\xra{!} \cY$.
The equivalence~(iv) follows directly from the equivalence $L\circ ! \simeq \id_\cZ$ argued above.

To establish the identity~(ii) it is enough to explain why the functor $!\circ L$ preserves colimits.
Because $L$ is assumed to do as much, it is enough to explain why the functor $\cZ\xra{!} \cY$ preserves colimits.  
That is, for each functor $\cI \to \cZ$, we must explain why the canonical morphism $\colim(\cI \to \cZ\xra{!} \cY) \to !\bigl(\colim(\cI \to \cZ)\bigr)$ in $\cY$ is an equivalence.  
Using that $L$ is conservative, it is enough to verify that the morphism $L\bigl(\colim(\cI \to \cZ\xra{!} \cY) \bigr) \to L\circ !\bigl(\colim(\cI \to \cZ)\bigr)$ in $\cZ$ is an equivalence.  
Using the assumption that $L$ preserves colimits, it is enough to verify that the morphism $\colim(\cI \to \cZ\xra{!}\cY \xra{L} \cZ) \to L\circ ! \bigl(\colim(\cI \to \cZ)\bigr)$ in $\cZ$ is an equivalence.  
This follows using the equivalence $L\circ ! \simeq \id_\cZ$ argued above.  
We conclude that $\cZ\xra{!}\cY$ preserves colimits, thereby completing the proof of this lemma.
\end{proof}

For the next two results, we make use of the following notation.
\begin{notation}\label{fun-W}
Let $\cC$ be an $\infty$-category and let $\cC^\sim\subset \cW\subset \cC$ be an $\infty$-subcategory containing the maximal $\infty$-subgroupoid.
For each $\infty$-category $\cK$, we denote the pullback
\[
\Fun_\cW(\cK,\cC)~:= \Fun(\cK^\sim,\cW)  \underset{\Fun(\cK^\sim,\cC)}\times \subset~\Fun(\cK,\cC)~,
\]
which is the $\infty$-subcategory of $\Fun(\cK,\cC)$ consisting of those natural transformations by $\cW$.  

\end{notation}

\begin{lemma}\label{on-localizations}
Let $\cC$ be an $\infty$-category and let $\cC^\sim\subset \cW\subset \cC$ be an $\infty$-subcategory containing the maximal $\infty$-subgroupoid.
Consider the simplicial space
\[
\sB\Fun_\cW([\bullet],\cC)
\]
which is the classifying space of the functor category of Notation~\ref{fun-W}.  
There is a canonical identification between the complete Segal localization of this simplicial space
\[
\sB\Fun_\cW([\bullet],\cC)^{\wedge}_{\rm cpt.Seg}~ \simeq~\cC[\cW^{-1}]
\]
and the complete Segal presentation of the $\infty$-categorical localization of $\cC$ by $\cW$.  

\end{lemma}

\begin{proof}
The is a routine argument among universal properties.
We use Rezk's complete Segal space presentation of $\oo$-categories: restriction along the standard functor $\bDelta\hookrightarrow \Cat_{\oo}$ induces a fully faithful embedding $\Cat_{\oo}\hookrightarrow \Fun(\bdelta^{\op},\spaces)$ whose essential image is characterized by the Segal and completeness localities.

We make two observations about the simplicial $\infty$-category $\Fun_\cW([\bullet],\cC)$,
both of which are manifest.
First, this simplicial $\infty$-category lies under the simplicial space $\Map([\bullet],\cC)\simeq \Fun_{(\cC)^\sim}([\bullet],\cC)$, which is the complete Segal presentation of $\cC$.  
Second, for each functor $\cC \to \cZ$ that carries each morphism in $\cW$ to an equivalence in $\cZ$, there is a unique simplicial functor $\Fun_{\cW}([\bullet],\cC) \to \Map([\bullet],\cZ)$ to the complete Segal presentation of $\cZ$ under the induced map between simplicial spaces $\Map([\bullet],\cC) \to \Map([\bullet],\cZ)$.  

Consider the simplicial space $\sB\Fun_{\cW}([\bullet],\cC)$, which is the object-wise classifying space. Thereafter, consider its complete and Segal localization, which we denote as $\sB\Fun_{\cW}([\bullet],\cC)^{\wedge}_{\rm cpt.Seg}$.
This complete Segal simplicial space presents an $\infty$-category, that we give the same notation, under $\cC$.
The second observation above, applied to the case $\cZ\simeq \cC[\cW^{-1}]$ under $\cC$, offers a canonical functor
\begin{equation}\label{ret}
\sB\Fun_{\cW}([\bullet],\cC)^{\wedge}_{\rm cpt.Seg}  \longrightarrow \cC[\cW^{-1}]
\end{equation}
under $\cC$.  
By construction, the composite functor $\cW \to  \cC \to \sB\Fun_{\cW}([\bullet],\cC)^{\wedge}_{\rm cpt.Seg}$ factors through the classifying space $\cW \to \sB\cW$.
As so, there is a unique functor 
\begin{equation}\label{sec}
\cC[\cW^{-1}] \longrightarrow \sB\Fun_{\cW}([\bullet],\cC)^{\wedge}_{\rm cpt.Seg}
\end{equation}
under $\cC$.
We assert that the functors~(\ref{ret}) and~(\ref{sec}) are mutual inverses, thereby demonstrating the equivalence 
\[
\sB\Fun_{\cW}([\bullet],\cC)^{\wedge}_{\rm cpt.Seg}
~\simeq~
\cC[\cW^{-1}]
\]
between $\infty$-categories under $\cC$.  
The composition of~(\ref{sec}) followed by~(\ref{ret}) is indeed equivalent to the identity functor since $\cC\to \cC[\cW^{-1}]$ is an epimorphism.
To see that the composition of~(\ref{ret}) followed by~(\ref{sec}) is indeed equivalent to the identity functor, apply the above discussion to $\cZ \simeq \sB\Fun_{\cW}([\bullet],\cC)^{\wedge}_{\rm cpt.Seg}$ under $\cC$; the result is a unique endofunctor of $\sB\Fun_{\cW}([\bullet],\cC)^{\wedge}_{\rm cpt.Seg}$ under $\cC$.  
\end{proof}

\begin{lemma}\label{theformalresult}
Let $\cC$ be an $\oo$-category that contains finite limits and geometric realizations, and such that finite limits commute with geometric realizations.
Let $\cC_0\subset \cC$ be a full $\oo$-subcategory.
The composite functor
\[
\cC_0^{\bDelta^{\op}}\longrightarrow \cC^{\bDelta^{\op}}\xra{~\colim~}\cC~,
\]
is a localization if the following condition is satisfied.
\begin{itemize}
\item For each simplicial object $A_\bullet\in \cC^{\bDelta^{\op}}$, there exists a simplicial object $F_\bullet \in \cC_0^{\bDelta^{\op}}$ and a map between simplicial objects in $\cC^{\bdelta^{\op}}$,
\[
F_\bullet \longrightarrow A_\bullet~,
\]
that induces an equivalence $|F_\bullet|\overset{\sim}\ra |A_\bullet|$ in $\cC$ between their colimits.
\end{itemize}
\end{lemma}
\begin{proof}
We apply Lemma~\ref{on-localizations} to the $\infty$-category $\cC_0^{\bdelta^{\op}}$ and the $\infty$-subcategory $\cW :=\colim^{-1}(\cC^{\sim})\subset \cC_0^{\bdelta^{\op}}$, which is the $\oo$-subcategory of $ \cC_0^{\bdelta^{\op}}$ consisting of those morphisms that become equivalences upon geometric realization in $\cC$.
The effect is that, to prove the result it suffices to show that, for each $[p]\in \bDelta$, the canonical functor
\[
\Fun_\cW([p],\cC_0^{\bdelta^{\op}})\ra \Map([p],\cC)
\]
induces an equivalence between the space of $p$-simplices of $\cC$ and the classifying space of $\Fun_\cW([p],\cC_0^{\bdelta^{\op}})$.
To do this it suffices to show that for each $p$-simplex $A:[p]\ra \cC$, the classifying space of the fiber over $A$
\[
\xymatrix{
\Fun_\cW([p],\cC_0^{\bdelta^{\op}})_{|A}\ar[r]\ar[d]
&\Fun_\cW([p],\cC_0^{\bdelta^{\op}})\ar[d]\\
\{A\}\ar[r]&\Map([p],\cC)\\}
\]
is contractible.

We prove this contractibility by showing that each such fiber $\infty$-category $\Fun_\cW([p],\cC_0^{\bdelta^{\op}})_{|A}$ is cofiltered; this is to say that each solid diagram among $\infty$-categories
\[
\xymatrix{
\cK   \ar[r]\ar[d]
&
\Fun_\cW([p],\cC_0^{\bdelta^{\op}})_{|A}
\\
\cK^{\tl}\ar@{-->}[ur]
&&
}
\]
can be filled, provided $\cK$ is finite -- here, $\cK^{\tl}$ denote the left-cone on $\cK$.  
Let $\cK$ be a finite $\infty$-category.  
We proceed by induction on $p$, and first consider the base case of $p=0$. 
Given such a $\cK$-point $G$, we can find a filler in the diagram among $\infty$-categories
\[
\xymatrix{
\cK\ar[rr]^-G    \ar[d]
&&
\Fun_\cW([0],\cC_0^{\bdelta^{\op}})_{|A}       \ar[d]
\\
\cK^{\tl}             \ar@{-->}[rr]_-{\underset{\cK}\limit \ G}
&&
\cC^{\bdelta^{\op}}\underset{\cC}\times\cC_{/A}
}
\]
defined as a limit diagram of the upper right composite.
This limit indeed exists since $\cC$ is assumed to admit finite limits. 
The assumption that geometric realizations in $\cC$ commute with finite limits grants that the canonical morphism in $\cC$ from the geometric realization $|\limit_\cK G| \to A$ in $\cC$ is an equivalence.
We conclude that the bottom horizontal arrow factors through the $\infty$-subcategory ${\cC^{\bdelta^{\op}}}_{|A} := \cC^{\bdelta^{\op}}\underset{\cC}\times \{A \xra{=}A\}\subset \cC^{\bdelta^{\op}}\underset{\cC}\times\cC_{/A}$.  
The bulleted condition applied to $A_\bullet := \limit_{\cK} G \in \cC^{\bdelta^{\op}}$ grants the existence of $F_\bullet \in \cC_0^{\bdelta^{\op}}$ and a map $F_\bullet \to \limit_{\cK} G$ that induces an equivalence on geometric realizations. 
This defines a extension
\[
\xymatrix{
\cK\ar[rr]^G\ar[d]&&\Fun_\cW([0],\cC_0^{\bdelta^{\op}})_{|A}\\
\cK^{\tl}\ar@{-->}_{F_\bullet \ra G}[urr]\\}
\]
where the functor $\cK^{\tl}\ra \cW_{|A}$ assigns the to the cone-point the aforementioned $F_\bullet$.

We now show the inductive step. 
Suppose $p>0$ and consider the inclusion $[p-1]\cong \{1<\ldots < p\}\subset [p]$. 
Let $A:[p]\ra \cC$ be a $p$-simplex of $\cC$, and choose an arbitrary $\cK$-point $G$ as before. 
We first argue the existence of a filler among the diagram of $\infty$-categories
\[
\xymatrix{
\cK      \ar[rr]^G   \ar[d]
&&
\Fun_\cW([p],\cC_0^{\bdelta^{\op}})_{|A}       \ar[d]
\\
\cK^{\tl}    \ar@{-->}[rr] 
&&
\Fun_\cW([p-1],\cC_0^{\bdelta^{\op}})_{|A_{|[p-1]}} 
\underset{\Fun_\cW([p-1],\cC^{\bdelta^{\op}})_{|A_{|[p-1]}}}\times 
\Fun_\cW([p], \cC^{\bdelta^{\op}})_{|A}~.
}
\] 
The projection to first factor, $\Fun_\cW([p], \cC^{\bdelta^{\op}})_{|A}$, is the limit diagram of the evident functor induced by the inclusion $\cC_0\hookrightarrow \cC$.
The projection to the first factor $\Fun_\cW([p-1],\cC_0^{\bdelta^{\op}})_{|A_{|[p-1]}}$ is by induction.
Inspecting the inductive construction of such fillers reveals that the resulting two projections to $\Fun_\cW([p-1],\cC^{\bdelta^{\op}})_{|A_{|[p-1]}}$ canonically agree, since finite limits are assumed to commute with geometric realizations.
This establishes the dashed functor making the diagram commute.

Restricting to the initial object of $\cK^{\tl}$ gives a functor $[p] \ra \cC^{\bdelta^{\op}}$ whose colimit is identified as $A$, and for which the restriction to $\{1<\ldots <p\}\subset [p]$ lies in the essential image of $\cC_0^{\bdelta^{\op}}$. 
Let $A^0_\bullet$ denote the restriction $\{0\}\subset [p]\xra{A_\bullet} \cC^{\bdelta^{\op}}$. 
We again apply the bulleted condition to find an object $F_\bullet\in \cC_0^{\bdelta^{\op}}$ and a morphism $F_\bullet \ra A_\bullet^0$ of simplicial objects in $\cC$ that induces an equivalence $|F_\bullet | \xra{\simeq} |A_\bullet^0|$ between geometric realizations.   
Concatenating with $F_\bullet$ gives the requisite lift $\cK^{\tl} \ra \Fun_\cW([p],\cC_0^{\bdelta^{\op}})_{|A}$.
\end{proof}

We now aim to check the bulleted condition of Lemma~\ref{theformalresult} for our case at hand.  
This is Lemma~\ref{secondcondition} below.  
We first establish a few intermediate results.

\begin{observation}\label{flim-geom-com}
Each of the $\infty$-categories 
\[
\Ch_{\Bbbk}\qquad \text{ and } \qquad  \m_{\sO(n)}(\Ch_{\Bbbk})\qquad \text{ and } \qquad \Alg_n^{\sf aug}
\]
has the property that finite limits commute with geometric realizations.

\end{observation}

\begin{proof}
Because $\Ch_{\Bbbk}$ is a stable $\infty$-category, finite limits commute with geometric realizations.
Because limits and colimits are computed object-wise in $\m_{\sO(n)}(\Ch_{\Bbbk}):= \Fun\bigl(\BO(n),\Ch_{\Bbbk})$, this $\infty$-category posesses this property as well.  
Because the forgetful functor $\Alg_n^{\sf aug} \to \m_{\sO(n)}(\Ch_{\Bbbk})$ preserves and creates limits as well as sifted colimits, the $\infty$-category $\Alg_n^{\sf aug}$ too posesses this property.  
\end{proof}

The given proof of the following lemma was suggested by the referee.

\begin{lemma}\label{homework}
Let $W$ be a finite $\sO(n)$-module over a field $\Bbbk$.
Let $W_\bullet \ra W$ be an augmented simplicial $\sO(n)$-module demonstrating a colimit.  
For any $N$, there exists a simplicial object $W'_\bullet$ in finite $(-N)$-coconnective $\sO(n)$-modules and a morphism of simplicial objects
\[
W'_\bullet \longrightarrow W_\bullet
\]
that induces an equivalence $|W'_\bullet| \simeq |W_\bullet| \simeq W$ between their colimits.
\end{lemma}

\begin{proof}
We will employ the $\oo$-categorical Dold--Kan correspondence, see \S1.2.4 of \cite{HA}. This asserts that for $\cC$ a stable $\oo$-category, there is an equivalence between simplicial objects and sequential objects
\[
\cC^{\bdelta^{\op}}\longrightarrow \cC^{\NN}
\]
which sends a simplicial object $X_\bullet$ to its skeletal filtration, $i \mapsto {\sf Sk}_iX_\bullet$. Let $W_\bullet \ra W$ be as above, and consider the map from the skeletal filtration ${\sf Sk}_\bullet W_\bullet \ra W$. Since $W$ is a compact object of $\m_{\sO(n)}$, the identity map $W\xra{\sim} \varinjlim_p {\sf Sk}_p W_\bullet$ lifts through some finite stage of the filtration, $W \ra {\sf Sk}_qW_\bullet$ for some sufficiently large $q$. Let $m$ be the maximum of $q$ and $N + \ell$, where $\ell$ is the largest degree in which $\sH_\ell W$ is nonzero. Given $m\geq q$, there then exists a map of sequential objects
\[
\xymatrix{
0 \ar[r]\ar[d]&\ldots\ar[r]& 0\ar[d]\ar[r]&W\ar[d]\ar[r]^=&W\ar[d]\ar[r]^=&\ldots\\
{\sf Sk}_0 W_\bullet\ar[r]&\ldots\ar[r]&{\sf Sk}_{m-1}W_\bullet \ar[r]& {\sf Sk}_{m}W_\bullet\ar[r]&{\sf Sk}_{m+1}W_\bullet \ar[r]&\ldots\\}
\]
which induces the equivalence $W \simeq \varinjlim_p {\sf Sk}_p W_\bullet$ on taking sequential colimits. Using the inverse equivalence in the $\oo$-categorical Dold--Kan correspondence, the top sequential object maps to $S^{m}_\bullet \ot W[-m]$, the tensor of the desuspension of $W$ with the simplicial set $S^{m}:= \Delta[m]/\partial\Delta[m]$. Note that this is a simplicial object in $(-N)$-coconnective chain complexes due to the inequality $m\geq N+\ell$. We thus obtain a map of simplicial objects
\[
S^{m}_\bullet \ot W[-m] \longrightarrow W_\bullet
\]
from an $(-N)$-coconnective simplicial object to $W_\bullet$ which induces an equivalence of colimits.
\end{proof}

\begin{lemma}\label{secondcondition}
For each simplicial object in $\Alg_n^{\sf aug}$, there exists a simplicial object  $F_\bullet$ in $\Alg_n^{\leq -n}$ and a morphism of simplicial objects
\[
F_\bullet \longrightarrow A_\bullet
\]
that induces an equivalence $|F_\bullet| \simeq |A_\bullet|$ between their colimits.
\end{lemma}
\begin{proof}
Let $\cM\subset \Alg_n^{\sf aug}$ be the full $\infty$-subcategory consisting of those augmented $n$-disk algebras $A$ with the following property.
\begin{itemize}
\item[~]
Let $N\geq n$.
Let $A_\bullet$ be a simplicial augmented $n$-disk algebra equipped with an identification $|A_\bullet|\simeq A$ of its geometric realization.
There exists a simplicial object $F_\bullet$ in augmented $n$-disk algebras whose augmentation ideal is $(-N)$-coconnective, together with a morphism $F_\bullet \ra A_\bullet$ of simplicial augmented $n$-disk algebras, that induces an equivalence $|F_\bullet|\simeq |A_\bullet|$ between their colimits. 
\end{itemize}
We prove that the inclusion $\cM\hookrightarrow \Alg_n^{\sf aug}$ is an equivalence between $\infty$-categories, which clearly implies the result. 

We prove this in two steps:
\begin{enumerate}
\item $\cM$ contains every free augmented $n$-disk algebra $\FF W$ generated by a finite $\sO(n)$-module $W$.
\item $\cM$ contains every augmented $n$-disk algebra.
\end{enumerate}

We prove~(1), that $\cM$ contains each free augmented $n$-disk algebra $\FF W$ on a finite $\sO(n)$-module. 
Given an augmented simplicial object $A_\bullet \ra \FF W$ in augmented $n$-disk algebras demonstrating a colimit, consider the augmented simplicial $\sO(n)$-module
\[
W_\bullet:= A_\bullet \underset{\FF W}\times W \longrightarrow W
\]
constructed by pulling back the augmented simplicial $\sO(n)$-module $A_\bullet \ra \FF W$ along the map between the augmentation $\sO(n)$-modules $W \ra \FF W$.   
Using that the forgetful functor $\Alg_n^{\sf aug} \to \m_{\sO(n)}(\Ch_{\Bbbk})$ preserves sifted colimits, 
Observation~\ref{flim-geom-com} gives that the augmented simplicial $\sO(n)$-module $W_\bullet \to W$ demonstrates a colimit.
By Lemma~\ref{homework} there exists a simplicial $(-N)$-coconnective and finite $\sO(n)$-module $W'_\bullet$ with a map $W'_\bullet \ra W_\bullet$ inducing an equivalence $|W'_\bullet| \xra{\simeq} |W_\bullet|\simeq W$ between their colimits, for any $N\geq n$.  
Applying the free-forgetful adjunction, we have a map of simplicial augmented $n$-disk algebras
\[
\FF(W'_\bullet)\longrightarrow A_\bullet~.
\]
Being a left adjoint, the free functor $\FF$ preserves geometric realizations.
Therefore, the induced map between colimits $|\FF(W'_\bullet)| \to |A_\bullet|\simeq \FF W$ is an equivalence between augmented $n$-disk algebras.
This completes the first step.

We now prove~(2). 
Let $A_\bullet \ra A$ be an augmented simplicial object witnessing a colimit among augmented $n$-disk algebras.  
Given $N\geq n$, we construct a simplicial augmented $n$-disk algebra $F_\bullet$ whose augmentation ideal is $(-N)$-coconnective, together with a map $F_\bullet \to A_\bullet$ between simplicial objects that induces an equivalence $|F_\bullet|\xra{\simeq}|A_\bullet|$ between their colimits.  
We construct $F_\bullet$ as a sequential colimit $F_\bullet := \varinjlim_\ell F_\bullet^\ell$, each cofactor in which is a simplicial augmented $n$-disk algebra whose augmentation ideal is $(-N)$-coconnective. 
We construct this sequential diagram 
\[
\colim_{\ell\geq 0} \NN_{\leq \ell} \simeq \NN \to \Fun(\bdelta^{\op},\Alg_{n}^{{\sf aug}, \leq -N})_{/A_\bullet}~,\qquad \ell\mapsto F_\bullet^\ell~,
\] 
by induction on $\ell$.  

As the base case we construct $F_\bullet^0$.
Choose a collection 
\[
\Bigl\{\bigl[\Bbbk[k_j] \xra{x_j} \Ker(A\to \uno)\bigr] \mid  j\in J_0\Bigr\}~\subset~  \sH_\ast\bigl(\Ker(A\to \uno)\bigr)
\]
of generators for the homology groups of the augmentation ideal of $A$.
For this collection, the natural morphism $\underset{j\in J_0}\coprod \FF[k_j] \ra A$ between augmented $n$-disk algebras is surjective on homology groups, where $\FF[x_j]$ is the free augmented $n$-disk algebra generated by the free $\sO(n)$-module on $\Bbbk[k_j]$.
For each $j\in J_0$ consider the augmented simplicial object $\FF[x_j]\underset{A}\times A_\bullet \to \FF[x_j]$ among augmented $n$-disk algebras.  Observation~\ref{flim-geom-com} grants that this augmented simplicial object witnesses a colimit.  
As so we can apply~(1) to obtain, for each $j\in J_0$, a simplicial augmented $n$-disk algebra $F_\bullet^{0,j}$ whose augmentation ideal is $(-N)$-coconnective, as it fits into a diagram among augmented $n$-disk algebras
\[
\xymatrix{
F_\bullet^{0,j}   \ar[r] \ar[d] 
&
A_\bullet\ar[d]
\\
\FF[x_j]   \ar[r]
&
A
}
\]
for which the vertical arrows induce equivalences on geometric realizations.
We now set
\[
F_\bullet^0~ :=~ \coprod_{j\in J_0}F_\bullet^{0,j}
\]
to be the coproduct in simplicial objects among augmented $n$-disk algebras.
By construction, there is a canonical morphism $F_\bullet^0\ra A_\bullet$ which induces a surjection on the homology groups of the geometric realizations $|F^0_\bullet| \ra |A_\bullet|\simeq A$.

We now establish the inductive step.
Assume we have constructed the diagram $F_\bullet^0\ra \ldots \ra F_\bullet^\ell$ over $A_\bullet$. 
Consider the pullback simplicial $n$-disk algebra $A_\bullet^\ell := \uno \underset{A_\bullet}\times F_\bullet^\ell$.  
Set $A^\ell:=|A_\bullet^\ell|$ to be the augmented $n$-disk algebra which is the geometric realization of this simplicial one.  
Through Observation~\ref{flim-geom-com} we identify this geometric realization as the pullback $A^\ell \simeq \uno\underset{A}\times |F_\bullet^\ell|$, the augmentation ideal of which is the kernel $\Ker(|F_\bullet^\ell|\to A)$.  
As in the base case, choose a collection 
\[
\Bigl\{\bigl[ \Bbbk[k_j]\xra{x_j} \Ker(A^\ell\to \uno)\bigr]  \mid  j \in J_\ell\Bigr\}~\subset~H_\ast\bigl(\Ker(A^\ell \to \uno)\bigr)
\]
of this augmentation ideal.  
As in the base case, we can choose, for each $j\in J_\ell$, a simplicial object $F_\bullet^{\ell,j}$ whose augmentation ideal is $(-N-1)$-coconnective, as it fits into a diagram among augmented $n$-disk algebras
\[
\xymatrix{
F_\bullet^{\ell,j}   \ar[r] \ar[d] 
&
A_\bullet^\ell\ar[d]
\\
\FF[x_j]   \ar[r]
&
A^\ell
}
\]
for which the vertical arrows induce equivalences on geometric realizations.
We now set $F_\bullet^{\ell+1}$ to be the pushout among simplicial objects in augmented $n$-disk algebras:
\[
\xymatrix{
\underset{j\in J_\ell}\coprod F_\bullet^{\ell,j}   \ar[r] \ar[d]
&
F_\bullet^\ell   \ar[d]
\\
\Bbbk   \ar[r]
&
F_\bullet^{\ell+1}.
}
\]
Lemma~\ref{cocon-cofib} guarantees that this pushout indeed has the property that its augmentation ideal is a simplicial $(-N)$-coconnective $\Bbbk$-module.  
By construction there is a canonical morphism $F_\bullet^{\ell+1}\ra A_\bullet$ between simplicial objects that induces a surjection between homology groups of the geometric realization $|F_\bullet^{\ell+1}|\ra A$.
This concludes the construction of the sequential diagram $\NN \xra{\ell\mapsto F_\bullet^\ell}  \Fun(\bdelta^{\op},\Alg_{n}^{{\sf aug}, \leq -N})_{/A_\bullet}$. 

Set $F_\bullet := \varinjlim_\ell F_\bullet^\ell$ to be the sequential colimit among augmented $n$-disk algebras.
It remains to check that the canonical morphism from the geometric realization $|F_\bullet|\to |A_\bullet|\simeq A$ is an equivalence. 
The morphism $|F_\bullet^0| \to A$ is a surjection on homology groups.
Given the factorization $|F_\bullet^0|\to|F_\bullet|\to A$, we conclude that the morphism $|F_\bullet|\to A$ is a surjection on homology groups.  
We must check that this morphism $|F_\bullet|\to A$ is an injection on homology groups.
Let $\Bbbk[k]\to |F_\bullet|$ represent an element of the kernel of $\sH_\ast(|F_\bullet|) \to \sH_\ast(A)$.  
Because $\Bbbk$, and therefore $\Bbbk[k]$, is compact, there is an $\ell\geq 0$ for which this representative factors as $\Bbbk[k]\to |F_\bullet^\ell| \to |F_\bullet|$.  
Because the map $\Bbbk[k]\to |F_\bullet^\ell| \to A$ factors through $0$, the construction of $F_\bullet^{\ell+1}$ is just so that the composite map $\Bbbk[k]\to |F_\bullet^\ell| \to |F_\bullet^{\ell+1}|$ factors through $0$.
Therefore, the map $\Bbbk[k]\to |F_\bullet|$ represents zero in the homology $\sH_\ast(|F_\bullet|)$.  
We conclude that the kernel of the homomorphism $\sH_\ast(|F_\bullet|) \to \sH_\ast(A)$ is zero, which completes this proof.  
\end{proof}

\begin{cor}\label{cor.important.localization}
The composite functor
\[
(\Alg_n^{\leq -n})^{\bdelta^{\op}}\hookrightarrow (\Alg_n^{\sf aug})^{\bdelta^{\op}}\overset{\colim}\longrightarrow \Alg_n^{\sf aug}
\]
is a localization.
\end{cor}
\begin{proof}
We check the bulleted condition of Lemma~\ref{theformalresult}. 
First, Observation~\ref{flim-geom-com} states that finite limits and geometric realizations commute in the $\infty$-category $\Alg_n^{\sf aug}$.
The second condition of the Lemma~\ref{theformalresult} is Lemma~\ref{secondcondition}.
\end{proof}

The next result completes the proof of Lemma~\ref{free-resolution}.

\begin{cor}\label{cor.atlast}
The inclusion
$
\Alg_n^{\leq -n}\hookrightarrow \Alg_n^{\sf aug}
$
strongly generates.

\end{cor}
\begin{proof}
By Corollary~\ref{cor.important.localization}, the functor
\[
(\Alg_n^{\leq -n})^{\bdelta^{\op}}\hookrightarrow (\Alg_n^{\sf aug})^{\bdelta^{\op}}\overset{\colim}\longrightarrow \Alg_n^{\sf aug}
\]
is a localization. 
Consider Lemma~\ref{select} applied to the case that $\cK = \bdelta^{\op}$, and that $\cC=\Alg_n^{\sf aug}$ with $\cC_0=\Alg_n^{\leq -n}$.
We have shown that the hypotheses of this lemma are satisfied, therefore the inclusion $\Alg_n^{\leq -n}\hookrightarrow \Alg_n^{\sf aug}$ strongly generates.
\end{proof}

\section{Hochschild homology of associative and enveloping algebras}

In the remainder of this paper, we detail the meaning and consequences of our main theorem in the 1-dimensional case, where it becomes a statement about usual Hochschild homology and where the Maurer--Cartan functor $\MC$ reduces to the familiar Maurer--Cartan functors for associative and Lie algebras. We will first describe the general case of an associative algebra, then we will further specialize to the case of enveloping algebras of Lie algebras in characteristic zero.

\subsection{Case $n=1$}

For an augmented associative algebra $A$, consider the Maurer--Cartan functor $\MC_A$ from Definition \ref{def:MC}, i.e., by considering $A$ as a framed 1-disk algebra. Our main theorem has the following consequence in dimension 1. This generalizes an essentially equivalent result for cyclic homology due to Feigin \& Tsygan in \cite{feigintsygan}; see also the operadic generalization of Getzler \& Kapranov in \cite{getzlerkapranov}.

\begin{cor}\label{assoc} Let $A$ be an augmented associative algebra over a field $\Bbbk$. There is an equivalence \[\hh_\ast(A)^\vee   ~\simeq   ~  \hh_\ast(\MC_A)\] between the dual of the Hochschild homology of $A$ and the Hochschild homology of the moduli functor $\MC_A$. If $A$ is either connected and degreewise finite, or $(-1)$-coconnective and finitely presented, then there is an equivalence 
\[
\hh_\ast(A)^\vee~ \simeq~ \hh_\ast(\DD A)
\] 
between the linear dual of the Hochschild homology of $A$ and the Hochschild homology of the Koszul dual of $A$.

\end{cor}
\begin{proof}
The result follows immediately from Theorem \ref{main} together with the equivalence $\int_{S^1}A \simeq \hh_\ast(A)$ (see Theorem 3.19 of \cite{fact}).
\end{proof}

Specializing further to the case where the associative algebra is the enveloping algebra of a Lie algebra.

\begin{cor}\label{Ug}
Let $\frak g$ be Lie algebra over a field $\Bbbk$ which is degreewise finite and connective. There is an equivalence
\[
\hh_\ast(\sU\frak g)^\vee  ~   \simeq   ~ \hh_\ast(\sC^\ast\frak g)
\]
between the dual of the Hochschild homology of the enveloping algebra of $\frak g$ and the Hochschild homology of the Lie algebra cochains of $\frak g$.
\end{cor}
\begin{proof}
We will apply Corollary \ref{assoc} to the case $A=\sU\frak g$. We first make a standard identification of the Koszul dual:
\[
\DD(\sU\frak g):= \Hom_{\Bbbk}\Bigl(\int_{\RR^+}\sU\frak g, \Bbbk\Bigr)\simeq \Hom_{\Bbbk}\Bigl(\Bbbk\underset{\sU\frak g}\ot\Bbbk, \Bbbk\Bigr) \simeq \Hom_{\sU\frak g}(\Bbbk,\Bbbk)=:~\sC^\ast\frak g.
\]
The first and last equivalences are definitional; the second equivalence above is $\ot$-excision for the closed interval; the third equivalence is the Hom-tensor adjunction.

By the Poincar\'e--Birkhoff--Witt filtration, $\sU\frak g \simeq \Sym(\frak g)$, with the given finiteness and connectivity conditions on $\frak g$ imply them for $\sU\frak g$.
Therefore, the conditions of Corollary~\ref{assoc} apply and so give the result.
\end{proof}

\begin{remark}
A different proof of Corollary \ref{assoc} can be given via Morita theory, after results of Lurie, using that Koszul duality interchanges quasi-coherent sheaves and ind-coherent sheaves; a treatment along these lines has been given by Campbell in \cite{campbell}. We briefly summarize: by Theorem 3.5.1 of \cite{dag10}, Koszul duality for associative algebras extends to a Koszul duality for modules and an equivalence $\perf^{\sf R}_A\simeq {\sf Coh}^{\sf L}_{\MC_A} $ between perfect right $A$-modules  and left coherent sheaves on $\MC_A$. Coherent sheaves are, by definition, the dual to perfect sheaves; Hochschild homology is a symmetric monoidal functor and therefore preserves duals. Consequently, the Hochschild homology $\hh_\ast(\perf^{\sf L}_{\MC_A})$ is the dual of $\hh_\ast(\perf^{\sf L}_A)$. Corollary \ref{assoc} then follows from showing formal descent for Hochschild homology of formal moduli problems, implying that the Hochschild homology of $\perf^{\sf L}_{\MC_A}$ agrees with that of $\MC_A$.
\end{remark}

\subsection{Lie algebras}
We will now specialize the preceding results in order to offer an interesting interpretation of our results in Lie theory over a field $\Bbbk$ of characteristic zero.
Our main duality result, Theorem~\ref{main}, specializes to one in Lie theory, Corollary \ref{mchoch}, which makes no reference to either $n$-disk algebras or factorization homology: there is a duality between the Hochschild homology of an enveloping algebra and the Hochschild homology of the associated Maurer--Cartan space.

\begin{convention}[Characteristic zero and framings]
Henceforward, we make the following choices to more conveniently make the connection with Lie algebras.  
We fix a field $\Bbbk$ of characteristic zero.
We work with (augmented) $\cE_n$-algebras in place of (augmented) $\Disk_n$-algebras; and likewise with \emph{framed} zero-pointed $n$-manifolds in place of zero-pointed $n$-manifolds. These easy substitutions do not affect the preceding arguments.  
\end{convention}

We consider the enveloping $n$-disk algebra $\sU_n \frak g$ of a Lie algebra $\frak g$. Maintaining the assumption of $\Bbbk$ being characteristic zero, we will make use of the following sequence of adjunctions
\begin{equation}\label{as1}
\xymatrix{
\Alg_{\Lie} \ar@/^2pc/[rrr]^{\sU_{n-1}}\ar@/^.5pc/[rr]^{ \ \sU_n}&\ar[l]\ldots &\ar[l]\Alg_{\cE_n}^{\sf aug}\ar@/^.5pc/[r]^{\sf Bar}&\ar[l]^\Omega\Alg_{\cE_{n-1}}^{\sf aug}\ar@/^.5pc/[r]^{\sf Bar}&\ar[l]^\Omega\ldots\\}
\end{equation}
in which the curved arrows are left adjoints.
The right adjoints shift the augmentation ideal down a single homological degree.
Each $(\bBar,\Omega)$-adjunction is given by Theorem 5.11 of~\cite{fact}.
The adjunctions involving each $\cE_n$-enveloping algebra construction $\sU_n$ are given by Knudsen in~\cite{lie}.
Also therein, Knudsen establishes a higher Poincar\'e--Birkhoff--Witt filtration, which gives an equivalence between $\Bbbk$-modules:
\begin{equation}\label{as2}
\sU_n\frak g ~\simeq~ \Sym(\frak g[n-1])~.
\end{equation}
A similar picture in the setting of D-modules, including the Poincar\'e--Birkhoff--Witt filtration, is given in \cite{chiral}.

The following result and proof follows that of Corollary \ref{Ug}.

\begin{cor}\label{mchoch} 
Let $\ov{M}$ be a framed compact $n$-manifold with partitioned boundary $\partial \ov{M} = \partial_{\sL} \sqcup \partial_{\sR}$.
Let $\frak g$ be a Lie algebra over a field $\Bbbk$ of characteristic zero.
Suppose $\frak g$ is either
\begin{itemize}
\item$n$-connective and finite in each degree, or 
\item $(-1)$-coconnective and finite as a $\Bbbk$-module.
\end{itemize}
There is a natural equivalence of chain complexes over $\Bbbk$
\[
\Bigl(\int_{\ov{M}\smallsetminus \partial_{\sR}}\sU_n\frak g\Bigr)^\vee ~  \simeq  ~\int_{\ov{M}\smallsetminus \partial_{\sL}} \sC^\ast \frak g
\] 
between the linear dual of the factorization homology with coefficients in the enveloping $n$-disk algebra of $\frak g$ and the factorization homology with coefficients in the Lie algebra cochains of $\frak g$.  
\end{cor}
\begin{proof}
We apply the simple version of Theorem \ref{main}, where the algebra $A$ is either connected or sufficiently coconnected so that the factorization homology on the dual side is in terms of $\DD^nA$ rather than $\MC_A$.
That is, the combination of Theorem \ref{compare-towers}, Theorem \ref{convergence}, and Proposition \ref{switch} (summarized in the introdution as Theorem \ref{summ}), gives an equivalence
\[
\Bigl(\int_{\ov{M}\smallsetminus \partial_{\sR}}A\Bigr)^\vee ~  \simeq  ~\int_{\ov{M}\smallsetminus \partial_{\sL}} \DD^n A
\]
provided that $A$ is either connected or $(-n)$-coconnective.

We now complete the proof in two steps. Firstly, we show the equivalence $\DD^n(\sU_n\frak g) \simeq \sC^\ast \frak g$. This is dual to the equivalence $\int_{(\RR^n)^+}\sU_n\frak g \simeq \sC_\ast\frak g$, which follows from (\ref{as1}) and the identification of $\int_{(\RR^n)^+}$ as the $n$-times iterated bar construction in \S5 of \cite{fact}.

Lastly, we show the given conditions on $\frak g$ give the required conditions on $A=\sU_n \frak g$. This follows from the identification $\sU_n\frak g\simeq \Sym(\frak g[n-1])$ and the fact that, in characteristic zero, symmetric powers preserve coconnectivity.
\end{proof}

\end{document}